\documentclass[a4paper,11pt,reqno]{amsart}

\usepackage{amssymb}

\usepackage[margin=1.15in,top=1.5in,bottom=1.5in]{geometry}

\newtheorem{thm}{Theorem}[section]
\newtheorem{lem}[thm]{Lemma}
\newtheorem{cor}[thm]{Corollary}
\newtheorem{prop}[thm]{Proposition}
\theoremstyle{definition}
\newtheorem{defn}[thm]{Definition}

\newtheorem{rem}[thm]{Remark}

\newcommand{\eps}{\epsilon}
\newcommand{\C}{{\mathcal C}}

\begin{document}

\title{Hypergraph containers}


\author{David Saxton}
\author{Andrew Thomason}

\thanks{The first author was supported by a grant from the EPSRC}

\address{Department of Pure Mathematics and Mathematical Statistics\\
Centre for Mathematical Sciences, Wilberforce Road, Cambridge CB3 0WB, UK}

\email{d.saxton@dpmms.cam.ac.uk}
\email{a.g.thomason@dpmms.cam.ac.uk}

\begin{abstract}
  We develop a notion of containment for independent sets in hypergraphs.
  For every $r$-uniform hypergraph $G$, we find a relatively small
  collection $\C$ of vertex subsets, such that every independent set of $G$
  is contained within a member of $\C$, and no member of $\C$ is large; the
  collection, which is in various respects optimal, reveals an underlying
  structure to the independent sets. The containers offer a straightforward
  and unified approach to many combinatorial questions concerned (usually
  implicitly) with independence.

  With regard to colouring, it follows that simple $r$-uniform hypergraphs
  of average degree~$d$ have list chromatic number at least
  $(1/(r-1)^2+o(1))\log_r d$. For $r=2$ this improves a bound due to Alon
  and is tight. For $r\ge3$, previous bounds were weak but the
  present inequality is close to optimal.

  In the context of extremal graph theory, it follows that, for each
  $\ell$-uniform hypergraph~$H$ of order~$k$, there is a collection $\C$ of
  $\ell$-uniform hypergraphs of order~$n$ each with $o(n^k)$ copies of~$H$,
  such that every $H$-free $\ell$-uniform hypergraph of order~$n$ is a
  subgraph of a hypergraph in~$\C$, and $\log|\C|\le c n^{\ell-1/m(H)}\log
  n$ where $m(H)$ is a standard parameter (there is a similar statement for
  induced subgraphs).  This yields simple proofs, for example, for the
  number of $H$-free hypergraphs, and for the sparsity theorems of
  Conlon-Gowers and Schacht. A slight variant yields a
  counting version of the K\L{R} conjecture.

  Likewise, for systems of linear equations the containers supply, for
  example, bounds on the number of solution-free sets, and the existence of
  solutions in sparse random subsets.

  Balogh, Morris and Samotij have independently obtained related results.
\end{abstract}

\maketitle

\section{Introduction}
A substantial number of theorems in the literature can be phrased naturally
in terms of independent sets in uniform hypergraphs, though this
phraseology is not often used explicitly. An $r$-uniform hypergraph, or
$r$-{\em graph}, $G$ is a pair $(V(G), E(G))$ comprising two sets, the
vertices $V(G)$ and edges $E(G)$ of~$G$, where each edge $e\in E(G)$ is a
set of $r$~elements of~$V(G)$. Hence a 2-graph is an ordinary graph. A set
$I\subset V(G)$ is {\em independent} if there is no edge $e\in E(G)$ with
$e\subset I$.

There are many questions that, on the face of it, have little to do with
hypergraphs, but which can be formulated naturally in terms of the number
of independent sets in some hypergraph or class of hypergraphs (examples
will be given later). Nevertheless, the
question {\it per se} of how many independent sets there can be in a graph
has attracted attention only relatively recently. The maximum number of
independent sets in a graph of given average degree can be determined
easily via the Kruskal-Katona theorem~\cite{KK1,KK2}, but for regular
graphs the maximum is harder to find: following a good estimate by
Alon~\cite{A3}, the exact value for bipartite graphs was determined by
Kahn~\cite{K} via an elegant entropy argument, and his result was extended
to all graphs by Zhao~\cite{Z}. There are at most
$(2^{d+1}-1)^{n/2d}=2^{n/2+O(n/d)}$ independent sets in a $d$-regular graph
of order$~n$ (that is, having $n$ vertices), and this number is attained by
$n/2d$ disjoint copies of $K_{d,d}$.

It would be convenient for many purposes if there were at most $2^{o(n)}$
independent sets in an $r$-graph $G$ of order~$n$ and average degree~$d$,
but examples like that just cited show this hope to be a forlorn one.
Nevertheless, for the applications we have in mind, it is enough to find a
good collection $\C$ of {\em containers} for independent sets: this is a
family of subsets of $V(G)$ such that, for each independent set $I$, there
is a set $C\in\C$ with $I\subset C$, and $|\C|=2^{o(n)}$. Of course, we
could just take $\C=\{V(G)\}$, but this collection would not be helpful:
for $\C$ to be of use, a further condition is needed that each container
$C\in\C$ is not large, in a sense made precise later
(see~\S\ref{subsec:mu}).

Another immediate candidate for $\C$ is the collection of {\em maximal}
independent sets, but this too can be large; for example, if $d$ is even,
adding a $1$-factor into the vertex classes of each $K_{d,d}$ of the graph
$(n/2d)K_{d,d}$ produces a $(d+1)$-regular graph with at least $2^{n/4}$
maximal independent sets. (The maximum number of maximal independent sets
in any graph of order~$n$ was determined by Moon and Moser~\cite{MM}.)

The main purpose of this paper is to show that every $r$-graph~$G$ of
average degree~$d$ and order~$n$ does have a small collection $\C$ of
containers.  Typically, but not always, $|\C|\le2^{c_dn}$ where $c_d$ is
approximately $d^{-1/(r-1)}$. Results of this kind were known previously in
special cases.  Sapozhenko~\cite{Sap1,Sap2,Sap5,Sap4,Sap3} treated regular
2-graphs. Containers for $r$-graphs were introduced and used in~\cite{ST}
for the restricted instance of simple regular $r$-graphs (a hypergraph is
{\em simple} or {\em linear} if every pair of vertices lies in at most one
edge). However, the most interesting applications require containers for
non-regular $r$-graphs.  Finding such containers presents significant
difficulties and the method here is unrelated to that of~\cite{ST}.
(Nevertheless, the method of~\cite{ST} is good enough to give easy proofs
of some of the results here --- see~\cite{STe}.)

We describe our main results about containers in~\S\ref{sec:containers}.
The fundamental result is Theorem~\ref{thm:cover} stated
in~\S\ref{subsec:mainthm}.  It is worth mentioning that the statement
applies to all $r$-graphs~$G$ but it gives useful information only if $d$
is large (though independently of~$n$). In order to state the main theorem
we need to introduce and motivate a couple of concepts (degree measure and
the co-degree function), but their definitions are quite straightforward.
This discussion all takes place in~\S\ref{sec:containers}. The main result
is, in some senses, optimal, as we shall explain.

As well as the main theorem, \S\ref{sec:containers} includes two
consequences of it, packaged for ready use in two different kinds of
applications.  These two varieties are worth emphasising, because they
highlight two ways in which we might require a container $C$ to be ``not
large'': in one version $e(G[C])$ is small, which is to say that
the container has only a few edges inside it, and in the other version
$|C|$ is small, meaning that the container does not have many
vertices. These two situations are quite different in the way they are
handled, though both are derived from the same main theorem.

The actual construction of the  containers is given
in~\S\ref{sec:online}. The construction is via an algorithm, just a few
lines long. This algorithm is needed only for the proof
of Theorem~\ref{thm:cover} and no understanding of it is required in order
to apply the theorem; nevertheless the algorithm clearly lies at the heart
of the whole process, and so~\S\ref{sec:online} includes some discussion with
the aim of illuminating what is going on.

In~\S\ref{sec:calc} we prove Theorem~\ref{thm:cover}; this comes down to
making some calculations that verify the performance of the algorithm. The
calculations are mostly straightforward, though at one point we used a
slightly more detailed argument than is necessary, in order to achieve
better constants.

Having proved the main result, we proceed in~\S\ref{sec:iteration}
and~\S\ref{sec:uniform} to derive the two consequences mentioned previously
(we include as well a more technical version of one of them, useful in more
sensitive applications).
The optimality of the main theorem, or at least one aspect of
it, is proved in~\S\ref{sec:optimality}, but a potentially better approach
to the algorithm is mentioned in~\S\ref{sec:postscript}.

Before getting down to the details of the container theorem,
in~\S\ref{sec:applics} we offer some motivation by outlining a few
applications. The details of these are given
in~\S\ref{sec:list}--\S\ref{sec:sparse}.

\subsection{A little notation}

We use standard notation. In particular, for $m,n\in{\mathbb N}$ we let
$[n]=\{1,\ldots,n\}$ and $[m,n]=\{m,\ldots,n\}$. For collections of subsets
we write, for example, $[m,n]^{(s)}=\{\sigma\subset[m,n]:|\sigma|=s\}$,
$[m,n]^{(>s)}=\{\sigma\subset[m,n]:|\sigma|>s\}$, and so on. As usual,
$\mathcal{P}(S)$ denotes the collection of all subsets of~S; we omit
parentheses where no confusion can arise, for instance writing
$\mathcal{P}[n]$ instead of $\mathcal{P}([n])$. If $G$ is a hypergraph we
write $e(G)=|E(G)|$ for the number of edges of~$G$ and $v(G)=|V(G)|$ for
the number of vertices of $G$. If $S\subset V(G)$ then $G[S]$ denotes the
subhypergraph of $G$ induced by~$S$, that is, $G[S]=(S,E(G)\cap\mathcal{P}
(S))$.

\section{Some applications of containers}\label{sec:applics}

The purpose of this section is to highlight some results that follow from
the existence of containers, in the hope of motivating the main result
itself, Theorem~\ref{thm:cover}. The applications involve list colouring,
extremal graph theory, and solutions of linear equations. 

The applications are of two essentially different kinds, namely those in
which we require $|C|$ to be bounded for each $C\in\C$, and those where we
require $e(C)$ to be bounded. In fact we give only one application where
$|C|$ is bounded, namely the one about list colouring: hence
Theorem~\ref{thm:uniform} (the version of the container theorem packaged
for bounds on $|C|$) is used only in this application.

The remaining applications require $e(C)$ to be bounded. On the face of it
they appear more
numerous but this appearance is deceptive: for example, all the results
concerning $H$-free graphs, including those involving sparse random graphs,
are actually direct corollaries of a single theorem about the class of
$H$-free graphs, namely Theorem~\ref{thm:ffree_cover}, and this theorem is
the only place in the argument where the container theorem is
invoked. Moreover, it is applied to just one hypergraph (more exactly, to
one hypergraph $G(N,H)$ for each $H$ and~$N$).

Likewise, the applications to solutions of linear equations involve
translating some given problem into a question about the independent sets
in a specific hypergraph~$G$, then finding containers for this~$G$, and then
interpreting these containers back in the original context.

The list colouring application is thus rather different to the others but
it is the one which originally motivated us (our early thoughts appeared
in~\cite{ST}), and it is the application which has shaped the algorithm
that we use to construct containers.

The technical details of the list colouring and extremal graph theory
applications are supplied later in~\S\ref{sec:list}--\ref{sec:sparse}. As
for the arithmetical applications, we state them here in order to
illustrate the use of the container theorem, but we give the details
elsewhere~\cite{STa}, so as to maintain the focus here on the container
theorem itself.

\subsection{List colourings}\label{subsec:list}

A $2$-graph $G$ is said to be $k$-{\em choosable} if, whenever for each
vertex $v\in V(G)$ we assign a list $L_v$ of $k$ colours to~$v$, then it is
possible to choose a colour for $v$ from the list~$L_v$, so that no two
adjacent vertices receive the same colour. The {\em list chromatic number}
$\chi_l(G)$ (also called the {\em choice number}) is the smallest $k$
such that $G$ is $k$-choosable. If all the lists are the same then a list
colouring is just an ordinary $k$-colouring and so $\chi_l(G)$ is at
least $\chi(G)$, the ordinary chromatic number of~$G$. This natural
definition was first studied by Vizing~\cite{V} and by Erd\H{o}s, Rubin and
Taylor~\cite{ERT}. One of the main discoveries of~\cite{ERT} is that
$\chi_l(G)$ can be much larger than $\chi(G)$, because
$\chi_l(K_{d,d}) = (1+o(1))\log_2 d$, whereas $\chi(K_{d,d}) = 2$.

In fact, unlike $\chi(G)$, $\chi_l(G)$ must grow with the minimum degree
of the graph~$G$. Alon~\cite{A1,A2} showed that $\chi_l(G)\ge
(1/2+o(1))\log_2 d$ holds for any graph $G$ of minimum degree~$d$.

The notion of $k$-choosability carries over directly to $r$-graphs, when it
is understood that the vertex colours are chosen so that no edge has all
its vertices the same colour. There is a straightforward reason, as pointed
out by Alon and Kostochka~\cite{AK1} (see too Haxell and Pei~\cite{HP}),
why for $r\ge3$ it is not true for
$r$-graphs $G$ that $\chi_\ell(G)$ grows with the average degree. Let $F$
be some graph on $n$ vertices, say $F=(n/2)K_2$, and let $G$ be some
$r$-graph each of whose edges contains an edge of~$F$. Then
$\chi_l(G)\le\chi_l(F)$, so in this example $\chi_l(G)=2$, whereas the
average degree of $G$ can be large. However, if we restrict to simple
$r$-graphs the situation is different. Haxell and Pei~\cite{HP} showed
$\chi_\ell(G)=\Omega(\log d/\log \log d)$ if $G$ is a Steiner triple
system, and Haxell and Verstra\"ete~\cite{HV} proved that
$\chi_l(G)\ge(1+o(1))\left(\log d/5\log\log d\right)^{1/2}$ for all simple
$d$-regular $3$-graphs~$G$. Alon and Kostochka~\cite{AK1} showed
$\chi_l(G)\ge(\log d)^{1/(r-1)}$ for simple $r$-graphs~$G$ of average
degree~$d$, and in~\cite{ST} it was shown that $\chi_l(G)=\Omega(\log d)$
for simple $d$-regular $r$-graphs. We extend this to all simple $r$-graphs,
at the same time giving a better constant.

\begin{thm}\label{thm:chil}
  Let $r\in\mathbb N$ be fixed. Let $G$ be a simple $r$-graph with
  average degree~$d$. Then, as $d\to\infty$,
$$
\chi_l(G)\,\ge\,\,(1+o(1))\,\frac{1}{(r-1)^2}
\log_rd
$$
holds. Moreover, if $G$ is regular then
$$
\chi_l(G)\,\ge\,\,(1+o(1))\,\frac{1}{r-1}\log_rd\,.
$$
\end{thm}

Note that, for $r=2$, this improves Alon's bound~\cite{A2} by a factor of~2
and is best possible. We think that the bound given for regular
$r$-graphs might hold for general $r$-graphs and moreover that it too might
be best possible (see~\S\ref{sec:list}).
For colourings of non-simple $r$-graphs, see~\S\ref{sec:postscript}.

\subsection{$H$-free graphs}\label{subsec:hfree}

An $\ell$-graph on vertex set $[N]$ is said to be $H$-{\em free} if it
contains no subgraph isomorphic to the $\ell$-graph~$H$.

As far as $H$-free graphs are concerned, our main result is this: for any
given $\ell$-graph $H$, though there are many $H$-free $\ell$-graphs, each
of these is contained in one of a very small collection of $\ell$-graphs
that are almost $H$-free.  More exactly, there is a small collection $\C$
of $\ell$-graphs, each $H$-free $\ell$-graph being a subgraph of an
$\ell$-graph in $\C$, and no $\ell$-graph in $\C$ having more than
$o(N^{v(H)})$ copies of~$H$. The main content of the theorem is that the
size of $\C$ is very small.  For graphs at least, Szemer\'edi's regularity
lemma gives a collection with $\log |\mathcal{C}| = o(N^2)$, but the size
of $\C$ in our theorem is much smaller. It is expressed in terms of a
parameter $m(H)$ that appears often in the literature.

\begin{defn}\label{def:mH}
 For an $\ell$-graph $H$ with $e(H)\ge2$, let
\[
 m(H) = \max_{H' \subset H,\,e(H') > 1} \frac{e(H')-1}{v(H')-\ell}.
\]
\end{defn}

Sometimes, $H$ is called (strictly) balanced if the maximum is attained
(uniquely) when $H^\prime=H$. However, this restriction is not needed in any
of our arguments and it is ignored.

We shall indicate shortly why the parameter $m(H)$ might be expected to
make an appearance here, but first we state our main theorem for $H$-free
$\ell$-graphs. As usual, let ${\rm ex}(N,H)$ be the maximum number of edges
in an $H$-free graph of order $N$ and let $\pi(H)=\lim_{N\to\infty}{\rm
  ex}(N,H)\binom{N}{\ell}^{-1}$. The symbol $\subset$ in the theorem
means ``is a subgraph of''.

\begin{thm}\label{thm:ffree_cover}
Let $H$ be an $\ell$-graph with $e(H)\ge2$ and let $\epsilon>0$.
For some $c>0$ and for every $N \ge c$, there exists a collection
$\C$ of $\ell$-graphs on vertex set~$[N]$ such that
\begin{itemize}
 \item[(a)]
   for every $H$-free $\ell$-graph $I$ on vertex set $[N]$, there exists
   $C\in\C$ with $I\subset C$,
 \item[(b)]
   for every $\ell$-graph $C \in \mathcal{C}$,
   the number of copies of $H$ in $C$ is at most $\epsilon N^{ v(H)}$,
   and $e(C) \le (\pi(H)+\epsilon) {N \choose \ell}$,
 \item[(c)]
   $\log |\mathcal{C}| \leq c N^{\ell-1/m(H)} \log N$,
 \item[(d)]
   moreover, for every $I$ in (a), there exists
   $T=(T_1,\ldots,T_s)$ where $T_i \subset I$, $s\le c$ and
   $\sum_ie(T_i) \le c N^{\ell-1/m(H)}$, such that $C=C(T)$.
\end{itemize}
\end{thm}

The meanings of~(a), (b) and (c) should be clear enough. Condition~(a) is the
basic property of the collection~$\C$, namely that all $H$-free graphs are
subgraphs of members of~$\C$. Condition~(b) is what is meant by the
containers themselves being small, which is that each $C$ contains few
copies of $H$. This immediately implies the bound on $e(C)$, via the
supersaturation theorem of Erd\H{o}s and
Simonovits~\cite{ST}. Condition~(c) says that the collection $\C$ is small.

Condition~(d) should be understood in the following way. The notation
$C=C(T)$ is used to mean that $C$ is a function of, or is determined
by,~$T$. Now $T$ is a collection of subgraphs of $I$ that are small, that
is, have few edges. The point of condition~(d) is that $\C$ is therefore
small, since $|\C|$ is at most the number of possible objects~$T$, which is
a small number because the subgraphs comprising $T$ are small. The bound on
$|\C|$ that~(d) directly implies is the one given in condition~(c), and for
most purposes we could dispense with~(d) because the bound in~(c) is good
enough. However, condition~(d) gives slightly more information, namely that
the graphs $T_i$ comprising $T$ are actually subgraphs of $I$ and not just
arbitrary graphs. This extra information can be just enough, in tight
corners (specifically, in Lemma~\ref{lem:sparse}), to give a better result
than what can be obtained by a direct use of~(c) (effectively it removes
the $\log N$), and we retain~(d) for this reason.

The existence of the collection~$\C$ follows straightforwardly from the
results in~\S\ref{sec:containers}, as shown in~\S\ref{sec:hfree}, by
applying them to the $e(H)$-graph $G=G(N,H)$, whose $n=\binom{N}{\ell}$
vertices are the $\ell$-sets in~$[N]$, and whose edges are subsets of
$V(G)$ spanning a copy of~$H$ in~$[N]$. The subsets of $V(G)$ are then
$\ell$-graphs with vertex set~$[N]$, and independent sets in $G$ correspond
to $H$-free $\ell$-graphs. The $\ell$-graphs in $\C$ are simply the
containers for the independent sets of $G$ supplied by our main container
theorem (more precisely, Corollary~\ref{cor:sparse_container}). In order to
apply the container theorem and so obtain Theorem~\ref{thm:ffree_cover},
all that is required is to calculate a simple parameter
of~$G(N,H)$. Details are in~\S\ref{sec:hfree}.

We can now indicate why the parameter $m(H)$ shows up in
Theorem~\ref{thm:ffree_cover}. It is well known, and not hard to check,
that if $t=o(N^{\ell-1/m(H)})$ then for some $H' \subset H$ almost all
$\ell$-graphs on $N$ 
vertices with $t$ edges contain many fewer copies of $H'$ than they do
edges. Thus most subsets of $V(G(N,H))$ of size $t$ are independent, or
close to it (that is, contain many fewer edges than vertices). For reasons
discussed in~\S\ref{subsec:optimality}, this means that $e^{\Omega(t)}$
containers are needed, and from this standpoint,
Theorem~\ref{thm:ffree_cover} is more or less best possible. Perhaps a more
convincing demonstration of optimality is that an improvement in the bound
on $|\C|$ in Theorem~\ref{thm:ffree_cover} would directly improve, say, the
bound on $p$ in Theorem~\ref{thm:ffree_sparse}, but the bound there is
well known (and readily checked) to be optimal.

We remark that Theorem~\ref{thm:ffree_cover} can be extended so that the
$\ell$-graphs $I$ need not be independent: they need only have few copies
of~$H$. The extension is given in Theorem~\ref{thm:ffree_cover_heavy}
but, again, we defer the technicalities to~\S\ref{sec:hfree}.

The remainder of this section includes further results on $H$-free
$\ell$-graphs; as pointed out earlier, they are all (except
Corollary~\ref{cor:C4}) just consequences of
Theorem~\ref{thm:ffree_cover}.

\subsection{The number of $H$-free graphs}

How many $H$-free $\ell$-graphs are there altogether on vertex set~$[N]$?
Choosing any maximum $H$-free graph and taking all its subgraphs supplies
at least $2^{(\pi(H)+o(1)) {N \choose \ell}}$ $H$-free graphs. But each
$H$-free graph is a subgraph of a member of the collection~$\C$ given by
Theorem~\ref{thm:ffree_cover}, so the total number of $H$-free graphs is at
most $|\C|2^{\,\max_{C\in\C} e(C)}$. Now, $\max_{C\in\C}e(C) \le
(\pi(H)+o(1)) {N \choose \ell}$, and Theorem~\ref{thm:ffree_cover} shows
that $|\C|=2^{o(N^\ell)}$, giving the following immediate consequence.

\begin{cor}\label{cor:ffree_count}
  Let $H$ be an $\ell$-graph. The number of $H$-free
  $\ell$-graphs on vertex set $[N]$ is $2^{(\pi(H)+o(1)) {N \choose \ell}}$.
\end{cor}

In the case $\ell=2$, this was proved for complete~$H$ by Erd\H{o}s,
Kleitman and Rothschild~\cite{EKR} and for general~$H$ by Erd\H{o}s, Frankl
and R\"odl~\cite{EFR}. Nagle, R\"odl and Schacht~\cite{NRS} proved it for
general~$\ell$ using hypergraph regularity methods.

For $\ell$-graphs $H$ which satisfy ${\rm ex}(N,H)=o(N^\ell)$ (when
$\ell=2$ this means $H$ is bipartite), we have $\pi(H)=0$, and
Corollary~\ref{cor:ffree_count} is unhelpful. Nevertheless our results can
still be useful, provided appropriate information about $G(N,H)$ is
available. The simplest case is $\ell=2$ and $H=K_{2,2}=C_4$, where it is
well known that ${\rm ex}(N,C_4)= (1/2+o(1))N^{3/2}$ (Erd\H{o}s, R\'enyi
and S\'os~\cite{ERS}), implying the trivial upper bound $2^{O(N^{3/2}\log
  N)}$ for the number of $C_4$-free graphs. 
Theorem~\ref{thm:iteration} describes what happens if we apply the main
container theorem repeatedly to a hypergraph: by applying this theorem to
$G(N,C_4)$, the following can be obtained.

\begin{cor}\label{cor:C4}
The number of $C_4$-free graphs on vertex set $[N]$ is at most
$2^{(300+o(1))N^{3/2}}$.
\end{cor}

We shall not prove Corollary~\ref{cor:C4}; we state it just as an
illustration of what can be derived by plugging numbers into a generic
container theorem. The argument is very similar to that for the upper bound
in Theorem~\ref{thm:sidon} on the number of Sidon sets, which also relies on
Theorem~\ref{thm:iteration}, and whose details can be found in~\cite{STa}.
The reason we do not give details for Corollary~\ref{cor:C4} is that
Kleitman and Winston~\cite{KW} obtained a finer bound, namely
$2^{(1.082+o(1))N^{3/2}}$. The number of $K_{s,t}$-free graphs
has been well estimated by Balogh and Samotij~\cite{BS}.
Recently, Morris and the first author~\cite{MS}, using container methods
and other techniques, have shown that the number of $C_{2k}$-free graphs
on vertex set $[N]$ is at most $2^{O(N^{1+1/k})}$, where $C_{2k}$ is
the cycle of length~$2k$. The order of the extremal
function ${\rm ex}(N,C_{2k})$ is unknown in general, though Bondy and
Simonovits~\cite{BoSi} proved ${\rm ex}(N,C_{2k})=
O(N^{1+1/k})$. Nevertheless, it is
further shown in~\cite{MS} that, for some $c>0$, there are more than
$2^{(1+c){\rm ex}(N,C_6)}$ $C_6$-free graphs of order~$N$ for
infinitely many~$N$.

\subsection{Induced-$H$-free graphs}\label{subsec:ihfree}

Alongside the many results about $H$-free graphs, there is a corresponding
corpus about {\em induced} $H$-free graphs, that is, graphs with no induced
subgraph isomorphic to~$H$. The number of induced $H$-free graphs was closely
estimated by Pr\"omel and Steger~\cite{PS}, and there have been many
subsequent refinements.

If $I$ is an induced $H$-free $\ell$-graph, we need to ask what kind of
object $C$ must be in order that the inclusion $I\subset C$ is helpful; if,
as in Theorem~\ref{thm:ffree_cover}, $C$ itself is just an $\ell$-graph and
$I\subset C$ means $I$ is a subgraph of~$C$, then the induced subgraphs of
$I$ differ from those of $C$, which is no use. We borrow the notion of {\em
  2-coloured multigraph} from~\cite{MT,T}. A 2-coloured
$\ell$-multigraph~$C$ on vertex set~$[N]$ is a pair of edge sets $C_R, C_B
\subset [N]^{(\ell)}$, which we call the red and the blue edge sets. Let
$I$ be an $\ell$-graph on $[N]$.  Then we write $I\subset C$ if
$E(I)\subset C_R$ and $[N]^{(\ell)}\setminus E(I)\subset C_B$. Thus edges
in $C_R\cap C_B$ always help towards the inclusion $I\subset C$. If we
construct (see~\S\ref{sec:hfree}) a hypergraph akin to $G(N,H)$ but which
encodes both red and blue edges, and apply the container theorem to it, we
obtain the following analogue of Theorem~\ref{thm:ffree_cover}.

\begin{thm}\label{thm:ind_ffree_cover}
Let $H$ be an $\ell$-graph and let $\epsilon>0$. For some $c>0$ and for $N$
sufficiently large, there exists a collection $\C$ of 2-coloured
$\ell$-multigraphs on vertex set~$[N]$ such that
\begin{itemize}
 \item[(a)]
   for every $\ell$-graph $I$ on vertex set $[N]$ with no induced copy of $H$
   there exists $C\in\C$ with $I\subset C$,
 \item[(b)]
   for every $C \in \mathcal{C}$, the number of copies of $H$ in $C$ is
   at most $\epsilon N^{ v(H)}$, 
 \item[(c)]
   $\log |\mathcal{C}| \leq c N^{\ell-(v(H)-\ell)/\left({v(H)\choose
         \ell}-1\right)} \log N$. 
\end{itemize}
\end{thm}

Note that $(v(H)-\ell)/({v(H)\choose \ell}-1) = 1/m(K)$ where $K$ is the
complete $\ell$-graph of order~$v(H)$. The form of this theorem is, to an
extent, reminiscent of Theorem~\ref{thm:ffree_cover}, and it arises from
the method of proof in which an induced copy of $H$ is modelled as a
red-blue colouring of~$K$. However the value $m(K)$ is not invariably
optimal: for
example, when $\ell=2$, $v(H)=4$ and $H$ is an induced path of length
three, then the number of induced $H$-free graphs (sometimes known as {\em
  cographs}) is only $N^{O(N)}$, so these graphs themselves comprise a
smaller collection of containers than that offered by the theorem.

Theorem~\ref{thm:ind_ffree_cover} can be used to recover basic results,
akin to Corollary~\ref{cor:ffree_count}, about the number of induced
$H$-free $\ell$-graphs. In fact we can state a probabilistic version just
as readily.  Let $G^{(\ell)}(N,p)$ be a random $\ell$-graph obtained by
choosing edges independently from the complete $\ell$-graph $K^{(\ell)}_N$
with probability~$p$.  Our result is stated in terms of a
function~$h_p(H)$, defined as follows (and discussed further, in slightly
different terminology, in~\cite{MT2}).  For a 2-coloured $\ell$-multigraph
$J$, with vertex set $[N]$ and having red and blue edge sets $J_R$ and
$J_B$, let
\[
 H_p(J) = -|J_R\setminus J_B| \log_2 p - |J_B\setminus J_R| \log_2 (1-p).
\]
The point of this definition is that, if $J_R \cup J_B = [N]^{(\ell)}$, then
the probability that $G^{(\ell)}(N,p)$ is a subgraph of $J$ is
$2^{-H_p(J)}$.  Let
\[
 \mbox{hex}_p(H,N) = \min\{\, H_p(J) : J_R \cup J_B = [N]^{(\ell)}, H
 \not\subset J\, \}\,.
\]
Then we put $h_p(H) = \lim_{N\to\infty} \mbox{hex}_p(H,N) {N \choose
  \ell}^{-1}$ (this limit exists, since an averaging argument shows
$(N-\ell)\mbox{hex}_p(H,N)\ge N\mbox{hex}_p(H,N-1)$).

\begin{thm}\label{thm:ind_ffree_measure}
Let $0<p<1$ be constant and let $H$ be an $\ell$-graph. Then
\[
\mathbb{P}(G^{(\ell)}(N,p) \mbox{ is induced-$H$-free}) =
2^{-(h_p(H)+o(1)){N\choose \ell}}.
\]
\end{thm}

For graphs, that is, $\ell=2$, this theorem was proved for $p=1/2$ by Pr\"omel
and Steger~\cite[Theorem~1.3]{PS} and for general $p$ by Bollob\'as and
Thomason~\cite[Theorem~1.1]{BT} (clarified by Marchant and
Thomason~\cite{MT2}). For $p=1/2$ and $\ell=3$ it was proved  by 
Kohayakawa, Nagle and R\"odl~\cite{KNR} using hypergraph regularity
techniques, 
Dotson and Nagle~\cite{DN} extending this to general~$\ell$.

It can be imagined that arguments similar to those described here
could be used to obtain container results about other structures, such as
tournaments.

\subsection{Linear equations}\label{subsec:sumfree}

Let $F$ be either a finite field or the set of integers
$[N]$.  We consider linear systems of equations $A x = b$, where $A$ is a
$k \times r$ matrix with entries in $F$, $x \in F^r$ and $b \in F^k$.  We
call such a triple $(F,A,b)$ a \emph{$k\times r$ linear system}.

\begin{defn}\label{defn:exteqnoX}
For a $k\times r$ linear system $(F,A,b)$, a subset $I \subset F$ is
\emph{solution-free} if there is no $x \in I^r$ with $Ax=b$,
and $\mbox{ex}(F,A,b)$ is the maximum size of a solution-free subset.
\end{defn}

The notion of a solution-free subset is analogous to that of an $H$-free
hypergraph in~\S\ref{subsec:hfree}. Once again, our contribution to this
topic is a container theorem for solution-free sets. It is obtained by
constructing a hypergraph $G$ whose independent sets correspond to
solution-free sets, after which a simple check of some parameter of $G$
allows the container theorem to be applied. A precise statement, however,
requires one or two technical definitions, so we omit it from here and
refer the reader to~\cite{STa}.

Nevertheless we mention a consequence for counting solution-free subsets.
For an equation $Ax=b$, how many solution-free subsets of $F$ are there?  A
well-known instance of this question is to find the number of subsets
$S\subset[N]$ containing no solution to $x+y=z$; the asymptotic answer,
conjectured by Cameron and Erd\H{o}s~\cite{CE}, was given by Green~\cite{G0}
and by Sapozhenko~\cite{Sap3}. 

For a general system, every subset of a
solution-free set is itself solution-free, so there are at least
$2^{\rm{ex}(F,A,b)}$ solution-free sets. For a single equation (the case
$k=1$), it was shown by Green~\cite{G} that there are at most
$2^{\rm{ex}(F,A,b)+o(|F|)}$ solution-free subsets; Sapozhenko
too~\cite{Sap5} has results of this kind.

The same bound does not always hold for $k\ge2$. If some variables are
closely tied to other variables --- say the equations imply that
$x=y$ --- then there can be significantly more than $2^{\rm{ex}(F,A,b)}$
solution-free sets. However, a (perhaps non-standard but) natural
condition on~$A$ rules out closely tied variables, and in this case the stated
bound holds good.

\begin{defn}
 We say that $A$ has \emph{full rank} if given any $b \in
F^k$ there exists $x \in F^r$ with $Ax=b$. We then say that $A$ is
\emph{abundant} if it has full rank and every $k \times (r-2)$ submatrix
obtained by removing a pair of columns from $A$ still has full rank.
\end{defn}

\begin{thm}\label{thm:eqn_countnoX}
There is a function $f:\mathbb{N}\to\mathbb{R}$ with $f(n)=o(n)$ such that
if $F$ is a finite field and $(F,A,b)$ is a $k \times r$ linear system with
$A$ abundant, then the number of solution-free subsets of $F$ is at most
$2^{\rm{ex}(F,A,b)+f(|F|)}$.

Likewise, for each fixed integer matrix $A$, there is a function
$g:\mathbb{N}\to\mathbb{R}$ with $g(n)=o(n)$ such that if $([N],A,b)$ is a
$k \times r$ linear system with $A$ abundant, then the number of
solution-free subsets of $[N]$ is at most $2^{\rm{ex}([N],A,b)+g(N)}$.
\end{thm}

For example, take $A = (1, 1, -1)$ and $b = (0)$.
Theorem~\ref{thm:eqn_countnoX} says that the number of sum-free subsets of
$[N]$ is $2^{N/2+o(N)}$, giving a new proof of the weak form of the
Cameron-Erd\H{o}s conjecture, proved independently by Alon~\cite{A3}, by
Calkin~\cite{C} and by Erd\H{o}s and Granville
(unpublished). Interestingly, whilst our container method for 2-graphs is
closely related to arguments of Sapozhenko in~\cite{Sap3}, our derivation
of the weak Cameron-Erd\H{o}s conjecture is via 3-uniform hypergraphs and
differs from that in~\cite{Sap3}.

Similar results hold when $F$ is an abelian group.  For the proof of
Theorem~\ref{thm:eqn_countnoX} we need the fact that a subset containing
few solutions is close in size to a solution-free subset. 
There appears to be no analogue to the simple supersaturation results that
helped us at similar points in~\S\ref{subsec:hfree}
and~\S\ref{subsec:ihfree}, so here we invoke the various removal lemmas
of Shapira~\cite{Sh} and of Kr\'al', Serra and Vena~\cite{KSV1,KSV2},
extending Green's original lemma~\cite{G}.

For linear systems where $\mbox{ex}(F,A,b)=o(|F|)$,
Theorem~\ref{thm:eqn_countnoX} is uninformative.
One of the most prominent examples is that
of Sidon sets. A set $A \subset [n]$ is \emph{Sidon} if every sum of two
elements is distinct, i.e., there are no solutions to $w+x=y+z$ with
$\{w,x\}\ne\{y,z\}$.
It is easy to see that a Sidon set has size at most $\lceil\sqrt{2n}\rceil$,
since each of the $|S|(|S|-1)/2$ values $x-y$, where $x,y \in S$ and $y<x$,
are distinct and lie in $\{1,\ldots,n-1\}$. Erd\H{o}s and Tur\'an~\cite{ET}
improved this upper bound to $|S|\le (1+o(1))\sqrt{n}$,
and there are examples achieving this bound.

It is natural to ask, as Cameron and Erd\H{o}s did~\cite{CE}, how many
Sidon sets there are, and the answer clearly lies between
$2^{(1+o(1))\sqrt{n}}$ and $2^{O(\sqrt n\log n)}$. Neither of these bounds,
it turns out, is tight.

\begin{thm}\label{thm:sidon}
There are between $2^{(1.16+o(1))\sqrt{n}}$ and $2^{(55+o(1))\sqrt{n}}$
Sidon subsets of~$[n]$.
\end{thm}

The lower bound gives a negative answer to the open question of whether
there are only $2^{(1+o(1))\sqrt{n}}$ Sidon sets. The upper bound follows
from an application of Theorem~\ref{thm:iteration}, similar to that in the
proof of Corollary~\ref{cor:C4}. Kohayakawa, Lee, R\"odl and
Samotij~\cite{KLRS} have obtained an upper bound of the same kind and with
a better constant. For details see~\cite{STa}.

\subsection{Sparsity}\label{subsec:sparse}

In recent times, there has been interest in the extent to which theorems
holding for dense structures hold also for sparse random
substructures. Our results can be applied in this context, and we give some
illustrative examples involving the notions of $H$-free graphs and
solution-free subsets already discussed.

The application of our results always fits a simple paradigm. Typically we want
some statement to hold for a random substructure, with high probability; by
considering an appropriate collection of containers, the fact that there are
a small number of containers means that the work is reduced, via the union
bound, to establishing a (generally much simpler) statement for a single
container.

For example, consider a random $\ell$-graph $G^{(\ell)}(N,p)$, as defined
in~\S\ref{subsec:ihfree}.
Evidently there are $H$-free subgraphs of $G^{(\ell)}(N,p)$ with $p\,{\rm
  ex}(N,H)$ edges, but are there significantly larger $H$-free subgraphs?
It was conjectured by Haxell, Kohayakawa and \L{u}czak~\cite{HKL1,HKL2},
and by Kohayakawa, \L{u}czak and R\"odl~\cite{KLR}, that if
$pN^{1/m(H)}\to\infty$ then $H$-free subgraphs of $G^{(\ell)}(N,p)$ almost surely
have at most $(1+o(1))p\,{\rm ex}(N,H)$ edges. This conjecture was recently
proved by Conlon and Gowers~\cite{CG} (for strictly balanced~$H$) and by
Schacht~\cite{Sch}, using different methods.  Our methods give an alternative
proof. For each container $C\in\C$ given by Theorem~\ref{thm:ffree_cover},
it is easily seen that, with high probability, $G^{(\ell)}(N,p)$ contains not
much more than $p e(C)\le (\pi(H) +o(1))p{N \choose \ell}$ edges of~$C$.
By the union bound this holds for all $C\in\C$, and hence also for all $H$-free
$\ell$-graphs.

\begin{thm}\label{thm:ffree_sparse}
Let $H$ be an $\ell$-graph and let $0<\gamma<1$.
For some $c>0$, for $N$ sufficiently large and for $p \ge cN^{-1/m(H)}$,
the following event holds with probability greater than 
$1-\exp\{-\gamma^3 p {N \choose \ell} / 512 \}$:
$$
\mbox{every $H$-free subgraph of $G^{(\ell)}(N,p)$ has at most 
$(\pi(H) + \gamma) p{N \choose \ell}$ edges.}
$$
\end{thm}

Kohayakawa, \L{u}czak and R\"odl~\cite{KLR} further conjectured a stability
version of Theorem~\ref{thm:ffree_sparse}, proved by Conlon and
Gowers~\cite{CG} for strictly balanced graphs and by Samotij~\cite{S2},
following Schacht~\cite{Sch}, for all graphs.
They also made a stronger,
technical, conjecture which has become known as the K\L{R} conjecture,
proved recently for balanced graphs by Balogh, Morris and
Samotij~\cite{BMS}. Theorem~\ref{thm:ffree_cover} can be used to derive all
these conjectures in a straightforward way, and indeed a counting version
of Theorem~\ref{thm:ffree_cover} (Theorem~\ref{thm:ffree_cover_heavy})
yields a counting version of the K\L{R} conjecture. Because of the
technical descriptions needed, and the fact that these results are
consequences of Theorem~\ref{thm:ffree_cover}, we defer further details
to~\S\ref{sec:sparse}.

The same arguments can be applied to solution sets of linear
equations. Here is a typical consequence.

\begin{thm}[Conlon and Gowers~\cite{CG},
  Schacht~\cite{Sch}]\label{thm:szem_sparse}
Let $\ell\ge3$ and $\epsilon>0$. There exists a constant $c>0$ such that
for $p\ge cN^{-1/(\ell-1)}$, if $X \subset [N]$ is a random
subset chosen with probability $p$, then with probability tending to $1$ as
$N \to \infty$, any subset of $X$ of size $\epsilon|X|$ contains an arithmetic
progression of length $\ell$.
\end{thm}

Further examples and details can be found in~\cite{STa}.

\section{Containers}\label{sec:containers}

A couple of simple notions are needed for the statement of the main
theorem, and we define these now. They are the co-degree function and
degree measure. The co-degree function is what will determine the number
$|\C|$ of containers needed. The size of individual containers will be
specified in terms of degree measure. 

\subsection{The co-degree function $\delta(G,\tau)$}\label{subsec:codeg}

The main difficulties in the construction of containers are already present
in the case of simple hypergraphs, where the authors' original motivation
lay. However the method can be adapted efficiently to any hypergraph. The
number of containers we construct (and to a much lesser extent their size)
depends on the way the edges overlap, but the dependence can be
encapsulated by a single parameter. This parameter appears in most of the
theorems.

The theorems are stated in terms of a parameter $\tau$, whose meaning will
become clearer later, but for now it is enough to say that the number
$|\C|$ of containers constructed will be approximately $2^{\tau n}$. It is
evident, then, that we shall want $\tau$ to be as small as possible. What
determines how small $\tau$ can be is a bound on the {\em co-degree
  function} $\delta(G,\tau)$. This function is usually quite
straightforward to compute; it is just a polynomial in $1/\tau$ whose
coefficients are expressed in terms of the edge overlaps in~$G$.

Here is the precise definition. We first define the degree of a subset of
vertices, in the natural way.

\begin{defn} The {\em degree} of a set of vertices $\sigma\subset V(G)$ is
  the number of edges containing $\sigma$; that is,
  $$
  d(\sigma)\,=\,|\{\,e\in E(G)\,:\, \sigma \subset e\,\}|\,.
  $$
  If $|\sigma|=1$, that is $\sigma=\{v\}$ where $v\in V(G)$, we generally
  write $d(v)$ instead of $d(\{v\})$.
\end{defn}

We can now define the co-degree function $\delta(G,\tau)$.

\begin{defn}\label{defn:delta}
  Let $G$ be an $r$-graph of order $n$ and average
  degree~$d$. Let $\tau>0$. Given $v\in V(G)$ and $2\le j\le
  r$, let
$$
d^{(j)}(v)\,=\,\max\,\{\,d(\sigma)\,:\,v\in\sigma\subset
V(G),\,|\sigma|=j\,\}\,.
$$
If $d>0$ we define $\delta_j$ by the equation 
$$\delta_j\,\tau^{j-1}nd\,=\,\sum_v \,d^{(j)}(v)\,.$$
Then the {\em co-degree function} $\delta(G,\tau)$ is defined by
$$
\delta(G,\tau)\,=\,2^{\binom{r}{2}-1}\sum_{j=2}^r\,
2^{-\binom{j-1}{2}}\delta_j\,.
$$
If $d=0$ we define $\delta(G,\tau)=0$.
\end{defn}

The powers of 2 in the definition are rather eye-catching but they are
a distraction; they are constants
introduced to make Lemma~\ref{lem:mucont} work smoothly
(see the comment in~\S\ref{sec:postscript}). It does no harm for now to 
ignore them and to think of $\delta(G,\tau)$ as $\sum \delta_j$
or even as $\max\delta_j$.

Given a hypergraph $G$, the degree function $\delta(G,\tau)$ is a polynomial in
$1/\tau$ with positive coefficients (provided $e(G)\ne0$); in particular
$\delta(G,\tau)$ increases to infinity as $\tau$ decreases to zero. One of
the conditions of the main theorem, Theorem~\ref{thm:cover}, and of most of
the other theorems, is an upper bound on $\delta(G,\tau)$, which clearly is
equivalent to a lower bound on~$\tau$. 

It is helpful to have some feel for what values $\tau$ might take, and here
are some observations intended to indicate what happens. A typical
application will involve making $\delta(G,\tau)$ less than some
constant, never larger than $1/r!$, so let us see what this implies
for~$\tau$.

First of all, consider the simplest case, that of an ordinary graph, when
$r=2$.  Then $d^{(2)}(v)=0$ or~$1$, so
$\delta_2\tau nd\le n$, that is, $\delta_2\le 1/\tau d$. Hence
$\delta(G,\tau)=\delta_2\le 1/\tau d$. Thus $\delta(G,\tau)$ is small
provided $\tau$ is larger than $1/d$.

For general~$r$, observe that, unless $G$ has isolated vertices,
$d^{(j)}(v)\ge1$ holds for all~$v$, and so $\delta_j\ge \tau^{1-j}/d$. The
largest of these bounds is $\delta_r\ge \tau^{1-r}/d$ ($\tau$ is invariably
less than one) and so, for fixed~$r$ and large~$d$, it will always be that
for $\delta(G,\tau)$ to be small we must choose $\tau$ at least as large
as~$d^{-1/(r-1)}$.

In a simple hypergraph, $d(\sigma)\le 1$ holds whenever $|\sigma|\ge2$, and
so $\delta_j\le \tau^{1-j}/d$. In this case the largest of the $\delta_j$'s
is $\delta_r$, and we can make $\delta(G,\tau)$ small by choosing $\tau$
just a little larger than $d^{-1/(r-1)}$. In fact, for any hypergraph whose
edges are sufficiently uniformly distributed, $\delta_r$ is once again the
$\delta_j$
which dominates, as a simple calculation (which we omit) shows, so here
again $\delta(G,\tau)$ is small provided $\tau$ is larger than $d^{-1/(r-1)}$.

Sometimes, though, the dominant $\delta_j$ is not $\delta_r$. One example
of this is in the case of Sidon sets: when $|S|<n^{2/3}$ it is the value of
$\delta_2$ which is the most important (see~\cite{STa}). Another example is
the hypergraph describing $H$-free $\ell$-graphs: here the most important
$\delta_j$ is determined by whichever subgraph $H^\prime\subset H$ achieves the
maximum of $(e(H^\prime)-1)/(v(H^\prime)-\ell)$, and this is how $m(H)$
enters in (see Lemma~\ref{lem:ffree_delta}). But in each of our examples
the values are easily checked.

In summary, we must always choose $\tau\ge d^{-1/(r-1)}$, and for simple or
uniformly distributed hypergraphs the value need not be much larger. But
there are applications which are far from uniformly distributed, where
$\tau$ needs to be larger and where the behaviour of $\delta(G,\tau)$ will
prove crucial.

\subsection{Degree measure}\label{subsec:mu}

We mentioned in the introduction that the containers must not be too
large. For some applications it suffices that $|C|\le (1-c)|G|$ for some
constant~$c$. This is achievable for regular hypergraphs but it clearly is
unattainable in general; for example, if $G=K_{d,n-d}$ (which, for
large~$n$, has average degree close to~$2d$) then some container must have
size at least $n-d$. Other applications require that the number of edges
inside a container, that is, $e(G[C])$, is small. This is always
attainable, but a bound on $e(G[C])$ does not of itself imply a bound
on~$|C|$ suitable for the first kind of application.

We in fact measure the size of containers by what we call {\em degree
  measure}. It turns out that if the degree measure is bounded then it is
possible to recover all the properties of containers that are needed.

\begin{defn}\label{def:mu}
Let $G$ be an $r$-graph of order~$n$ and average degree~$d$. Let $S\subset
V(G)$. The {\em degree measure} $\mu(S)$ of $S$ is defined by
$$
\mu(S)\,=\,\frac{1}{nd}\,\sum_{u\in S}d(u)\,.
$$
\end{defn}

Thus $\mu$ is a probability measure on $V(G)$. Note that if $G$ is regular
then $\mu(S)=|S|/n$, which is the {\em uniform} measure of~$S$. Thus a
bound on $\mu(S)$ automatically gives a bound on $|S|$ for regular
graphs. For general graphs, obtaining a useful bound on $|S|$ from a bound
on $\mu(S)$ is a little more indirect (Lemma~\ref{lem:g}).

The dependence of $e(G[S])$ on $\mu(S)$ is much more straightforward,
by reason of the following inequality, in which $G$ is an $r$-graph
of order $n$ and average degree~$d$:
\begin{equation}
e(G[S])\le (1/r)\sum_{v\in  S}d(v)=(1/r)\mu(S)nd = \mu(S)e(G)\,.
\label{eqn:mutoe}
\end{equation}
Hence a bound on $\mu(S)$ at once gives a bound on $e(G[S])$.

We mentioned at the outset of the paper that each container should not be
too large. In a regular $r$-graph it is easily shown that $|I|\le (1-1/r)n$
for every independent set~$I$. However, for general $r$-graphs the ratio
$|I|/n$ can be arbitrarily close to one. An important feature of degree
measure is that the bound $\mu(I)\le 1-1/r$ holds for all independent sets
in all $r$-graphs (see inequality~(\ref{eqn:mubig}) and the
remark following it). So we might hope that every $r$-graph has containers
with $\mu(C)$ bounded away from one, and this is exactly how things turn
out. A bound of this kind is enough to meet our needs. In the next
subsection we state the main theorem, and afterwards,
in~\S\ref{subsec:tight} and~\S\ref{subsec:uniformly}, we indicate how it
can yield containers either with $e(G[C])$ small or with $|S|$ small.

\subsection{The main theorem}\label{subsec:mainthm}

The essential idea underlying the main theorem is this: there is an
algorithm that, from any small set $T\subset V(G)$ of vertices,
produces another subset $C=C(T)\subset V(G)$. Typically, $C(T)$ is much
larger than~$T$, but it is guaranteed that $\mu(C)$ is bounded away from
one. Moreover, and importantly, for any independent set~$I$, there is some
small subset $T\subset I$ such that $I\subset C(T)$.

Observe now that if we define $\C$ to be the collection of all sets $C(T)$
produced from small sets~$T$, then this collection $\C$ is a collection of
containers having exactly the properties we want. The construction
guarantees that for each independent set $I$ there is a container $C\in\C$
with $I\subset C$, with $\mu(C)$ bounded away from one, and the number
$|\C|$ of containers is at most the number of small sets, which gives a
useable upper bound on~$|\C|$.

We already introduced the parameter~$\tau$. This parameter will measure how
big the sets $T$ must be in order for the theorem to work: essentially,
$\tau$ will be the value of $\mu(T)$. For regular $G$ this means $|T|=\tau
n$, and the number of sets of this size is ({\em very} approximately)
$2^{\tau n}$, so explaining the bound $|\C|\le 2^{\tau n}$ referred to
in~\S\ref{subsec:codeg}. This is why we want $\tau$ to be as small as
possible.

The theorem includes a parameter $\zeta$, which is some small constant at
our disposal. Often we shall take $\zeta=1/12r!$ but sometimes it is useful
(such as in the list colouring application)
to choose a smaller value. The constraint on the size of $\tau$ in the
theorem arises from the requirement that $\delta(G,\tau)\le \zeta$. As
discussed in~\S\ref{subsec:codeg}, this inequality implies a lower bound
on~$\tau$.  We would thus want to take $\zeta$ as large as possible, but
taking it too large spoils the bound on $\mu(C)$. The choice $\zeta=1/12r!$
generally works well. Recall from the discussion in~\S\ref{subsec:codeg}
that $\tau$ is then a small negative power of~$d$. Thus, for large $d$,
$\tau$ will be vanishingly small compared to the constants $r$ and~$\zeta$.

The preceding comments should help to illuminate the main thrust of the
theorem, but some further comments should be made about the detailed statement.

First of all, we shall not actually generate $C$ from just one small set
but instead from an $r$-tuple $T=(T_{r-1},\ldots,T_0)\in \mathcal{P}(I)^r$ of
small sets. The principles of the remarks made above remain the same. The
precise bound on the size of $T$ is $\mu(T_i)\le 2\tau/\zeta$, which (in
the light of what has been said) is of order~$\tau$. 

Secondly, we use a piece of shorthand.
  Let $T=(T_{r-1},\ldots,T_1,T_0)\in\mathcal{P}([n])^r$ and let $w\in[n]$.
  Then we define
  $
  T\cap[w]\,=\, (T_{r-1}\cap [w],\ldots,T_1\cap [w],T_0\cap [w]).
  $
The relevance of this will be discussed later.

Finally, we say that an $r$-graph $H$ is {\em $b$-degenerate} if 
$e(H[S])\le b|S|$ for every
subset $S\subset V(H)$.

\begin{thm}\label{thm:cover}
  Let $G$ be an $r$-graph with vertex
  set~$[n]$. Let
  $\tau,\zeta>0$ satisfy $\delta(G,\tau)\le\zeta$. Then there
  is a function $C:\mathcal{P}([n])^r\to\mathcal{P}([n])$, such that, for every
  independent set $I\subset[n]$ there exists
  $T=(T_{r-1},\ldots,T_0)\in\mathcal{P}(I)^r$ with
 \begin{itemize}
 \item[(a)] $I \subset C(T)$,
 \item[(b)] $\mu(T_0), \ldots, \mu(T_{r-1}) \le 2\tau/\zeta$, 
 \item[(c)] $|T_0|,\ldots,|T_{r-1}|\le 2\tau n/\zeta^2$,  and
 \item[(d)] $\mu(C(T)) \le 1 - 1/r! + 4\zeta+2r\tau/\zeta$.
 \end{itemize}
 Moreover if $G$ is simple then
 $C(T)\cap[w]=C(T\cap[w])\cap[w]$ for all $T\in\mathcal{P}([n])^r$ and
 $w\in[n]$.

 In fact, the above is true for all sets $I\subset[n]$
 for which either $G[I]$ is $\lfloor\tau^{r-1} \zeta e(G)/n\rfloor$-degenerate
 or $e(G[I])\le 2r \tau^re(G)/ \zeta$.
\end{thm}

\begin{rem}\label{rem:cover} The discussion preceding the theorem has
  hopefully helped to explain it, but a few more observations are worth making.
\begin{itemize}
\item Roughly speaking, the theorem says that for each $I$ there exists
  $T\subset I$ with $\mu(T)\lesssim\tau$, $I \subset C(T)$ and $\mu(C)\lesssim
  1-1/r!$, provided $\tau$ is large enough to make $\delta(G,\tau)$ small.
\item Assertions (b) and (c) each offer different, though obviously related,
  ways to bound the size of~$T$; for each bound, there are applications
  where it is the more convenient.
\item We refer to the property $C(T)\cap[w]=C(T\cap[w])\cap[w]$, which
  holds for simple graphs (but see~\S\ref{sec:postscript}), as the {\em
    online property}, because the construction is behaving somewhat like
  an online algorithm: the vertices of the container lying
  within the first $w$ vertices are already determined by
  $T\cap[w]$. (Nevertheless, knowledge of the whole of $G$ is needed to
  determine $C(T\cap[w])$.) The online property is important only for
  certain applications, principally Theorem~\ref{thm:uniform}. For now, the
  property can safely be ignored.
\item The container construction method makes essentially no use of the
  independence of the sets~$I$, so we include an extension to two kinds of
  sparse subset, where either $G[I]$ is $b$-degenerate for some small~$b$,
  or else $e(G[I])\le bn$ for some other~$b$. Both types of sparsity are
  useful. Allowing $e(G[I]) \le bn$ is used in
  Theorem~\ref{thm:ffree_cover_heavy}, and hence Theorem~\ref{thm:klr3};
  $b$-degeneracy is used in Theorem~\ref{thm:list_colour_planar}.
\end{itemize}
\end{rem}

As mentioned in the introduction, Theorem~\ref{thm:cover} has a variety of
consequences and weaker forms which are easier to apply directly.  We state
a couple of them now: they are developed further
in~\S\ref{sec:iteration}--\S\ref{sec:uniform}.

\subsection{Tight containers}\label{subsec:tight}

The first corollary is packaged for use when we want $e(G[C])$ to be small
for each container~$C$. This is the corollary we use to prove, say,
Theorem~\ref{thm:ffree_cover} in~\S\ref{sec:hfree}. It makes no mention of
degree measure. 

The way to make $e(G[C])$ small is to apply the container theorem
repeatedly, as follows.  Suppose $I$ is an independent set in~$G$. Observe
that Theorem~\ref{thm:cover} gives a container $C$ with $I\subset C$ and
$\mu(C)\le 1-c$, where $c$ is some positive constant (perhaps around
$1/r!$). By inequality~(\ref{eqn:mutoe}) this means $e(G[C])\le
(1-c)e(G)$. But $I$ is an independent subset in $G[C]$, so we can apply
Theorem~\ref{thm:cover} again, this time to the graph $G[C]$, to obtain a
container $C'$ with $I\subset C'$ and $e(G[C'])\le (1-c)e(G[C])\le (1-c)^2
e(G[C])$. Repeated applications allow us to obtain containers with as few
edges as we wish, the only constraint being that the main theorem ceases to
be effective when the number of edges remaining is very small. Of course,
repeated applications increase the total number of containers, but this
turns out to be inexpensive.

The following corollary is the simplest of the ones obtained in this way, in
which, as usual, the size of the collection of containers is bounded by a
simple function of~$\tau$.

\begin{cor}\label{cor:sparse_container}
  Let $G$ be an $r$-graph on vertex set $[n]$. Let
  $0<\epsilon, \tau<1/2$. Suppose that $\tau$ satisfies $\delta(G,\tau) \le
  \epsilon/12r!$. Then there exists a constant
  $c=c(r)$, and a function $C:\mathcal{P}([n])^s\to \mathcal{P}[n]$ where
  $s\le c\log(1/\epsilon)$, with the following properties.
  Let $\mathcal{T}=\{(T_1,\ldots,T_s)\in \mathcal{P}([n])^s: |T_i|\le
  c\tau n, 1\le i\le s\}$, and let $\C =\{C(T):T\in\mathcal{T}\}$. Then
\begin{itemize}
 \item[(a)] for every independent set $I$ there exists
  $T=(T_1,\ldots,T_s)\in\mathcal{T}\cap\mathcal{P}(I)^s$ with $I \subset C(T)
  \in \C$,
 \item[(b)] $e(G[C]) \le \epsilon e(G)$ for all $C \in \C$,
 \item[(c)] $\log|\C| \le c \log(1/\epsilon) n \tau \log(1/\tau)$.
\end{itemize}

Moreover, (a) holds for all sets $I\subset[n]$ for which either $G[I]$ is 
$\lfloor\epsilon\tau^{r-1} e(G)/12r!n\rfloor$-degenerate or
$e(G[I])\le 24 \epsilon r!r\tau^r e(G)$.
\end{cor}

The main points here are again that~(a) shows there is a container for each
independent set, (b)~shows that each container has few internal edges,
and~(c) bounds the size of~$|\C|$. Condition~(a) contains the extra
information that the container $C(T)$ for~$I$ is constructed from $T$, a
few small subsets of~$I$, which can be useful sometimes, as mentioned
in the discussion after Theorem~\ref{thm:ffree_cover}.

The corollary holds provided $\delta(G,\tau)$ is bounded above as
specified. As discussed
in~\S\ref{subsec:codeg}, this implies a lower bound on~$\tau$. In applications
where it matters, the value $c(r)=800r!^3r$ can be taken.

Corollary~\ref{cor:sparse_container} is proved in~\S\ref{sec:iteration},
together with a finer result of this kind.

\subsection{Uniformly bounded containers}\label{subsec:uniformly}

Next we give a consequence of Theorem~\ref{thm:cover} packaged for
applications when the size $|C|$ of the container is of interest. We shall
use it to prove the list colouring result, Theorem~\ref{subsec:list}
in~\S\ref{thm:chil}. The package is somewhat more subtle than
Corollary~\ref{cor:sparse_container}. We would like a bound on
$|C|$ of the form $(1-c)n$ for some constant~$c$, but, as noted before
in~\S\ref{subsec:mu}, this does not always hold, say when
$G=K_{d,n-d}$. What can be said in such circumstances that is useful?

Given $S\subset V(G)$, write $\overline S$ for $V(G)-S$ and $e(\overline
S,S)$ for the number of edges meeting both $\overline S$ and~$S$. The sum
$\sum_{v\in S}d(v)$ counts edges inside $G[S]$ $r$ times each, together with
edges meeting both $\overline S$ and~$S$ at most $r-1$ times each. Hence
$$
\sum_{v\in S}d(v)-re(G[S])\,\le\,(r-1) e(\overline S, S)\,\le\,
(r-1)\sum_{v\notin S}d(v)=(r-1)\mu(\overline S)nd
$$
or, in other words, $\mu(S)-re(G[S])/nd\le (r-1)\mu(\overline S)$. Now
$\mu(\overline S)=1-\mu(S)$ and so we have
\begin{equation}\label{eqn:mubig}
e(G[S])\,\ge\,(\mu(S)-1+{1\over r})\,nd\,.
\end{equation}
In particular, as mentioned in~\S\ref{subsec:mu}, if $I$ is independent
then $\mu(I)\le 1-1/r$. Moreover if $e(G[C])\le \eps e(G)=\eps nd/r$, as in
Corollary~\ref{cor:sparse_container}, then $\mu(C)\le 1-1/r+\eps/r$. Now
if $G$ is regular, then degree and uniform measures coincide; therefore in
this case Corollary~\ref{cor:sparse_container} supplies containers with
$|C|\le (1-1/r+\eps/r)n$.

As we know, we cannot always bound $|C|$ usefully for non-regular
graphs. However, it turns out we can use a bound on $\mu(C)$ to bound the
ratio $|C\cap[v]|/v$ away from one for some values of~$v$ and, when
suitably formulated, such a bound is enough for our application to list
colouring. In order to establish this bound, we shall need the online
property, and so we prove the bound only for simple graphs (for which the
online property holds --- however, see~\S\ref{sec:postscript}).  An
important point is that we cannot make use of
iteration as in~\S\ref{subsec:tight} to obtain smaller containers, because
the online property does not survive iteration. (Indeed, consider the
discussion in~\S\ref{subsec:tight}, and suppose $C$ is determined by $T$
and then $C'$ by~$T'$. For $C'$ to have the online property we would need
to be able to determine $C^\prime \cap [v]$ from $T\cap[v]$ together with
$T'\cap[v]$. But, to compute $C^\prime \cap [v]$, it is necessary to know
the whole of $G[C]$, for which it is necessary to know the whole of~$C$,
and thus to know the whole of~$T$ rather than just $T\cap[v]$; hence the
online property fails.) Since we cannot use iteration, we have an interest
in making the bound on $\mu(C)$ in Theorem~\ref{thm:cover} as small as we
can (see~\S\ref{subsec:optimality}).

Let $\C$ be a collection of containers for~$G$. For each initial
segment~$[v]$ of the vertex set~$[n]$, the set $\C_v=\{C\cap[v]: C\in\C\}$
of restrictions to $[v]$ is a collection of containers for the induced
subgraph~$G[v]$. For our application, it would be enough to find a segment
$[v]$ for which the collection $\C_v$ of restrictions is
well-behaved: that is, $|C\cap [v]|\le(1-c)v$ for each $C$ and $|C_v|\le
2^{\tau v}$. We could then work with the subgraph $G[v]$ rather than
with~$G$. In the example of $G=K_{d,n-d}$, the first $2d$ vertices induce
$K_{d,d}$, so we could take $v=2d$ and find good containers for $G[v]$.
But $G$ does not always have such a nice subgraph
(see~\S\ref{sec:online}).

Something slightly different does work, though. Each container $C=C(T)$
nominates a vertex $v=g(C)$ so that both restrictions $C\cap[v]$ and
$T\cap[v]$ are simultaneously constrained (Lemma~\ref{lem:g}). By the
online property, $T\cap[v]$ determines $C\cap[v]$. This limits the number
of possible sets $C\cap[v]$, and it turns out to be enough for our
application.

To state the precise theorem, we make a couple of technical changes to the
outline just given. First, we work with tuples $(C_1,\ldots,C_t)$ of
containers rather than individual containers, since this is ultimately more
efficient. Secondly, we include a lower bound $k$ on $g(C)$ to make sure
$v$ is not too small. In order to convert degree to uniform measure, we
also ask that the vertices be ordered by decreasing degree.

\begin{thm}\label{thm:uniform}
  Let $G$ be a simple $r$-graph on vertex set $[n]$, for
  which the degree sequence is decreasing. Let $0<\zeta\le 1/12r!$.
  Suppose that $\delta(G,\tau)\le \zeta$, that $\tau\le\zeta^2/r$, and
  that $k\in[n]$ satisfies $\mu([k])\le \zeta/2r!$. Let $t\in\mathbb{N}$.

  Then there exists a collection $\C\subset\mathcal{P}[n]$ and a map
  $g:\C^t\to[k,n]$, with the following properties:
 \begin{itemize}
 \item[(a)] for all independent sets $I$ there is some $C\in\C$ with
   $I\subset C$,
 \item[(b)] for all $v\in[n]$
$$
\log\left|\{\,(C_1\cap[v],\ldots,C_t\cap[v])\,:\,g(C_1,\ldots,C_t)=v\,
\}\right|\,\le\, r\zeta^{-2} vt\tau\log(1/\tau),
$$
 \item[(c)] and for all $(C_1,\ldots,C_t)\in{\C}^t$
$$
\frac{1}{t}\sum_{i=1}^t|C_i\cap[v]|\,\le\,(1-\frac{1}{r!}+8\zeta)v\,,
$$
where $v=g(C_1,\ldots,C_t)$.
 \end{itemize}

 Moreover, (a) holds for all sets $I\subset[n]$ for which either $G[I]$ is
 $\lfloor\tau^{r-1}\zeta e(G)/n\rfloor$-degenerate
 or $e(G[I])\le 2 r\tau^re(G)/ \zeta$.
\end{thm}

The main features of this theorem are hopefully recognisable by now: (a)
means each independent set is in a container, (b) means each tuple of
containers nominates a vertex~$v$, and the number of restricted containers
for any nominated~$v$ is small in terms of~$v$, and (c) means the uniform
measure of the restricted containers is bounded away from one. 

It is worth noting that, for regular graphs, Theorem~\ref{thm:uniform}
follows (up to the odd constant) from Corollary~\ref{cor:sparse_container},
because in that case, as remarked following inequality~(\ref{eqn:mubig}),
$|C|\le (1-1/r + \eps/r)$ for every $C\in\C$, and we can just take
$g(C_1,\ldots,C_t)=n$.

Theorem~\ref{thm:uniform} is proved in~\S\ref{sec:uniform} and applied in
the proof of Theorem~\ref{thm:chil}.

\subsection{Optimality}\label{subsec:optimality}

We conclude this section about the main theorem with some observations
as to what extent it is best possible. There are two aspects to optimality:
the bound on $\mu(C)$ and the bound on $|\C|$, the latter being implied by
the size of~$\tau$.

It is easy to produce examples of $G$ and $I$ with $\mu(I)= 1-1/r$, and so
the best bound on $\mu(C)$ that one could hope for in general is $\mu(C)\le
1-1/r$. The bound in Theorem~\ref{thm:cover} is, essentially, $\mu(C)\le
1-1/r!$. In fact the algorithm that we use to prove Theorem~\ref{thm:cover}
does not give containers smaller than this: there is a description
in~\cite{STx} and in~\cite{S} of an example of a graph $G$, an ordering of
its vertices, and an independent set~$I$ such that the container $C$
constructed for $I$ satisfies $\mu(C)\approx 1-1/r!$. In this sense,
Theorem~\ref{thm:cover} is best possible. The fact that the
algorithm does achieve this bound is proved in~\S\ref{subsec:container} and,
for reasons mentioned in~\S\ref{subsec:uniformly}, we put some effort into
that proof, even though a shorter argument would give a useful but weaker bound.

We remark that the simple algorithm in~\cite{STe} gives $\mu(C)\le
1-1/4r^2$ in one shot, but the number of containers produced is larger
than here. (That method too applies only to simple hypergraphs, though
this is not quite such a drawback as might at first appear.)

The more important aspect of optimality regarding Theorem~\ref{thm:cover}
is the bound on $|\C|$ implicit in the bound
on~$\tau$. Theorem~\ref{thm:optimality} below states in what ways the bound is
optimal, but before stating the theorem we explain informally what lies
behind it.

A simple counting argument shows that if most sets of size $t$ are
independent, or nearly independent (having many fewer edges than vertices),
then any collection $\C$ of containers must satisfy $|\C|=e^{\Omega(t)}$,
provided each container $C\in\C$ has size bounded away from~$n$. In a graph
of order $n$ and average degree~$d$, most sets of size $o(nd^{-1})$ are
nearly independent; for $r$-graphs the same is true of most sets of size
$o(nd^{-1/(r-1)})$. This is already enough to show that the main theorem
comes quite close to being best possible, given that (as remarked
in~\S\ref{subsec:codeg}) $\tau$ is often of order $d^{-1/(r-1)}$.

We can say even more for certain kinds of containers. Let us call $\C$ {\em
  internally generated} if there is some function $C(T)$ such that
$\C=\{C(T): T\in\mathcal{T}\}$, where $\mathcal{T}\subset
\mathcal{P}(V(G))$, and for every independent set~$I$, there exists
$T\in\mathcal{T}$ with $T\subset I\subset C(T)$.  In this case, we show
that $\mathcal T$ must contain sets $T$ of size at least $\Omega(t)$ and,
moreover, $|\mathcal{T}|=e^{\Omega(t n\log (1/t))}$. Again, if $\tau$ is of
order $d^{-1/(r-1)}$, this indicates that the bounds in
Corollary~\ref{cor:sparse_container} are best possible to within a constant
factor, and, certainly if $C(T)$ is injective, the logarithmic factor
in~(c) cannot be removed, meaning~(a) must be retained. (We do not assert,
though, that the function $C(T)$ given by the algorithm in this paper is
injective.)

Finally, the constraint imposed on $\tau$ by the co-degree function
$\delta(G,\tau)$ is, in a sense, also optimal. For this we assume a certain
amount of symmetry in the $r$-graph~$G$, say $G$ is vertex and
edge transitive: this condition certainly holds in several cases of
interest, such as in~\S\ref{subsec:hfree} and~\S\ref{subsec:sparse} when
$G=G(N,H)$.  Given an $r$-graph $G$, let $G_j$ be the $j$-graph whose edges
are the $j$-sets $\sigma$ of maximum degree in~$G$. By symmetry, every edge
of $G$ contains an edge of $G_j$ and so, crucially, every independent set
of $G_j$ is independent in $G$ too. Thus a collection $\C$ of containers
for $G$ furnishes a collection for~$G_j$; hence, by considering nearly
independent sets in~$G_j$, we can obtain lower bounds on $|\C|$ in terms of
the average degree of~$G_j$. The maximum over all $j$ of these bounds turns
out to be exactly (up to constant factors) the size of $\C$ given by
Corollary~\ref{cor:sparse_container}, determined by the constraint
on~$\delta(G,\tau)$. Hence the co-degree function is in some way capturing
the right property of~$G$. Readers familiar with the parameter
$m(H)$ defined in~\S\ref{subsec:hfree} will perhaps recognise the spirit of
this argument and so sense why $m(H)$ appears in
Theorem~\ref{thm:ffree_cover}.

\begin{thm}\label{thm:optimality}
Let $G$ be an $r$-graph of average degree~$d$ and vertex set~$[n]$. Let
$\C\subset {\mathcal P}[n]$ be such that, for every independent set $I$ of
$G$ there is some $C\in\C$ with $I\subset C$. Suppose that $|C|\le (1-c)n$
for all $C\in\C$. Then there is a positive constant
$\gamma=\gamma(r,c)$ such that the following hold.
\begin{itemize}
\item[(i)] $|\C|\ge e^{\gamma nt}$, where $t=d^{-1/(r-1)}$;
\item[(ii)] if $\C$ is internally generated (see above) and $\C=\{C(T):
  T\in\mathcal{T}\}$, then $|T|\ge \gamma tn$ for some $T\in\mathcal{T}$,
  and $|\mathcal{T}|\ge e^{\gamma nt\log (1/t)}$, where $t=d^{-1/(r-1)}$;
\item[(iii)] if $G$ is vertex and edge transitive, and $\delta(G,\tau)\ge1$,
  then (i) and (ii) hold with $t=\tau$.
\end{itemize}
\end{thm}

The proof of this theorem is given in~\S\ref{sec:optimality}.

\section{The Algorithm}\label{sec:online}

In this section we describe the method of building containers and establish
the basic facts about them.

Even for 2-graphs, it is not immediate that a useful container theorem
exists. The starting point for our method is the work of
Sapozhenko~\cite{Sap1,Sap2,Sap5,Sap4}, who gave a way to build containers for
regular 2-graphs. In~\S\ref{subsec:egr2} we describe our method for
2-graphs, which illustrates some of the essential features of the general
method, though obviously not all.

This is only the starting point, of course, because a method for 2-graphs
gives very little clue as to how to approach $r$-graphs. A method we found
that works for simple regular $r$-graphs was described in~\cite{ST}
(refined in~\cite{STe}), good enough to produce a good lower bound for
$\chi_l(G)$. This would yield a similar bound for non-regular $r$-graphs if it
were true that every such graph contains an almost regular subgraph. The
requirement on the subgraph here can be made extremely weak but
nevertheless it cannot be satisfied; there are examples of $r$-graphs where
every subgraph is far from regular, somewhat along the lines of the
$2$-graphs of Pyber, R\"odl and Szemer\'edi~\cite{PRS}. A construction due
to Verstra\"ete is described by Dellamonica and R\"odl~\cite{DR}.

One of our requirements for a good construction is that it must satisfy the
online property, which further limits the options. The method described
here fulfils these needs (though see~\S\ref{sec:postscript}). It is in some
ways almost opposite in
approach to that in~\cite{ST,STe}. We endeavour to motivate the method to
the extent that we can. Nevertheless, the reason behind one or two features
might become clearer after reading~\S\ref{sec:calc}, in which the
properties of the algorithm are proven.

\subsection{Example for $r=2$}
\label{subsec:egr2}

To introduce some (but not all) of the ideas used in the algorithm, we
prove Theorem~\ref{thm:cover} for independent sets in $2$-graphs.  As
mentioned, for regular graphs, the strategy reduces to something close to,
but not identical to, that of Sapozhenko~\cite{Sap1,Sap2,Sap5,Sap4}.  We have
tried to make the notation as similar as possible to that used in the main
algorithm, although the reader should be aware that there are some
differences.

An important general feature of the method can be described immediately,
which holds as well for $r$-graphs as for 2-graphs. The construction of a
subset $T_1$ from an independent set~$I$, and the construction from~$T_1$
of a set which contains~$I$, are achieved by the same algorithm, run in two
slightly different modes that we call {\em prune mode} and {\em build
  mode}. In prune mode the algorithm receives $I$ as input and it
outputs~$T_1$; in build mode the algorithm receives $T_1$ as input and it
outputs a container.

Let the vertices of $G$ be the set $[n]$. This just gives an ordering to
the vertices: we assume no properties of the ordering. In prune mode, the
algorithm begins with $T_1=\emptyset$. It then examines the vertices one by
one in the order $1,\ldots,n$; when it reaches vertex~$v$, it checks
whether $v$ is in $I$ and whether $v$ has some further property --- we
might say there is some ``membership rule'' that $v$ must satisfy. If $v\in I$
and $v$ passes the membership rule, then $v$ is added to~$T_1$, otherwise $v$
is not added to~$T_1$. This is done for $v=1,\ldots,n$ in turn, and the
algorithm then outputs~$T_1$. In build mode, the algorithm initialises a
set~$C_1=[n]$. It then examines each vertex $v$ in turn. If $v$
satisfies the membership rule, then $v$ is removed from~$C_1$. This is done for
$v=1,\ldots,n$ in turn, and the algorithm then outputs $C_1\cup T_1$.

It will be seen that $T_1\subset I$, by construction, and that $I\subset
C_1\cup T_1$, because the vertices left in $C_1$ failed the membership rule,
and this includes all the vertices of $I$ except those in~$T_1$. (Notice in
passing that we did not need $I$ to be independent.)

As stated, the build algorithm did not appear to make use of the
input~$T_1$. However, the algorithm will also construct an auxiliary
structure (in the case $r=2$, this is just a set) along the way. The membership
rule is specified in terms of the current state of this auxiliary
structure. The structure will be updated only when a vertex is in $T_1$ and
passes the rule, so both modes of the algorithm have the information to
update the structure properly.

To be more specific, here is the membership rule we use for $r=2$: the
requirement is of course to produce small sets~$T_1$ and not too large
sets~$C_1$. Let $G$ be a $2$-graph with vertex set~$[n]$ and average
degree~$d$. Let~$\zeta>0$. The algorithm uses an auxiliary set $\Gamma_1
\subset [n]$, which is initally empty. For $v=1,\ldots,n$, consider
\[
F(v) = \{w \in [v+1,n] : \{v,w\} \in E(G) \mbox{ and } w \not\in \Gamma_1 \}.
\]
We take the following as the membership rule: that $|F(v)|\ge \zeta d(v)$.  In
addition to checking the rule, if $v\in T_1$ then the algorithm adds $F(v)$
to~$\Gamma_1$. Both modes of the algorithm know whether $v\in T_1$: prune
mode because it is constructing~$T_1$, and build mode because it is
given~$T_1$. Hence both modes construct the same $\Gamma_1$ and both are
using the same rule. 

Note that, as the algorithm proceeds, $\Gamma_1$ is just the set of
vertices which have an earlier neighbour in~$T_1$, because $F(v)$ is
precisely those vertices $w$ for which $v$ is an earlier neighbour but no
other vertex yet in $T_1$~is.

We now use, for the first time, the fact that $I$ is independent. Because
the vertices in $\Gamma_1$ are neighbours of those in~$T_1\subset I$, we
know $I\cap \Gamma_1=\emptyset$. We also know that $I\subset C_1\cup T_1$,
the set output by build mode. So our final container is $C(T_1)=(C_1\cup
T_1)-\Gamma_1$. The notation $C(T_1)$ means that $C(T_1)$ can be
constructed just from~$T_1$.

The function $C:\mathcal{P}[n] \to \mathcal{P}[n]$ satisfies the main points of
Theorem~\ref{thm:cover} for $r=2$. Recall from~\S\ref{subsec:codeg} that
$\delta(G,\tau)\le 1/\tau d$ for a 2-graph, so take $\tau=1/\zeta d$. The
set $T_0$ is not needed so take $T_0=\emptyset$. We know $I\subset C(T_1)$,
which is assertion~(a) of the theorem. Now we check the measures of $T_1$
and of $C(T_1)$.

Whenever a vertex $v$ is added to~$T_1$, the rule is satisfied, so
$\Gamma_1$ increases by at least $\zeta d(v)$. But $|\Gamma_1|\le n$, so $
\zeta \mu(T_1) = (1/nd) \sum_{v \in T_1} \zeta d(v) \le (1/n
d)|\Gamma_1|\le 1/d $. Therefore $\mu(T_1)\le 1/\zeta d$, which gives
Theorem~\ref{thm:cover}(b) comfortably.

Write $C^*=C_1-\Gamma_1$. Every $w$, with $vw\in E(C^*)$ for some $v<w$, lies in
$F(v)$ by definition of $C_1$. (Note here that $\Gamma_1$ grows during the
procedure so, at the time $F(v)$ was defined, $\Gamma_1$ might have been
smaller than it is at the end, or in other words there might have been
$u\in F(v)$ when the rule was tested, such that $u$ ends up in
$\Gamma_1$ and not in~$C^*$. But this only helps.) Since $v\in C_1$, $v$
failed the rule, so $|F(v)|< \zeta d(v)$. Therefore $ e(G[C^*]) \le
\sum_{v\in C^*} \zeta d(v) \le \zeta nd $.  On the other hand,
$e(G[C^*])\ge (\mu(C^*)-1/2)nd$ by inequality~(\ref{eqn:mubig}). Thus
$\mu(C^*) \le 1/2 + \zeta$. Finally, $C(T_1) = (C_1\cup
T_1)-\Gamma_1\subset C^*\cup T_1$ and so $ \mu(C) \le \mu(C^*)+\mu(T_1)\le
1/2 + \zeta + 1/\zeta d $, giving Theorem~\ref{thm:cover}(d).

We have not completely proved Theorem~\ref{thm:cover} for $r=2$, since we
have not shown condition~(c): indeed, at present the bound on the degree
measure of~$T_1$ does not imply a bound on~$|T_1|$, since this requires a
lower bound on the degrees of the vertices of the graph. One way round this
is to amend the membership rule so as to allow $v\in T_1$ only if $d(v)\ge\zeta
d$. It can be shown that this increases $\mu(C(T_1))$ by only~$\zeta$, and
all vertices in $T_1$ now have degree at least~$\zeta d$, so $|T_1|\zeta d
\le \sum_{v\in T_1}d(v)=\mu(T_1)nd$. Condition~(c) then follows from~(b).

We did not check the online property, but that is not hard to do. Nor did
we check the cases where $I$ is not independent, but, as has been seen, the
independence of~$I$ was barely used. In fact, the set $T_0$ is there to
take care of these cases.

\subsection{The general algorithm}\label{subsec:algorithm}

For general $r$-graphs we use the same method of an algorithm running in
prune/build mode as used in~\S\ref{subsec:egr2}, but we need a membership rule
that will handle edges of size $r>2$. Moreover, if the $r$-graph is not
simple, we need to handle overlapping edges carefully.

Given an $r$-graph $G$ with vertex set $[n]$, we in fact run the algorithm $r-1$
times; we label these runs by $s=r-1,r-2,\ldots,1$ in turn. In run~$s$, the
algorithm has input~$I$ and output~$T_s$ (in prune mode) or input $T_s$ and
output $C_s$ (in build mode). Along the way it builds an auxiliary {\em
  multigraph}~$P_s$. Here, $P_s$ is $s$-uniform but multiple edges are
allowed; in other words $E(P_s)$ is a {\em multiset}. The multigraph $P_s$
is constructed from $P_{s+1}$, which is supplied by the previous run of the
algorithm. For the first run, when $s=r-1$, we supply $P_r=G$.

Each edge $\{u_{s-1},u_{s-2},\ldots,u_0\}\in E(P_s)$ with
$u_{s-1}<u_{s-2}<\cdots<u_0$ will come from an edge
$\{v_{r-1},v_{r-2},\cdots,v_s,u_{s-1},u_{s-2},\ldots,u_0\}\in E(G)$, where
$v_{r-1}<\cdots<v_s<u_{s-1}$ and $v_j\in T_j$, $r-1\ge j\ge s$.
Equivalently, each edge of $P_s$ is an edge of $P_{s+1}$ whose first
vertex, which is in~$T_s$, has been removed. The reason $P_s$ is defined as
a multigraph, even if $G$ itself does not have multiple edges, is so that
distinct edges of $G$ give rise to distinct edges of $P_s$. As will be seen,
this allows more vertices to be added to~$T_s$, which in turn conveys more
information about the independent set.

The multigraph $P_1$ is $1$-uniform: its edges are sets containing single
vertices. If $\{u_0\}\in E(P_1)$ then there is an edge
$\{v_{r-1},v_{r-2},\cdots,v_1,u_0\}\in E(G)$ with $v_j\in T_j$, $r-1\ge
j\ge 1$. So, if $I$ is an independent set and the sets $T_j$ are chosen
within $I$, as they will be, then $u_0\notin I$, and so the container $C$
can be chosen from vertices not in~$E(P_1)$. Note here that $E(P_1)$ is
playing the role that $\Gamma_1$ did in~\S\ref{subsec:egr2}. Our first aim,
then, is to ensure that $E(P_1)$ is as large as possible, and to this end
we attempt to make $E(P_s)$ large for each~$s$. However this aim has to be
balanced against keeping the sets $T_s$ small.

Hence we shall choose a parameter $\tau$, so that, roughly speaking, $T_s$
will comprise a proportion $\tau$ of the vertices (in degree measure), and
we aim to design the algorithm so that, ideally, the size of $E(P_s)$ will be
roughly $\tau$ times the size of $E(P_{s+1})$. This means the average
degree of $P_s$ will typically be around $\tau^{r-s}d$.  The parameter
$\tau$ is the same as that discussed in~\S\ref{subsec:codeg} and the
constraint $\tau\ge d^{-1/(r-1)}$ described there is precisely what is
needed to ensure that $E(P_1)$ contains something worthwhile.

However not every edge of $G$ with its first $r-s$ vertices in
$T_{r-1},\ldots,T_s$ will be admitted as an edge of~$P_s$, but only a
selection of these. We do not allow edges into $P_s$ if they increase the
degree of some vertex, or the degree of some subset $\sigma\subset[n]$,
beyond some agreed threshold. We define the degree of
$\sigma$ in the multigraph~$P_s$ to be
$$
d_s(\sigma)\,=\,|\{e\in E(P_s):\sigma\subset e\}|\,,
$$
where we are counting edges with multiplicity in the
multiset $E(P_s)$. (Naturally we may write $d_s(v)$ instead of $d_s(\{v\})$
if $v\in[n]$.)
There are several reasons for wanting to
bound the degrees in $P_s$. One reason is the hope of keeping the vertex
degrees near to $\tau^{r-s}$ times the degrees in~$G$, so that degree
measure in $P_s$ relates to measure in~$G$; in particular, small sets of
vertices cannot account for most of the edges of $P_s$ unless those sets
have large measure in~$G$. A second reason for controlling degrees of
subsets is that only by doing so can we restrain the degrees of vertices at
later stages: this comes out in the proof of Lemma~\ref{lem:maxdeg}.

So we proceed in the following way. We begin with $P_r=G$, and then apply
the algorithm below to construct $P_s$ from $P_{s+1}$ using $T_s$, with $s$
taking the values $r-1,r-2,\ldots,1$ in turn.  During the application of
the algorithm, the degrees $d_s(\sigma)$ in $P_s$ will grow, as edges are
added. We denote by $\Gamma_s$ the collection of vertices and subsets whose
degrees have reached their bound, and we do not permit the addition to
$P_s$ of any edge which contains a current member of $\Gamma_s$. The set
$\Gamma_s$ will grow too during the construction. We remark that, as
defined in~\S\ref{subsec:egr2}, $\Gamma_1$ was the the set of vertices in
$P_1$ that have positive degree: in the general algorithm we have the
option of specifying a larger threshold for entry into~$\Gamma_1$.

Two real numbers are included in the input to the algorithm. The parameter
$\tau$ is the more important and has already been discussed. The parameter
$\zeta$ is a small constant, used in the rule to decide membership
of~$T_s$.

As in~\S\ref{subsec:egr2}, the membership rule involves a collection $F$ of
edges in $P_s$ with first vertex~$v$, the rule being: 
$d(v)\ge \zeta d$ and $|F|\ge\zeta \tau^{r-s-1}d(v)$. 
Further, if $v\in T_s$, then $F$ is added
to $E(P_s)$, and the set $\Gamma_s$ is updated appropriately. In the
general algorithm $F$ is a multiset rather than a set, to maintain the
condition that edges of $P_s$ correspond to different edges of~$G$.

Again as in~\S\ref{subsec:egr2}, the independence of the set $I$ is not
actually used by the algorithm, and it is useful to define the algorithm
for general subsets $I\subset[n]$.

\begin{tabbing}
\quad\={\sc Algorithm}\\
\>\quad\={\sc input} \quad\=an $r$-graph $G$ on vertex set~$[n]$\\
\>\>\>an $(s+1)$-multigraph $P_{s+1}$ on vertex set~$[n]$\\
\>\>\>parameters $\tau,\zeta >0$\\
\>\>\>{\em in prune mode} \= a subset
$I\subset [n]$\\ 
\>\>\>{\em in build mode}  \> a subset
$T_s\subset[n]$\\
\\
\>\>{\sc output} \>an $s$-multigraph $P_s$ on vertex set~$[n]$\\
\>\>\>{\em in prune mode}  \> a subset
$T_s\subset [n]$\\
\>\>\>{\em in build mode}  \> a subset
$C_s\subset [n]$\\
\\
\>\>put $E(P_s)=\emptyset$ and $\Gamma_s=\emptyset$\\
\>\>{\em in prune mode} \= put $T_s=\emptyset$\\
\>\>{\em in build mode} \> put $C_s=[n]$\\
\\
\>\>for $v=1,2,\ldots,n$ do:\\
\>\>\quad\=let $F=\{f\in[v+1,n]^{(s)}: \{v\}\cup f\in E(P_{s+1}),\,
\mbox{ and }\forall \sigma\in \Gamma_s\,\, \sigma\not\subset f\,\}$\\
\>\>\>\quad[{\em here $F$ is a multiset with multiplicities inherited from 
$E(P_{s+1})$}]\\
\>\>\>{\em in prune mode} \quad\= if $d(v)\ge\zeta d$ and $|F|\ge\zeta
\tau^{r-s-1}d(v)$ and $v\in I$, add $v$ to $T_s$ \\
\>\>\>{\em in build mode} \> if $d(v)\ge\zeta d$ and $|F|\ge\zeta
\tau^{r-s-1}d(v)$, remove $v$ from $C_s$\\
\>\>\>if $v\in T_{s}$ then\\
\>\>\>\quad\= add $F$ to $E(P_s)$\\
\>\>\>\>for each $u\in[v+1,n]$,\qquad\= if $d_s(u)> \tau^{r-s}d(u)$, add
$\{u\}$ to $\Gamma_s$\\
\>\>\>\>for each $\sigma\in[v+1,n]^{(>1)}$, \>if $d_s(\sigma)> 2^s\tau
d_{s+1}(\sigma)$, add $\sigma$ to $\Gamma_s$
\end{tabbing}

The algorithm therefore adds to $P_s$ $s$-edges which, with $v\in T_s$
as first vertex, form an edge of $P_{s+1}$ and which do not contain (at
that moment) any subset in~$\Gamma_s$. The degree threshold for a vertex
entering $\Gamma_s$ is in terms of its degree $d(u)$ in the original graph
$G$, whereas for a larger subset $\sigma$ it is in terms of its degree in
$P_{s+1}$; this difference is for technical reasons arising in the proof of
Lemma~\ref{lem:mucont}. See~\S\ref{sec:postscript} for further comment.

The basic feature of prune/build modes using a membership rule is now
invoked: if $T_s$ is constructed from $I$ by running the algorithm in prune
mode, and then $C_s$ is constructed from $T_s$ by re-running the algorithm
in build mode, then $I\subset C_s\cup T_s$ holds. Another
option for a container arises if $I$ is independent. We noted earlier that,
in this case, if $\{u_0\}\in E(P_1)$ then $u_0\notin I$. In particular, if
$I$ is independent and $\{u_0\}\in\Gamma_1$ then $u_0\notin I$, because
$u_0$ is a vertex whose degree in~$P_1$ has risen above some (non-negative)
threshold so $\{u_0\}\in E(P_1)$. Abusing notation slightly, we write
$I\subset[n]-\Gamma_1$.  Therefore each of $C_s\cup T_s$, $1\le s\le r-1$,
and $[n]-\Gamma_1$ is a container for $I$; our aim is to ensure that at
least one of these is a good container, meaning that its size is not close
to~$[n]$.

Here then is a way of viewing the operation of the algorithm. If
$\Gamma_1$ is large then $[n]-\Gamma_1$ is a good container for~$I$. If
$\Gamma_1$ is not large then, since the degrees in $P_1$ are bounded, the
average degree of $P_1$ must be small. But $P_r=G$, whose average degree is
not small, so there must be some $s$ for which $P_{s+1}$ has large average
degree (of order $\tau^{r-s-1}d$) but $P_s$ has small average degree (much
smaller than $\tau^{r-s}d$). Since the degrees are bounded, there
must have been plenty of vertices of $P_{s+1}$ which could have contributed
edges to $E(P_s)$ but did not do so. Why did they not do so? Only because
they are not in $I$ and so not available for~$T_s$. These are exactly the
vertices which are removed from $C_s$: hence for this value of $s$,
$C_s\cup T_s$ will be a good container for~$I$.

We add an observation here about simple graphs~$G$, for use when
considering the online property (Lemma~\ref{lem:online}). The set
$\Gamma_s$ is non-uniform in order to handle sets $\sigma$ for which
$d_s(\sigma)$ has become too large. If $G$ is simple it is unnecessary to
cater for such a possibility. If $|\sigma|\ge2$ then $\sigma$ can appear in
at most one edge of $G$, so $d_s(\sigma)$ will be zero or one. Formally,
for the way the algorithm is stated, $\sigma$ will quite likely be inserted
into $\Gamma_s$ as soon as it appears in some edge of~$F$ (because
$2^s\tau^{r-s}$ is likely to be less than one). However this has no effect
on the subsequent construction of $P_s$ because $\sigma$ will never again
appear in~$F$. Hence, for simple graphs, the appearance in the algorithm of
$\sigma$ with $|\sigma|\ge2$ can be ignored.

\subsection{Properties of the construction}

We are thus led to two important definitions.

\begin{defn}\label{def:ttuple}
  Let $G$ be an $r$-graph on vertex set $[n]$ and let $I\subset[n]$. Let
  $\tau,\zeta >0$. Let $T_{r-1},\ldots,T_1$ be the sets constructed by
  repeated applications of the algorithm in prune mode. Let $B=\{v\in[n]:
  d(v)<\zeta d\}$. Let $T_0=I\cap(\Gamma_1\setminus B)$. Then we define
  $$T(G,I,\tau,\zeta)\,=\,(T_{r-1},\ldots,T_1,T_0)\,\in\,\mathcal{P}(I)^r\,.
  $$
\end{defn}

The $r$-tuple $T$ is the fruit of running the algorithm in prune mode, from
which the container for $I$ will be built. As noted earlier, if $I$ is an
independent set then $I\cap \Gamma_1=\emptyset$, so $[n]-\Gamma_1$ is a
container for~$I$. We shall in fact use the slightly larger container
$[n]-(\Gamma_1\setminus B)$, but the difference is negligible because
$\mu(B)<\zeta$. The introduction of $B$ ensures that, just as the vertices
in $T_{r-1},\ldots,T_1$ have degree at least $\zeta d$, so do those
in~$T_0$. This will be needed to prove Theorem~\ref{thm:cover}(c).

Now comes the main definition --- that of containers.

\begin{defn}\label{def:container}
  Let $G$ be an $r$-graph on vertex set $[n]$ and let
  $T=(T_{r-1},\ldots,T_1,T_0)\in\mathcal{P}([n])^r$. Let
  $\tau,\zeta>0$. Let $C_{r-1},\ldots,C_1$ be constructed by repeated
  applications of the algorithm in build mode, using
  $T_{r-1},\ldots,T_1$. Let $B=\{v\in[n]: d(v)<\zeta d\}$. Let
  $C_0=[n]-(\Gamma_1\setminus B)$.  The {\em container} $C(G,T,\tau,\zeta)$
  is then
  $$
  C(G,T,\tau,\zeta)\,=\,(C_{r-1}\cap C_{r-2}\cap\cdots\cap C_1\cap C_0)
  \cup T_{r-1}\cup T_{r-2}\cdots\cup T_1\cup T_0\,.
  $$
\end{defn}

\begin{lem}\label{lem:container}
If $T=T(G,I,\tau,\zeta)$ then $I \subset C(G,T,\tau,\zeta)$.
\end{lem}
\begin{proof} We noted earlier that $I\subset C_s\cup T_s$ for
  $s>0$. Moreover, $I\subset C_0\cup T_0$ by definition, since
  $C_0=[n]-(\Gamma_1\setminus B)$ and $T_0=I\cap(\Gamma_1\setminus
  B)$. Hence $I \subset C(G,T,\tau,\zeta)$.
\end{proof}

Before computing the size of the containers $C(G,T,\tau,\zeta)$ and the
number of them, we check the online property, namely that
$C(G,T,\tau,\zeta)\cap[w]$ is determined just by $T\cap[w]$.
Recall that we are asserting the online property only for simple graphs.

\begin{lem}\label{lem:online}
  Let $G$ be a simple $r$-graph on vertex set $[n]$ and let
  $T\in\mathcal{P}([n])^r$. Then, for each $w\in
  [n]$, $C(G,T,\tau,\zeta)\cap [w]=C(G,T\cap[w],\tau,\zeta)\cap[w]$ holds.
\end{lem}

\begin{proof} The tuple $T$ supplies sets $T_{r-1},\ldots,T_1$ as inputs
  for the algorithm in build mode, which produces sets $C_{r-1},
  C_{r-2},\ldots, C_1$ and $\Gamma_1$. The set $T_0$ is supplied directly
  by $T$, and we take $C_0=[n]-(\Gamma_1\setminus B)$. Let $T'=T\cap[w]$,
  where $T'=(T'_{r-1},\ldots,T'_1,T'_0)$ and $T_s'=T_s\cap[w]$, $0\le s\le
  r-1$. Let $C_{r-1}',C_{r-2}',\ldots, C_1',C_0'$ and $\Gamma_1'$ be the
  corresponding sets produced when the inputs to the algorithm are
  $T'_{r-1},\ldots,T'_1,T'_0$. We need to show that
  $C_s'\cap[w]=C_s\cap[w]$ for all~$s$; Definition~\ref{def:container} then
  shows $C(G,T,\tau,\zeta)\cap [w]=C(G,T\cap[w],\tau,\zeta)\cap[w]$.

  Let $P_r,\ldots,P_1$ be the multigraphs used during the runs of the
  algorithm with the original inputs, and $P_r',\ldots,P'_1$ those used
  with the truncated inputs. The crucial point is that, though $P_s$ and
  $P'_s$ might have different edges inside the vertex set~$[w+1,n]$, they
  are otherwise identical; that is, if $e$ is an $s$-edge with $e\cap
  [w]\ne\emptyset$, then $e\in E(P_s)$ if and only if $e\in E(P'_s)$ (with
  the same multiplicity). This can be seen for $s=r,r-1,\ldots,1$ in turn,
  given that $P_r=P_r'=G$. Consider the run of the algorithm building
  $P_s$, as $v$ runs through $v=1,\ldots,w$. As mentioned at the end
  of~\S\ref{subsec:algorithm}, subsets $\sigma$ with $|\sigma|\ge2$ have no
  effect because $G$ is simple. By induction on~$v$, the singletons $\{u\}$
  in $\Gamma_s$ are the same as those in $\Gamma'_s$ while $v\in[w]$, and
  the set $F$ is the same for $P_s$ as for $P'_s$, because $F$ is defined
  by edges in $P_{s+1}$ whose first vertex lies in~$[w]$. Hence the
  membership rule, the set of singletons $\{u\}$ added to $\Gamma_s$, and
  the set of edges $F$ added to $E(P_s)$,
  are the same in $P'_s$ and $P_s$.  In particular,
  $C_s'\cap[w]=C_s\cap[w]$ for $s\ge1$. For the same reasons,
  $\Gamma_1'\cap[w]=\Gamma_1\cap[w]$. The set $B$ is defined in terms of
  $G$ itself, and so $C_0'\cap [w]=[w]-(\Gamma_1'\setminus
  B)=[w]-(\Gamma_1\setminus B)=C_0\cap[w]$.
\end{proof}

\section{Container calculations}\label{sec:calc}

In this section we estimate the measure of the tuples $T(G,I,\tau,\zeta)$
and of the containers $C(G,T,\tau,\zeta)$, thereby proving
Theorem~\ref{thm:cover}.

\subsection{Degrees and co-degrees}\label{subsec:maxdeg}

Before making these estimates we need information on how large the degrees
can be in $P_s$. The intention behind the set $\Gamma_s$ is to prevent
degrees being much larger than the target degrees, namely $\tau^{r-s}d(u)$
for the vertex~$u$; after the degree of~$u$ attains this level, no
further edges containing~$u$ are added to $P_s$. However,
when a vertex $u$ enters $\Gamma_s$, it
does so because some multiset $F$ has been added to $E(P_s)$. Since $F$ can
include many edges that contain~$u$, the degree $d_s(u)$ can increase
significantly in one step, from an initial value at most the target
value~$\tau^{r-s}d(u)$ to something much larger. The extent of this problem
depends ultimately on the way the edges of $G$ overlap each other.
The reason $\Gamma_s$ is defined the way it is in the algorithm, is to keep
control of the degree problem without increasing~$\tau$ more than is
necessary. This definition lies at the heart of the efficiency of the
algorithm. Control of the degrees can be expressed succinctly in terms of
the co-degree function $\delta(G,\tau)$.

First we need a small calculation.

\begin{lem}\label{lem:asj}
  For $2\le s\le r$ and $2\le j\le s$, let $a_s^{(j)}$ be given by the
  equations $a_r^{(j)}=\delta_j$ and $a_s^{(j)}=2^sa_{s+1}^{(j)} +
  a_{s+1}^{(j+1)}$ for $s<r$, where $\delta_j$ was defined in
  Definition~\ref{defn:delta}. Then $a_s^{(2)} \le
  4^{2-s}\delta(G,\tau)$ holds for $s\ge2$.
\end{lem}
\begin{proof}
  Since $a_s^{(2)}\ge 2^s a_{s+1}^{(2)} \ge 4a_{s+1}^{(2)}$, it is enough
  to prove that $a_2^{(2)}\le \delta(G,\tau)$. Now by dint of the
  definition it is clear that $a_s^{(j)}$ is a linear combination of the
  numbers $\delta_{j+\ell}$, $\ell\ge0$. We claim that the coefficient of
  $\delta_{j+\ell}$ in $a_s^{(j)}$ is at most
  $2^{\binom{r}{2}-\binom{s+\ell}{2}+\ell}$. This is certainly true if
  $s=r$, since the only positive coefficient is that of $\delta_j$
  (i.e.~$\ell=0$). For $s<r$ we may prove the claim on the assumption that
  it is true for $s+1$. If $\ell=0$ then the coefficient of
  $\delta_{j+\ell}$ in $a_{s+1}^{(j+1)}$ is zero, and the claim follows
  because
  $2^{\binom{r}{2}-\binom{s}{2}}=2^s2^{\binom{r}{2}
    -\binom{s+1}{2}}$.
  If $\ell\ge1$ we have
$$
2^s\,
2^{\binom{r}{2}-\binom{s+1+\ell}{2}+\ell} +
2^{\binom{r}{2}-\binom{s+\ell}{2}+\ell-1}
= 2^{\binom{r}{2}-\binom{s+\ell}{2}+\ell}\left[2^{-\ell}+2^{-1}\right]\le
2^{\binom{r}{2}-\binom{s+\ell}{2}+\ell}
$$
and the claim follows in this case too. Hence the claim always holds, and 
so
$$
a_2^{(2)}\,\le\,
2^{\binom{r}{2}}\sum_{\ell=0}^{r-2}2^{-\binom{\ell+2}{2}+\ell}
  \delta_{2+\ell}\,=\,2^{\binom{r}{2}-1}\sum_{j=2}^r
  2^{-\binom{j-1}{2}}\delta_j
\,=\,\delta(G,\tau)\,,
$$
by definition of $\delta(G,\tau)$.
\end{proof}

Here is the main lemma about degrees in~$P_s$, and it shows the role of the
co-degree function $\delta(G,\tau)$ in the analysis of the algorithm. As
explained in~\S\ref{sec:online}, we would ideally like $\sum_{u\in
  U}d_s(u)\approx \tau^{r-s}\mu(U)nd$ for each subset $U\subset[n]$. The
lemma shows that this holds as an upper bound, with a small error expressed
in terms of $\delta(G,\tau)$.

\begin{lem}\label{lem:maxdeg}
  Let $G$ be an $r$-graph on vertex set $[n]$ with average degree~$d$. Let
  $P_r=G$ and let $P_{r-1},\ldots,P_1$ be the multigraphs constructed by
  some run of the algorithm, either in build mode or in prune mode. Then
$$
\sum_{u\in U}d_s(u)\,\le\,(\mu(U)+4^{1-s}\delta(G,\tau))\, \tau^{r-s}\,nd
$$
holds for all subsets $U\subset[n]$ and for $1\le s\le r$.
\end{lem}
\begin{proof} Recall that, as the algorithm proceeds, an element enters the
  set $\Gamma_s$ when its degree exceeds some threshold: for a vertex, when
  $d_s(u)> \tau^{r-s}d(u)$, and for a larger set when $d_s(\sigma)>
  2^s\tau d_{s+1}(\sigma)$. Let $u\in U$. If $u\notin\Gamma_s$ then $d_s(u)\le
  \tau^{r-s}d(u)$. If $u\in\Gamma_s$ then $u$ was added to $\Gamma_s$ after
  some other vertex $v$ was inspected and some multiset $F$ was added to
  $E(P_s)$, raising $d_s(u)$ beyond $\tau^{r-s}d(u)$. After $u$ was added
  to $\Gamma_s$, $d_s(u)$ did not change. Hence $\sum_{u\in
    U}d_s(u)\le\tau^{r-s}\sum_{u\in U}d(u)$ plus the extra contribution from the
  multisets~$F$. It is these extra contributions that we must now examine
  and bound in
  terms of $\delta(G,\tau)$. To do this, we must consider all the elements
  $\sigma\in\Gamma_s$, not just the vertices. Each of these enters
  $\Gamma_s$ when its degree exceeds its threshold by a little extra. These
  extras percolate down to form the extra for the vertex~$u$, in a way that
  Lemma~\ref{lem:asj} is designed to capture.

Let us do the calculation. By analogy with Definition~\ref{defn:delta} we define
$$
d_s^{(j)}(u)\,=\,\max\,\{\,d_s(\sigma)\,:\,u\in\sigma\in[n]^{(j)}\,\}\,,
$$
for $j\ge2$, where here it is the final values of these quantities that are
used --- that is, we measure these quantities in the output
multigraph~$P_s$.  

When $s=r$ the lemma is true by definition of $\mu(U)$, so from now on we
assume $s\le r-1$.
Let $u\in[n]$; then $d_s^{(j)}(u)=d_s(\sigma)$ for some
$\sigma\in[n]^{(j)}$ with $u\in\sigma$. If
$\sigma\notin\Gamma_s$ then $d_s(\sigma)\le 2^s\tau d_{s+1}(\sigma)$. If
$\sigma\in\Gamma_s$ then $\sigma$ was added to $\Gamma_s$ after some vertex
$v\notin\sigma$ was inspected and $F$ was added to $E(P_s)$. Before this took
place, $d_s(\sigma)\le2^s\tau d_{s+1}(\sigma)$ held; since the number of
edges of $F$ containing $\sigma$ was at most $d_{s+1}(\sigma\cup\{v\})$, we
have, in both cases,
\begin{equation}
d_s^{(j)}(u) = d_s(\sigma) \le2^s\tau d_{s+1}(\sigma) +
d_{s+1}(\sigma\cup\{v\})\le 2^s \tau d_{s+1}^{(j)}(u)+d_{s+1}^{(j+1)}(u)\,.
\label{eqn:dsj}
\end{equation}
We claim that
$$
\sum_{u\in[n]} d_s^{(j)}(u)\le a_s^{(j)}\tau^{r-s+j-1}nd\,,
$$
where $a_s^{(j)}$ was defined in Lemma~\ref{lem:asj}. Indeed, for $s=r$
the claim (with equality) is just the definition of $\delta_j$, and for
$s\le r-1$ it follows immediately by induction (on $r-s$) from
inequality~(\ref{eqn:dsj}) and the definition of $a_s^{(j)}$. Hence, for
$s\ge1$, we have by Lemma~\ref{lem:asj}
\begin{equation}
\sum_{u\in[n]} d_{s+1}^{(2)}(u)\le
4^{1-s}\tau^{r-s}nd\,\delta(G,\tau)\,.
\label{eqn:ds2}
\end{equation}

Now let $u\in U$. As mentioned before, either $d_s(u)\le \tau^{r-s}d(u)$
or $u$ was added to $\Gamma_s$ after some vertex $v$ was inspected
and $F$ was added to $E(P_s)$. Since $F$ has at most $d_{s+1}(\{u,v\})$
edges containing~$u$, the degree of $u$ in $P_s$ is at most
$\tau^{r-s}d(u)+d_{s+1}(\{u,v\})$. Now $d_{s+1}(\{u,v\})\le
d_{s+1}^{(2)}(u)$ so, using~(\ref{eqn:ds2}), we have
$$
\sum_{u\in U}d_s(u)\,\le\, \sum_{u\in
  U}\left(\tau^{r-s}d(u)+d_{s+1}^{(2)}(u)\right)
\,\le\,
\tau^{r-s}\mu(U)nd+4^{1-s}\tau^{r-s}nd\,\delta(G,\tau)\,,
$$
which establishes the lemma.
\end{proof}

\subsection{The measure of the sets $T_s$}\label{subsec:muT}

We now estimate the measures of the sets $T_s$. Ideally, they would have
degree measure at most~$\tau$, or, more exactly, $\tau/\zeta$. In fact such
a bound does hold with a small error determined by $\delta(G,\tau)$.

\begin{lem}\label{lem:muTs}
Let $I\subset[n]$ and $T=T(G,I,\tau,\zeta)=(T_{r-1},\ldots,T_1,T_0)$. Then
$\mu(T_s)\le (\tau/\zeta)(1+\delta(G,\tau))$ for $1\le s\le r-1$.
\end{lem}
\begin{proof}
  The set $T_s$ is output when the algorithm is run in prune mode. During
  the run of the algorithm, each vertex $v$ which enters $T_s$ contributes
  a set $F$ of at least $\zeta \tau^{r-s-1}d(v)$ edges to
  $E(P_s)$. But the total size of $E(P_s)$ is limited, because the degrees
  in $P_s$ are constrained. Writing $d$ for the average degree of~$G$, 
  Lemma~\ref{lem:maxdeg} yields
\begin{align*}
 \zeta\tau^{r-s-1}\mu(T_s)nd=\sum_{v\in T_s}\zeta\tau^{r-s-1}d(v)
\le e(P_s)\le\sum_{u\in[n]}d_s(u)\\\le
\tau^{r-s}nd(1+4^{1-s}\delta(G,\tau))
\end{align*}
and this proves the lemma.
\end{proof}

The set $T_0$ needs a different argument. As noted before, $T_0=\emptyset$
if $I$ is independent. 

\begin{lem}\label{lem:muT0}
  Let $G$ be an $r$-graph on vertex set $[n]$ with average
  degree~$d$. Let $I\subset[n]$ and
  $T=T(G,I,\tau,\zeta)=(T_{r-1},\ldots,T_1,T_0)$. If $e(G[I])\le bn$ where 
  $b\le2\tau^rd/ \zeta$, then $\mu(T_0)\le 2\tau/\zeta$. If $G[I]$
  is $b$-degenerate where $b\le \zeta\tau^{r-1}d /r $,  then $\mu(T_0)\le
  (\tau/\zeta)(1+\delta(G,\tau))$.
\end{lem}
\begin{proof}
  Recall that $T_0=I\cap (\Gamma_1\setminus B)$. So, for each $v\in T_0$,
  $d_1(v)>\tau^{r-1}d(v)$ holds because $v\in\Gamma_1$. Here the degree
  $d_1(v)$ is in the multigraph $P_1$. Let $J=T_{r-1}\cup\cdots\cup T_1\cup
  T_0\subset I$. Recall that distinct $1$-edges $\{v\}$ in the multigraph
  $P_1$ correspond to distinct $r$-edges $\{v_{r-1}, \ldots, v_1, v \}
  \subset I$ with $v_{r-1}<\cdots< v_1<v$ and $v_s \in T_s$, so these edges
  lie in~$G[J]$. It follows that $\tau^{r-1}\mu(T_0)nd\le\sum_{v\in T_0}
  d_1(v)\le e(G[J])$.

  If $e(G[I])\le bn$ then $\tau^{r-1}\mu(T_0)nd\le e(G[J])
  \le e(G[I])\le bn\le 2\tau^rdn/\zeta $, so $\mu(T_0)\le 2\tau/\zeta$.

  If $G[I]$ is $b$-degenerate, then $e(G[J])\le b|J|$. Recall from the
  construction of the sets $T_s$, $s\ge 1$, in prune mode that $d(v)\ge
  \zeta d$ for $v\in T_s$. For $v\in T_0$ we have $d(v)\ge
  \zeta d$ because $v\notin B$. Thus $d(v)\ge\zeta d$ for all $v\in J$. So
  $$
  \tau^{r-1}\mu(T_0)nd\le e(G[J])\le b|J|\,\le\,\frac{b}{\zeta d}\sum_{v\in
    J}d(v)\,= \,\frac{b}{\zeta d}\mu(J)nd\,\le \,\frac{\tau^{r-1}}{r}\mu(J)nd
  \,.
  $$
  Therefore $r\mu(T_0)\le\mu(J)\le\mu(T_{r-1})+\cdots+\mu(T_0)$,
  and so $(r-1)\mu(T_0)\le \mu(T_{r-1})+\cdots+\mu(T_1) \le
  (r-1)(\tau/\zeta)(1+\delta(G,\tau))$ by
  Lemma~\ref{lem:muTs}.
\end{proof}

\subsection{The measure of the container
  $C(G,T,\tau,\zeta)$}\label{subsec:container}

We now prove the crucial fact that the measure of the container
$C(G,T,\tau,\zeta)$ is bounded above by some constant less than one. This
can be established with a fairly simple argument, but just a
little more care yields a bound close to $1-1/r!$, which is best
possible, as described in~\S\ref{subsec:optimality}.

It is in order to achieve this bound that the number $2^s$
appears in the algorithm, in the condition $d_s(\sigma)>2^s\tau
d_{s+1}(\sigma)$ for entry of $\sigma$ into~$\Gamma_s$. Hence it is that
powers of~2 appear in the definition of $\delta(G,\tau)$, having
permeated there via Lemma~\ref{lem:maxdeg}. The condition can be relaxed
to $d_s(\sigma)>k\tau d_{s+1}(\sigma)$ for some smaller value of $k$, 
with some slight reduction in the constants in the definition of
$\delta(G,\tau)$, but at the expense of a weaker bound on
$\mu(C(G,T,\tau,\zeta))$. See~\S\ref{sec:postscript} for further comment.

\begin{lem}\label{lem:mucont} Let
  $T=(T_{r-1},\ldots,T_0)\in\mathcal{P}([n])^r$. Then 
$$
\mu(C(G,T,\tau,\zeta))\,\le\, 1 - \frac{1}{r!} + \frac{15}{4}\zeta +
\frac{1}{4}\delta(G,\tau)+ \sum_{s=0}^{r-1}\mu(T_s)\,.
$$
\end{lem}
\begin{proof}
  Recall from Definition~\ref{def:container} that
  $C_0=[n]-(\Gamma_1\setminus B)$, where $B=\{v\in[n]:d(v)<\zeta d\}$ and
  $d$ is the average degree of~$G$. Recall too that $C_{r-1},\ldots,C_1$
  are constructed by the algorithm in build mode. Let
  $C=C_{r-1}\cap\cdots\cap C_0$.  We define $D_1= [n]- (C\setminus B)$. Now
  $\mu(C)\le 1-\mu(D_1)+\mu(B)$ and $\mu(B)<\zeta$. By
  Definition~\ref{def:container}, $C(G,T,\tau,\zeta)=C\cup
  T_{r-1}\cup\cdots\cup T_0$. So to prove the lemma it is enough to prove that
  $\mu(D_1)\ge 1/r!-11\zeta/4-\delta(G,\tau)/4$.
  
  We first outline the argument, before filling in the details.
  We have $C_0=[n]-(\Gamma_1\setminus B)$, so
  $(C_0\setminus B)\cap \Gamma_1=\emptyset$. Since $C\subset C_0$, this
  means $(C\setminus B)\cap \Gamma_1=\emptyset$, so $\Gamma_1\subset
  [n]-(C\setminus B)=D_1$. We extend $D_1$ to a partition $D_1,\ldots,D_r$
  of~$[n]$ as follows:
\begin{align*}
  D_1 &= [n] \setminus  (C\setminus B)\qquad\mbox{where $\Gamma_1\subset
    D_1$ and $B\subset D_1$}\\
  D_2 &= \{v\in[n]: \{v\}\in \Gamma_2, v\notin D_1 \}\\
  D_3 &= \{v\in[n]: \{v\}\in \Gamma_3, v\notin (D_1 \cup D_2)\} \\
    &\ \vdots \\
  D_{r-1} &= \{v\in[n]: \{v\}\in \Gamma_{r-1}, v\notin  (D_1 \cup \cdots \cup
    D_{r-2})\} \\
  D_r &= [n] \setminus (D_1 \cup \cdots \cup D_{r-1})\,.
\end{align*}
We aim to bound $\mu(D_s)$ above, for $s\ge2$, in terms of
$\mu(D_{s-1}),\ldots,\mu(D_1)$. By induction, this means $\mu(D_s)$ is
bounded above in terms of $\mu(D_1)$, and, since $\mu([n])=1$ and
$D_1,\ldots,D_r$ partition~$[n]$, we obtain a lower bound on $\mu(D_1)$ as
desired. It is convenient to define $D_{<s}=D_1\cup\cdots\cup D_{s-1}$,
$D_{\le s}=D_s\cup D_{<s}$ and so on. Notice that
$C\setminus B=D_{\ge2}$.

To find the desired upper bound for $\mu(D_s)$ in terms of $\mu(D_{<s})$,
we look at edges of $P_s$. The point of the definition of $D_s$ is that if
$v\in D_s$ then $\{v\}\in\Gamma_s$ so $d_s(v)\ge \tau^{r-s}d(v)$; thus the
number edges of $P_s$ meeting $D_s$ can be bounded below in terms of
$\mu(D_s)$. 
As we shall explain, we expect very few edges of $P_s$ to lie inside $D_{\ge
  s}$, so nearly all edges
meeting $D_s$ meet $D_{<s}$ too, and since the number of edges meeting $D_{<s}$
is bounded above in terms of $\mu(D_{<s})$ by Lemma~\ref{lem:maxdeg}, we
are done. So the fundamental point of the proof is that $D_{\ge s}$ should
contain few edges of $P_s$. 

Define the following trio of subsets of the edges of~$P_s$, for each
$s\ge2$:
\begin{align*}
  X_s &= \{\,f\in E(P_s)\,:\, |f\cap D_{<s}|\ge1,\,\, |f\cap D_{>s}|\ge 2\,\}\\
  Y_s &= \{\,f\in E(P_s)\,:\, f\subset D_{\ge s}\,\}\\
  Z_s &= \{\,f\in Y_s\,:\,\sigma\subset f\mbox{ for some } \sigma\in
  \Gamma_{s-1}, |\sigma|\ge2\,\}\,. 
\end{align*}
Here, $Y_s$ is the set previously discussed of edges inside $D_{\ge   s}$;
if $Y_s$ is empty, or
small, then the above sketch proof works. We come to $X_s$ and $Z_s$
shortly.

Suppose first that $G$ is a simple $r$-graph. As noted
in~\S\ref{subsec:algorithm}, the sets $\sigma$ with $|\sigma|\ge2$ play no
role, and can be deleted from the algorithm; in particular $Z_s=\emptyset$.
Consider an edge in~$Y_s$. It has a first vertex~$v$, where $v\in D_{\ge
  s}$. Now $D_{\ge2}=C\setminus B\subset C=C_{r-1}\cap\cdots\cap C_0$, so in
particular $v\in C_{s-1}\setminus B$. By definition of $C_{s-1}\setminus
B$, all but $\zeta\tau^{r-s}d(v)$ edges of $P_s$ with first vertex~$v$ meet
$\Gamma_{s-1}$ and so meet $D_{<s}$, by definition of the~$D_i$ (note this
is true even if $s=2$). This means there are at most $\zeta\tau^{r-s}d(v)$
edges in $Y_s$ with first vertex~$v$, so $Y_s$ is indeed relatively tiny,
as desired.

If $G$ is not simple, then the argument of the previous paragraph yields
that $Y_s\setminus Z_s$ is tiny, so we need worry only if $Z_s$ is
large. But this would mean there are many $\sigma\in\Gamma_{s-1}$ with
$d_s(\sigma)$ large, which, by definition of $\Gamma_{s-1}$, implies
$d_{s-1}(\sigma)$ is large. This turns out to give rise to many edges
in $X_{s-1}$. However, we shall see below in~(\ref{eqn:muD}) that, unlike
edges in $Y_s$, edges in $X_s$ only {\em 
  improve} the original estimate for $\mu(D_s)$ in terms of
$\mu(D_{<s})$. So we trade off a loss in $\mu(D_s)$ caused by $Y_s$, that
is, by~$Z_s$, for a gain in $\mu(D_{s-1})$ caused by $X_{s-1}$. The
relative trade-off can be weighted in favour of $X_{s-1}$ by the $2^s$ term in
the definition of $\Gamma_s$ in the algorithm, and this is precisely the
reason for its appearance.

Now we can start the proof. We count edges in $P_s$ but take
into account both $X_s$ and~$Y_s$. We can take $d_s(v)\ge \tau^{r-s}d(v)$
for all $v\in D_s$: for $s<r$ this is because $\{v\}\in \Gamma_s$, and for
$s=r$ it holds trivially. Note that $X_s\cap Y_s=\emptyset$ and each member
of $E(P_s)\setminus Y_s$ meets $D_{<s}$. So, for $s\ge2$,
\begin{align}
\tau^{r-s}\mu(D_s)nd &\le\sum_{v\in D_s}d_s(v) \nonumber\\
&=\sum_{f\in E(P_s)} |f\cap D_s|\nonumber\\
&=\sum_{f\in E(P_s)\setminus(X_s\cup Y_s)} |f\cap D_s|\,+\,
\sum_{f\in X_s} |f\cap D_s|\,+\,\sum_{f\in Y_s} |f\cap D_s|\nonumber\\
&\le (s-1)|E(P_s)\setminus(X_s\cup Y_s)|+(s-3)|X_s|+s|Y_s|\nonumber\\
&= (s-1)|E(P_s)\setminus Y_s|-2|X_s|+s|Y_s|\nonumber\\
&\le (s-1)\sum_{v\in D_{<s}}d_s(v)-2|X_s|+s|Y_s|\nonumber\\
&\le
(s-1)\tau^{r-s}\,nd\,(\,\mu(D_{<s})+4^{1-s}\delta(G,\tau))-2|X_s|+s|Y_s|\,,
\label{eqn:muD}
\end{align}
where the last line employs Lemma~\ref{lem:maxdeg}. This is the bound on
$\mu(D_s)$ that we want.

For convenience, we further define the numbers $x_s,y_s,z_s$ by
$|X_s|=x_s\tau^{r-s}nd$, $|Y_s|=y_s\tau^{r-s}nd$ and
$|Z_s|=z_s\tau^{r-s}nd$. Observe that $X_2=\emptyset$ because edges in
$X_s$ have at least three vertices, and $Z_2=\emptyset$ because $\Gamma_1$
contains no $\sigma$ with $|\sigma|=2$. Thus we have the initial conditions
$x_2=z_2=0$.

Tidying up~(\ref{eqn:muD}), we obtain $\mu(D_s)\le
(s-1)(\mu(D_{<s})+4^{1-s}\delta(G,\tau))-2x_s+sy_s$. Adding
$\mu(D_{<s})=\mu(D_{\le s-1})$ to each side gives
$$
\mu(D_{\le s})\le s\mu(D_{\le s-1}) -2x_s+sy_s+(s-1)4^{1-s}\delta(G,\tau)
$$
for each $s\ge2$. Multiplying this inequality by $1/s!$ and summing over
$s=2,\ldots,r$, noting that $\mu(D_{\le r})=1$, $D_{\le 1}=D_1$ and $x_2=0$,
we obtain
\begin{equation}
\frac{1}{r!}\,\le\, \mu(D_1)-2\sum_{s\ge 3} \frac{x_s}{s!}+\sum_{s\ge
  2}\frac{y_s}{(s-1)!} + \frac{1}{4}\delta(G,\tau)\,,
\label{eqn:xy}
\end{equation}
where we used $\sum_{s\ge 2}4^{1-s}(s-1)/s! < 1/4$.

Let $s\ge2$ and let $f\in Y_s\setminus Z_s$. If $f$ contains a
subset $\sigma\in\Gamma_{s-1}$ then $|\sigma|=1$, say $\sigma=\{u\}$. But
$\{u\}\in \Gamma_{s-1}$ implies $u\in D_{<s}$ by definition of $D_{s-1}$
(even if $s=2$),
which contradicts $f\in Y_s$. Thus $f$ contains no member of
$\Gamma_{s-1}$. Let $v$ be the first vertex of~$f$. Now $v\in D_{\ge
  s}\subset D_{\ge 2}=C\setminus B$, so $v\in C_{s-1}\setminus B$. 
By the construction of~$C_{s-1}$,
$v$ is the first vertex of fewer than $\zeta\tau^{r-s}d(v)$ edges of $P_s$
that contain no member of~$\Gamma_{s-1}$, so it is the first vertex of
fewer than $\zeta\tau^{r-s}d(v)$ edges in $Y_s\setminus Z_s$.  Therefore
$|Y_s|-|Z_s|\le \sum_{v\in D_{\ge s}}\zeta\tau^{r-s}d(v)
=\zeta\tau^{r-s}\mu(D_{\ge s})nd$.  Hence $y_s-z_s \le \zeta \mu(D_{\ge
  s})\le \zeta$. In particular $y_2\le \zeta$, because $z_2=0$.

Let $s\ge3$ and put $S=\{\sigma\in \Gamma_{s-1}: |\sigma|\ge2,
\,\sigma\subset D_{\ge s}\}$.  By definition of $Z_s$, each member of $Z_s$
contains a member of~$S$. Let $F$ be the set of edges of $P_{s-1}$ that
contain a member of~$S$. Then each edge in $F$ contains at least two
vertices of~$D_{\ge s}$; therefore $F\subset X_{s-1}\cup Y_{s-1}$. Hence
\begin{align*}
z_s\tau^{r-s}nd \,= \,|Z_s|&\le \sum_{\sigma\in S} d_s(\sigma)\\
&\le\sum_{\sigma\in S}\frac{1}{\tau 2^{s-1}}d_{s-1}(\sigma)
\quad\mbox{by definition of }\Gamma_{s-1}\\
&= \frac{1}{\tau 2^{s-1}}\sum_{f\in F}
|\{\sigma\in S: \sigma\subset f\}|\\
&\le \frac{1}{\tau}\,|F|\quad\mbox{since $|f|=s-1$ for each $f\in F$}\\
&\le \frac{1}{\tau}\,|X_{s-1}\cup Y_{s-1}|
= (x_{s-1}+y_{s-1})\tau^{r-s}nd\,.
\end{align*}

Thus $z_s\le x_{s-1}+y_{s-1}$ for $s\ge3$. Since $y_s\le z_s+\zeta$ this means
$y_s\le x_{s-1}+y_{s-1}+\zeta$; by repeating and applying both $x_2=0$
and $y_2\le \zeta$, this yields $y_s\le x_{s-1}+x_{s-2}+\cdots+x_3+(s-1)\zeta$ for
$s\ge3$. The inequality holds for $s=2$ also. Substituting this
inequality into inequality~(\ref{eqn:xy}) we obtain
$$
\frac{1}{r!}\,\le\, \mu(D_1) + \sum_{s\ge 3} x_s\left(-\frac{2}{s!}
  +\sum_{j=s}^{r-1}\frac{1}{j!}\right) +
\zeta\sum_{s\ge2}\frac{1}{(s-2)!} +\frac{1}{4} \delta(G,\tau)\,.
$$
The coefficient of $x_s$ is negative, and so $1/r!\le \mu(D_1)+
11\zeta/4+\delta(G,\tau)/4$, which is what we needed to prove.
\end{proof}

\subsection{Proof of Theorem~\ref{thm:cover}}\label{subsec:mainproof}

As expected, our choices for $T$ and $C(T)$ in Theorem~\ref{thm:cover}
will usually be $T=T(G,I,\tau,\zeta)$ and $C(T)=C(G,T,\tau,\zeta)$.

\begin{proof}[Proof of Theorem~\ref{thm:cover}]
  Notice that the theorem is trivial if $\zeta\ge1/4r!$, since in that case
  the function $C(T)=[n]$ works, with $T=(\emptyset,\ldots,\emptyset)$
  representing all~$I$. Recall too from~\S\ref{subsec:codeg} that
  $\delta(G,\tau)\to0$ as $\tau\to\infty$. Hence the condition
  $\delta(G,\tau)\le\zeta$ is satisfiable by making $\tau$ large enough,
  although if $\tau\ge\zeta/2r$ the theorem is similarly trivial.

  In the remaining cases we take $T=T(G,I,\tau,\zeta)$ and
  $C(T)=C(G,T,\tau,\zeta)$. Then assertion~(a) of the theorem holds because
  of Lemma~\ref{lem:container}, and the online property holds for simple
  graphs because of Lemma~\ref{lem:online}.

  Let $d$ be the average degree of~$G$. By Lemmas~\ref{lem:muTs}
  and~\ref{lem:muT0}, assertion~(b) holds for sets
  $I$ for which $G[I]$ is $\lfloor\tau^{r-1} \zeta
  e(G)/n\rfloor$-degenerate, that is, $b$-degenerate with $b\le \zeta \tau^{r-1}
   d/r$, using the fact that $\delta(G,\tau)\le\zeta\le1$. Likewise, 
  (b)~holds for sets~$I$ for which $e(G[I])\le 2r \tau^re(G)/
   \zeta$, that is, $e(G[I])\le bn$ with $b\le2\tau^rd/ \zeta$. Either of
   these implies~(b) for independent sets~$I$, by taking $b=0$.

  For every $v\in T_s$ we have $d(v)\ge \zeta d$: for $s\ge 1$ this holds
  by the definition of the algorithm, and for $s=0$ it holds by the
  definition of $T_0$. Hence $|T_s|\le (1/ \zeta d)\sum_{v\in T_s}d(v)
  = (n/\zeta)\mu(T_s)$. Thus (c)~follows from~(b).
  
  Finally, property~(d) follows from Lemma~\ref{lem:mucont} and
  assertion~(b), so we are done.
\end{proof}

\section{Tight containers}\label{sec:iteration}

We turn now to the first of our packaged versions of the container theorem,
Corollary~\ref{cor:sparse_container}, which supplies containers with
$e(G[C])$ small. The way to obtain sparser containers by repeated
applications of the container theorem was discussed
in~\S\ref{subsec:tight}, and here we calculate what is achievable. Given an
independent set $I$ in the $r$-graph~$G$, we apply the container theorem to
obtain a container $C$ with $I\subset C$. We then apply the container
theorem again, this time to $G[C]$, to obtain a sparser container~$C'$,
then apply the theorem to $G[C']$, and so on, until the container is as
sparse as we need, or the average degree in the container is so small that
a further application of the container theorem yields no information.

The only point that needs consideration is how much effort we are willing,
or able, to put into the calculation of the codegree function
$\delta(G[C],\tau)$ at each stage. This function determines how small
$\tau$ can be and hence how many (or few) containers are built.  Evidently
the degree $d(\sigma)$ of some set $\sigma\subset V(G)$ is no larger in
$G[C]$ than it is in~$G$, so the simplest approach is 
just to use the original values to obtain an upper bound for
$\delta(G[C],\tau)$. This works
well for a limited number of iterations and it is the basis of the proof of
Corollary~\ref{cor:sparse_container}.

However there are applications where the number of iterations is large ---
growing with~$n$ (examples are Corollary~\ref{cor:C4} and
Theorem~\ref{thm:sidon}), and where care is needed in keeping track of the
codegree function. In such circumstances, Theorem~\ref{thm:iteration} can
be used; Corollary~\ref{cor:sparse_container} is then a special case of
this theorem.

We begin with a simple lemma to help count the number of containers
being generated by iteration.

\begin{lem}\label{lem:entropy}
There are at most $\exp\{s\theta n(1+\log(1/\theta))\}$ $s$-tuples
of subsets $T_1,\ldots,T_s\subset[n]$ with $|T_1|+\ldots+|T_s|\le s\theta
n$, where $0\le\theta \le 1$.
\end{lem}
\begin{proof}
  Let there be $N_j$ such $s$-tuples with $|T_1|+\ldots+|T_s|=j$. We wish
  to bound $N=N_0+N_1+\ldots+N_{\lfloor s\theta n\rfloor}$. The generating
  function for the numbers of subsets of $[n]$ of size~$i$ is $(1+x)^n$.
  Hence the coefficient of $x^j$ in $((1+x)^n)^s$ is the number of ways to
  choose sets $T_1,\ldots,T_s$ of sizes $t_1,\ldots,t_s$ such that
  $t_1+\cdots+t_s=j$; in other words
  $N_0+N_1x+N_2x^2+\ldots=((1+x)^n)^s$. Therefore, since $0\le\theta\le1$,
  we have $\theta^{s\theta n}N\le (1+\theta)^{ns}\le e^{s\theta n}$.
\end{proof}

The next theorem is a version of Theorem~\ref{thm:cover} stripped of
references to degree measure. It is this theorem that we shall apply
iteratively.

\begin{thm}\label{thm:coveroff}
  Let $G$ be an $r$-graph on vertex set $[n]$.
  Suppose that $\delta(G,\tau)\le 1/12r!$, where $0<\tau<1/2$. Then
  there exists a collection $\C\subset\mathcal{P}[n]$ such that
 \begin{itemize}
 \item[(a)] for every independent set $I$ there exists
   $T=(T_{r-1},\ldots,T_0) \in \mathcal{P}(I)^r$ with $I\subset
   C(T)\in\C$ and $|T_i| \le 288r!^2\tau n$,
 \item[(b)] $\log|\C|\le 288r!^2rn\tau\log(1/\tau)$, and
 \item[(c)] $e(G[C]) \le (1 - 1/2r!)e(G)$ for all $C\in\C$.
 \end{itemize}

Moreover, (a) holds for all sets $I\subset[n]$ for which either $G[I]$ is
$\lfloor\tau^{r-1}e(G)/12r!n\rfloor$-degenerate or $e(G[I])\le
24r!r\tau^re(G)$.
\end{thm}
\begin{proof}
  Let $\zeta = 1/12r!$. We may assume that $\tau \le \zeta^2 / r$, since
  otherwise we may take $T_{r-1},\ldots,T_0$ to be a partition of $I$ into
  sets of size at most $n/r$ and $C(T)=I$, in which case the constraints in
  the theorem are easily satisfied (since $|\C|\le2^n$: here we used
  $\tau<1/2$).
  Apply Theorem~\ref{thm:cover} to $G$ with $\zeta=1/12r!$. For each
  set $I$ we have $T=(T_{r-1},\ldots,T_0)$ and a container $C(T)$ satisfying
  properties (a)--(d) of that theorem.
  Take~$\C$ to be the collection of all such~$C$.
  Since $\tau\le
  \zeta^2/r$, we have $2r\tau/\zeta \le 2\zeta$, so $\mu(C)\le
  1-1/r!+6\zeta = 1-1/2r!$. It follows from
  inequality~(\ref{eqn:mutoe}) that $e(G[C])\le (1-1/2r!)e(G)$.

  Hence (a) and (c) of the present theorem are satisfied and it remains to
  check~(b). Theorem~\ref{thm:cover} tells us that each container $C$ is
  specified by sets $T_0,\ldots,T_{r-1}$ each of size at most $\theta n$,
  where $\theta = 2\tau /\zeta^2= 288r!^2\tau \le 2/r\le 1$. By
  Lemma~\ref{lem:entropy} we have
  $$
  \log |\C|\le r\theta n(1+\log(1/\theta))\le r\theta n\log(1/\tau)
  = 288r!^2rn\tau\log(1/\tau) \,,
  $$
  establishing~(b) and completing the proof.
\end{proof}

Repeated applications of Theorem~\ref{thm:coveroff} lead to the next
theorem. The rather technical appearance is the natural consequence of
retaining conditions on the codegree function at each stage, so that, if
information on this function is available, then use can be made of it.

\begin{thm}\label{thm:iteration}
  Let $G$ be an $r$-graph on vertex set $[n]$. Let $e_0 \le e(G)$. Suppose
  that, for each $U \subset [n]$ with $e(G[U]) \ge e_0$, the function
  $\tau(U)$ satisfies $\tau(U)<1/2$ and $\delta(G[U],\tau(U)) \le 1/12r!$.
  For $e_0\le m\le e(G)$ define
\begin{align*}
 f(m) &= \max\{\, - |U| \tau(U) \log \tau(U) : U \subset [n],\, e(G[U]) \ge
 m \} \\ 
\tau^* &= \max\{\, \tau(U) : U \subset [n],\, e(G[U]) \ge e_0 \}
\end{align*}
Let $k= \log(e_0/e(G))/\log(1-1/2r!)$.
Then there exists a collection $\C \subset \mathcal{P}[n]$ such that
\begin{itemize}
 \item[(a)] for every independent set $I$ there exists
   $T=(T_1,\ldots,T_s) \in \mathcal{P}(I)^s$ with $I \subset C(T) \in \C$,
   $|T_i| \le 288r!^2\tau^* n$ and $s \le (k+1)r$,
 \item[(b)] $e(G[C]) \le e_0$ for all $C \in \C$,
 \item[(c)] $\log|\C| \le 288r!^2r \sum_{0\le i< k} f( e_0 / (1-1/2r!)^i )$.
\end{itemize}

Moreover, (a) holds for all $I\subset[n]$ for which either $G[I]$ is
$\lfloor\tau(U)^{r-1}e(G[U])/12r!|U|\rfloor$-degenerate or $e(G[I])\le
24r!r\tau(U)^re(G[U])$, for all $U\subset[n]$ with $e(G[U])\ge e_0$.
\end{thm}
\begin{proof}
We will show that for all $t$ with $e_0\le t\le e(G)/(1-1/2r!)$, there
exists a collection
$\C_t \subset \mathcal{P}[n]$ satisfying conditions (a)--(c),
where the constant $e_0$ has been replaced by $t$ in (a)--(c),
and $k$ is replaced by $k(t)= \log(t/e(G))/\log(1-1/2r!)$.

When $t \ge e(G)$, we may take $\C_t = \{[n]\}$. Otherwise, suppose $t <
e(G)$.  It is enough to show that $\C_t$ exists provided $\mathcal{D} =
\C_{t/(1-1/2r!)}$ exists. Each $D\in\mathcal{D}$ is specified by a tuple
$T'=(T_1,\ldots,T_{s'})$ with $s'\le (k(t/(1-1/2r!))+1)r=k(t)r$. 
If $e(G[D]) \le t$, let $\C_t(D) =\{ D \}$. Otherwise, 
apply Theorem~\ref{thm:coveroff} with
$\tau=\tau(D)\le \tau^*$ to the $r$-graph $G[D]$, and let
$\C_t(D)$ be the collection of containers given by the theorem.
Then put $\C_t = \bigcup_{D \in \mathcal{D}} \C_t(D)$.

If $C \in \C_t(D)$ then $C$ is specified completely by $T'$, together
with the $r$-tuple appearing in condition~(a) of Theorem~\ref{thm:coveroff}
if the theorem was applied. Hence $C$ is
specified completely by a tuple of size at most $(k(t)+1)r$, so satisfying
condition (a).
If
$D \in \mathcal{D}$ then either $e(G[D]) \le t$ in which case
$|\C_t(D)| = 1$, or $e(G[D]) > t$ in which case
\[
\log |\C_t(D)| \le 288r!^2r |D|\tau(D)\log(1/\tau(D)) \le 288r!^2r f(t).
\]
Hence
\[
 \log |\C_t| \le \log |\mathcal{D}| + 288r!^2r f(t)
\le 288r!^2r \sum_{0\le i< k(t)} f( t / (1-1/2r!)^i ).
\]
Finally for $C \in \C_t(D)$, note that $e(G[C]) \le t$, since if
$e(G[D]) > t$ then by condition~(c) of Theorem~\ref{thm:coveroff}
$e(G[C]) \le (1-1/2r!) e(G[D]) \le t$.
\end{proof}

For certain applications the technical detail of
Theorem~\ref{thm:iteration} is not needed; what is required is a simple
statement that a few iterations will produce a container with a negligible
proportion of the original edges. Such a statement was presented earlier as
Corollary~\ref{cor:sparse_container}.

\begin{proof}[Proof of Corollary~\ref{cor:sparse_container}]
  Let $e_0 = \epsilon e(G)$.  Observe that for $U \subset [n]$, if $e(U)
  \ge \epsilon e(G)$ then $\delta(G[U],\tau) \le \delta(G,\tau)/\epsilon
  \le 1/12r!$. Therefore we may apply Theorem~\ref{thm:iteration} to the
  graph $G$ with $e_0= \epsilon e(G)$ and $\tau(U) = \tau$ for all $U$.
  Then $\tau^*=\tau$ and $f(m)= n\tau \log(1/\tau)$. Hence we obtain a
  collection $\C$ satisfying conditions (a) and (b) of the corollary, and
\[
 \log|\C| \le 288r!^2r\left(1+ \frac{\log \epsilon}{\log(1-1/2r!)}\right)
 n \tau \log(1/\tau), 
\]
giving condition (c).
\end{proof}

\section{Uniformly bounded containers}
\label{sec:uniform}

Theorem~\ref{thm:cover} provides containers of degree measure bounded away
from one. In this section we seek containers of uniform measure bounded
away from one.

For regular hypergraphs, the results of~\S\ref{sec:iteration} suffice, as
pointed out in~\S\ref{subsec:uniformly}. However for non-regular
hypergraphs we need something else. For reasons outlined
in~\S\ref{subsec:uniformly}, we consider initial intervals
$[v]\subset[n]$, and look for an interval such that $|C\cap [v]|$ is
bounded away from~$v$. There will in fact be many such intervals, as the
next lemma shows. This is the basic lemma which translates information
about $\mu$-measure into information about uniform measure. In the lemma,
$S$ is a multiset, so $\mu(S)$, $|S\cap [v]|$ and so on have their natural
interpretations counting with multiplicities.

\begin{lem}\label{lem:lemmoid}
Let $\mu:[n]\to\mathbb R$ be a measure with $\mu(1)\ge\mu(2)\ge
\cdots \ge \mu(n)$, and let $S\subset[n]$ be a multiset. Let $W=\{v:|S\cap
[v]|\ge\alpha v\}$. Then
$$
\alpha\,\mu(W)\,\le\,\mu(S)
$$
holds for all $\alpha\ge0$.
\end{lem}
\begin{proof}
  Let $W=\{w_1,\ldots,w_k\}$ where $k=|W|$ and
  $w_1<w_2<\ldots<w_k$. Define the numbers $s_1,\ldots,s_k$ by
  $s_1=|S\cap[w_1]|$ and $s_i = |S\cap [w_{i-1}+1,w_i]|$ for $i\ge2$.  Then
  we have $\mu(S\cap [w_{i-1}+1,w_i])\ge s_i\mu(w_i)$, because
  $\mu(1)\ge\mu(2)\ge \cdots \ge \mu(n)$. Therefore
\begin{align}
\mu(S)&\,\ge\,\mu(S\cap[w_1]) + \mu(S\cap[w_1+1,w_2]) +
\cdots +  \mu(S\cap[w_{k-1}+1,w_k])  \nonumber \\
&\,\ge\, s_1\mu(w_1)+s_2\mu(w_2)+\cdots+s_k\mu(w_k)  \nonumber \\
&\,=\, \sum_{i=1}^k \alpha\mu(w_i) + \sum_{i=1}^k(s_1+\cdots+s_i - \alpha
 i)(\mu(w_i)-\mu(w_{i+1}))  \nonumber \\
&\,=\, \alpha\mu(W) + \sum_{i=1}^k (s_1+\cdots+s_i - \alpha
 i)(\mu(w_i)-\mu(w_{i+1}))\,, \label{eqn:lemmoid}
\end{align}
where $\mu(w_{k+1})$ is defined to be zero.  Now
$|S\cap[w_i]|=s_1+\cdots+s_i$ holds for $1\le i\le k$, and so
$s_1+\cdots+s_i \ge \alpha w_i$, because $w_i\in W$. In particular,
$s_1+\cdots+s_i \ge \alpha i$, since $w_i\ge i$.  Moreover $\mu$ is a
decreasing function, so each summand in (\ref{eqn:lemmoid}) is
non-negative, and the lemma follows.
\end{proof}

In fact we shall need not just that $|C\cap[v]|$ is bounded for a single
container $C$ but that the average $(1/t)\sum_{i=1}^t |C_i\cap[v]|$ is
bounded for a collection $C_1,\ldots,C_t$. We shall, at the same time, be
interested in the sets $T_1,\ldots,T_s$ used to construct these containers,
and we need to find a $v$ for which the average of the $|T_j\cap[v]|$ is
simultaneously under control. For technical reasons arising when we come to
the application, very small values of~$v$ will be of no use, so we set a
lower bound on its value. The next lemma prepares the way.

\begin{lem}\label{lem:g}
  Let $\mu$ be a probability measure on $[n]$ with $\mu(1)\ge\mu(2)\ge
\cdots \ge \mu(n)\ge0$.
  Let $T_1,\ldots,T_s,C_1,\ldots,C_t$ be subsets of $[n]$, with
  $\mu(T_i)\le \lambda$ for $1\le i\le s$ and $\mu(C_j)\le 1-c-\eta$ for
  $1\le j\le t$, where $c,\eta>0$. Suppose moreover that $k\in[n]$ and
  $\mu([k])\le \eta c$. Then there exists $v\in[k, n]$ with
$$
\frac{1}{s}\sum_{i=1}^s|T_i\cap [v]|< \frac{\lambda}{\eta}v\mbox{\qquad and
  \qquad} \frac{1}{t}\sum_{i=1}^t |C_i\cap[v]| < (1-c)v\,.
$$
\end{lem}
\begin{proof}
  Let $U=\{v:\sum_{i=1}^s |T_i\cap[v]| \ge
  s\lambda v/\eta\}$.
  Writing $S$ for the multiset which is the disjoint union
  of $T_1,\ldots,T_s$, so that $\mu(S)\le s\lambda$ and $|S\cap [v]| =
  \sum_{i=1}^s |T_i\cap[v]|$, we can apply Lemma~\ref{lem:lemmoid} with
  $\alpha = s\lambda/\eta$ to obtain $\mu(U)\le \mu(S)/\alpha\le \eta$.
  
  In like manner, let $W=\{v:\sum_{i=1}^t |C_i\cap[v]| \ge
  t(1-c)v\}$. Writing now $S$ for the multiset which is the disjoint union
  of $C_1,\ldots,C_t$, so that $\mu(S)\le t(1-c-\eta)$ and $|S\cap [v]| =
  \sum_{i=1}^t |C_i\cap[v]|$, we apply Lemma~\ref{lem:lemmoid} with
  $\alpha=t(1-c)$ to obtain $\mu(W)\le t(1-c-\eta)/\alpha =
  1-\eta/(1-c)$.

  It follows that $\mu(U\cup W\cup[k]) \le \eta + 1 -\eta/(1-c) + \eta c <
  1$, so there exists $v\in[n]$ not contained in $U\cup W\cup[k]$. This $v$
  satisfies the conditions of the corollary (indeed, with $v\in[k+1,n]$).
\end{proof}

We can now prove the main result about containers and uniform measure,
namely Theorem~\ref{thm:uniform}, which was discussed
in~\S\ref{subsec:uniformly}. The idea of the proof is roughly as
follows. Theorem~\ref{thm:cover} supplies a set of containers. For each
tuple $(C_1,\ldots,C_t)$ of these containers we use Lemma~\ref{lem:g} to
nominate a vertex $v=g(C_1,\ldots,C_t)$ so that the restrictions to $[v]$
of the $C_i$ and of their generating sets $T_j$ are simultaneously bounded
in uniform measure. The online property means that the restrictions
$C_i\cap[v]$ are determined by the $T_j\cap[v]$, which are small, and so
the number of restricted containers is small.

\begin{proof}[Proof of Theorem~\ref{thm:uniform}]
  Apply Theorem~\ref{thm:cover} to~$G$ to obtain a collection $\C$ of
  containers $C(T)$ for $T=(T_{r-1},\ldots,T_0)\in\mathcal{P}([n])^r$. By
  assumption, $\tau\le\zeta^2/r$, and so $2r\tau/\zeta\le
  2\zeta$. Therefore $\mu(C(T))\le  1-1/r!+6\zeta$.

  Let $(C_1,\ldots,C_t)\in\C^t$, where $t\in\mathbb{N}$. Each $C_i$ is
  specified by an $r$-tuple of sets $T_j$, so the whole collection
  $(C_1,\ldots,C_t)$ is specified by $rt$ sets which, after re-labelling,
  we call $T_1,\ldots,T_{rt}$, with $\mu(T_i)\le 2\tau/\zeta$ for $1\le
  i\le rt$. Let $c=1/r!-8\zeta$, so (since $\zeta\le1/12r!$) $c>1/4r!$. Let
  $\eta=2\zeta$. By assumption,
  $\mu([k])\le\zeta/2r!$, so $\mu([k])<\eta c$. Hence the conditions of
  Lemma~\ref{lem:g} are satisfied with $s=rt$ and $\lambda=2\tau/\zeta$, and
  so there exists $v\in[k,n]$ with
$$
\frac{1}{s}\sum_{i=1}^s|T_i\cap [v]|< \frac{\tau}{\zeta^2}v\mbox{\qquad and
  \qquad} \frac{1}{t}\sum_{i=1}^t |C_i\cap[v]| < (1-\frac{1}{r!}+8\zeta)v\,.
$$
Define $g(C_1,\ldots,C_t)=v$. Then~(a) and~(c) of the theorem are
satisfied.

To obtain~(b), we need that the containers have the online property: in
other words, the $t$-tuple $(C_1\cap[v],\ldots,C_t\cap[v])$ is determined
by $T_1\cap[v],\ldots, T_s\cap[v]$. This online property is guaranteed by
Theorem~\ref{thm:cover}. Hence the size of the set
$Z=\{(C_1\cap[v],\ldots,C_t\cap[v]):g(C_1,\ldots,C_t)=v \}$ is bounded by
the number of tuples $(T_1\cap[v],\ldots, T_s\cap[v])$. Now
$\sum_{i=1}^s|T_i\cap [v]|< s\theta v$, where $\theta=\tau/\zeta^2
< 1$. So by Lemma~\ref{lem:entropy}
$$
\log|Z|\,\le\, s\theta v(1+\log(1/\theta))\,\le\,s\theta
  v\log(1/\tau)\,=\, \zeta^{-2} vtr\tau\log(1/\tau)\,,
$$
which completes the proof.
\end{proof}

\section{List colourings}\label{sec:list}

In~\cite{ST}, a lower bound for the list colouring number of a regular
hypergraph was proved. Theorem~2.1 of that paper, based on a simple
probabilistic argument, gave a bound of approximately $(\log k)/\log (1/c)$
provided there is a collection $\C$ of containers for the independent sets,
with $|C|\le(1-c)n$ for each $C\in\C$ and with $|\C|\le e^{n/k}$. The proof
fails to work for a general hypergraph because it is not possible to find
containers of bounded size.

As mentioned in~\S\ref{subsec:uniformly},
Corollary~\ref{cor:sparse_container} supplies suitable containers for
regular hypergraphs, and the number of containers is fewer than
in~\cite{ST}. This gives a direct improvement on the result
of~\cite{ST}. However, to obtain a similar result for general hypergraphs
we must make use of Theorem~\ref{thm:uniform}.

It is worth recapping briefly the simple argument of~\cite{ST}, because it
explains the basis of what follows though without the technicalities. It
also gives a clear illustration of why containers are useful.

Let $G$ be an $r$-graph with vertex set $[n]$.  Let $[t]$ be some set of
colours and let ${\mathcal L} =\{ L_u : u \in [n],\,L_u\subset[t] \}$ be a
collection of colour lists, one for each vertex, with $|L_u|=\ell$ for each
$u\in[n]$. A colouring of $G$ is a choice function $f:[n]\to[t]$ with
$f(u)\in L_u$ such that no edge is monochromatic. If we can find a
collection $\mathcal{L}$ with no colouring, then $\chi_l(G)>\ell$, which is
our goal. We choose the lists $L_u$ at random from a palette $[t]$ with $t$
around $\ell^2$ (so choosing with replacement is much the same as choosing
without). If the lists admit a choice function $f$, then, for each colour
$i\in[t]$, the set of vertices with $f(u)=i$ is independent. Thus there
exists a collection of independent sets $(I_1,\ldots,I_t)$ with $u\in
I_{f(u)}$ for all $u\in[n]$. We say that $\mathcal{L}$ is {\em compatible}
with $(I_1,\ldots,I_t)$ if such a choice function $f$ exists with $u\in
I_{f(u)}$ for all~$u$; in other words, the graph can be coloured so that
all the vertices receiving colour~$i$ lie within $I_i$, $1\le i\le
t$. Notice that we did not specify that $I_i$ is precisely the set of
vertices~$u$ with $f(u)=i$, only that it contains them all.

Let $\mathcal I$ be the collection of independent sets. It follows that if
we can find a collection $\mathcal{L}$ compatible with no tuple
$(I_1,\ldots,I_t)\in\mathcal{I}^t$, then we have shown $\chi_l(G)>\ell$. We
say that such an $\mathcal{L}$ is $\mathcal I$-{\em incompatible}.  Suppose
now that $|I|\le (1-c)n$ for all $I\in\mathcal{I}$. Roughly speaking
(precise calculations come in the proof of
Lemma~\ref{lem:uncompatible_lists}), given a tuple
$(I_1,\ldots,I_t)$, an average vertex $u$ will lie in at most $(1-c)t$ of
the $I_i$, so the probability that $L_u$ contains a colour $f(u)$ with
$u\in I_{f(u)}$ is at most $1-c^\ell\le e^{-c^\ell}$. Hence
the probability of $\mathcal{L}$ being compatible with a given tuple
$(I_1,\ldots,I_t)$ is at most $e^{-nc^\ell}$, and so the probability that
$\mathcal{L}$ fails to be $\mathcal I$-incompatible is at most
$|\mathcal{I}|^te^{-nc^\ell}$. If this were less than one then there would
exist an $\mathcal I$-incompatible collection $\mathcal{L}$.  Unfortunately
$|\mathcal{I}|$ can be as large as $2^{\Omega(n)}$ and the approach yields
nothing.

However, the same argument can be made with the containers $\C$ in place of
the independent sets $\mathcal I$; for each independent set $I_j$ above
there must be a $C_j$ containing it, and for a choice function to work
there must be a tuple $(C_1,\ldots,C_t)$ with which $f$ is compatible,
meaning $u\in C_{f(u)}$ for each $u\in[n]$. We now want $\mathcal{L}$ to be
$\C$-incompatible, that is, compatible with no $(C_1,\ldots,C_t)$, and,
assuming $|C_i|\le(1-c)n$ for all~$i$, the probability of this failing is
at most $|\C|^te^{-nc^\ell}$. If $|\C|=2^{\tau n}$ with
$\tau=d^{-1/(r-1)}$, then this probability is less than one for some $\ell$
with $\ell=\Omega(\log d)$, which is therefore a lower bound for
$\chi_l(G)$.

The existence of $\mathcal{L}$, contingent on the existence of a suitable set
of containers~$\C$, is proved in detail in the next lemma. The main
difference between the lemma and the preceding sketch is that we cannot
assume $|C_i|\le(1-c)n$ for each container, and instead we must use the
properties of $\C$ given by Theorem~\ref{thm:uniform}.

\begin{lem}\label{lem:uncompatible_lists}
  Let $0<\epsilon, c<1$. Then there exists
  $k_0=k_0(\epsilon,c)$, such that the following property holds for all
  $k>k_0$.

  Let $\ell = \lfloor (1-\epsilon) \log k / \log (1/c) \rfloor$
  and let $t = \lfloor 2\ell^2/c\rfloor$. Let $n>k$ and let $\C \subset
  \mathcal{P}[n]$. Suppose that there is a map $g:{\C}^t\to[k,n]$, such
  that
$$
\frac{1}{t}\sum_{i=1}^t|C_i\cap[v]|\,\le\,(1-c)v \leqno{\quad(a)}
$$
holds for every $(C_1,\ldots,C_t)\in{\C}^t$, where
$v=g(C_1,\ldots,C_t)$.  Suppose moreover that
$$
\left|\{\,(C_1\cap[v],\ldots,C_t\cap[v])\,:\,g(C_1,\ldots,C_t)=v\,
\}\right|\,\le\,e^{vt/k}\leqno{\quad(b)}
$$
holds for all $v\in[n]$.
Then there is a collection of lists $\{ L_u : u \in [n] \}$, each of size
$|L_u| =\ell$, which is $\C$-incompatible.
\end{lem}

\begin{proof}
  For each $u \in [n]$, let $L_u \in [t]^{(\ell)}$ be a subset of $[t]$ of
  size $\ell$ chosen uniformly and independently at random, and let
  ${\mathcal L}   =\{ L_u : u \in [n] \}$ be the collection of lists. We
  need to show that, with positive probability, ${\mathcal L}$ is
  compatible with no tuple $(C_1,\ldots,C_t) \in \C^t$.

  Given some $(C_1,\ldots,C_t) \in \C^t$, then ${\mathcal L}$ is compatible with
  $(C_1,\ldots,C_t)$ if there is a choice function $f:[n]\to[t]$ with $u\in
  C_{f(u)}$  for all~$u$. We define, for each
  $u\in[n]$, the set of colours
  \[
  B_u = B_u(C_1,\ldots,C_t) = \{ i \in [t] : u \in C_i \}\,.
  \]
  We can find a choice function if, and only if, we can select $f(u)\in
  L_u\cap B_u$ for each $u\in[n]$; in other words, if $L_u\cap
  B_u\ne\emptyset$. Hence we shall prove the theorem by showing that, with
  positive probability, for every tuple $(C_1,\ldots,C_t)$ there is some
  $u\in[n]$ with $L_u\cap B_u=\emptyset$.

  In fact, we claim something stronger: with positive probability,
  ${\mathcal L}$ {\em rejects} every tuple $(C_1,\ldots,C_t)$, meaning that
  there is some $u\in[v]$ with $L_u \cap B_u = \emptyset$, where
  $v=g(C_1,\ldots,C_t)$. Notice that the event that $(C_1,\ldots,C_t)$ is
  rejected depends only on $\mathcal L$ and on the tuple
  $(C_1\cap[v],\ldots,C_t\cap[v])$; it is because the conditions of the
  theorem give information about this tuple that we work with the stronger
  claim.

  To establish the claim, fix for the time being
  some tuple $(C_1,\ldots,C_t)$ and let $v=g(C_1,\ldots,C_t)$. Let
  $u\in[v]$ and write $\mathbf{1}_{u\in C_i}$ for the indicator that $u\in
  C_i$. By condition~(a) of the theorem, we have
\begin{equation}\label{eqn:BC}
  \sum_{u\in[v]}|B_u|=\sum_{u\in[v]}\sum_{i=1}^t \mathbf{1}_{u\in C_i}
  =\sum_{i=1}^t|C_i\cap[v]|\le (1-c)vt\,.\\
\end{equation}
Let $p_u$ be the probability that $L_u \cap B_u =
  \emptyset$, or equivalently $L_u \subset [t] \setminus B_u$. Then
\[
 p_u = \mathbb{P}(L_u \cap B_u = \emptyset) =
 \binom{z_u}{\ell}\binom{t}{\ell}^{-1}
\quad\mbox{where $z_u =t - |B_u|$}\,.
\]
We note here that $\ell\ge 1$ if $k_0$ is large enough and thus $ct>\ell$.
Write $z$ for the average of the values $z_u$ for $u\in[v]$. Then
inequality~(\ref{eqn:BC}) yields $ vz= \sum_u
t-|B_u|\geq vct$.
So we have
\begin{align}
 \sum_{u\in[v]} p_u =\sum_{u\in[v]}\binom{z_u}{\ell} \binom{t}{\ell}^{-1}
\,&\ge\, v  \binom{z}{\ell}\binom{t}{\ell}^{-1} \nonumber\\
  \,&\geq\, v \binom{ct}{\ell}\binom{t}{\ell}^{-1}
  \,\geq\, v (c-(\ell-1)/t)^\ell \,.\nonumber
\end{align}
Since $\ell\ge1$ we have $(\ell-1)/t\le (\ell-1)/(2\ell^2/c-1) \le c/2\ell$, and so
$(c-(\ell-1)/t)^\ell\ge c^\ell(1-1/2\ell)^\ell\ge c^\ell/2$.  Hence the probability that
$\mathcal L$ fails to reject $(C_1,\ldots,C_t)$ is
\begin{align}
 \mathbb{P}( B_u \cap L_u \ne \emptyset \mbox{ for all } u \in [v]) &=\,
 \prod_{u\in[v]} (1-p_u)  \nonumber\\
 &\leq\, \exp \{-\sum_{u\in[v]} p_u\}  \leq \exp \{-v c^\ell/2 \}\,.\nonumber
\end{align}
As mentioned, the event that $(C_1,\ldots,C_t)$ is not rejected depends
only on the tuple $(C_1\cap[v],\ldots,C_t\cap[v])$ and, by condition~(b) of
the theorem, there are at most $\exp\{vt/k\}$ of these tuples as
$(C_1,\ldots,C_t)$ ranges over ${\C}^t$. Hence if we fix~$v$ and write
$P_v$ for the probability that there is some tuple $(C_1,\ldots,C_t)$ with
$v=g(C_1,\ldots,C_t)$ which is not rejected, then, recalling the
definitions $\ell = \lfloor (1-\epsilon) \log k / \log (1/c) \rfloor$ and
$t = \lfloor 2\ell^2/c\rfloor$, we have
\begin{align*}
P_v\,&\,\le \exp \{ vt/k  - vc^\ell/2 \}\\
  &\,\le \exp \left\{ \frac{v}{2k} \left[ 
\frac{4}{c}\left(\frac{(1-\epsilon) \log k}{\log 1/c}\right)^2
- k^{\epsilon} \right] \right\} &\mbox{since $c^\ell\ge k^{\eps-1}$}\\
  &\,\le \exp \left\{ -\frac{v}{2k} k^{\epsilon/2}
  \right\}&\mbox{if $k_0$, and so $k$, is large enough}\\
&\,\le\exp \left\{-\frac{1}{2}v^{\epsilon/2}\right\} 
&\mbox{since $k\le v$}\\
&\,\le v^{-2}&\mbox{if $k_0$, and so $v\ge k_0$, is large enough.}
\end{align*}

Finally, if we consider all tuples $(C_1,\ldots,C_t)\in{\C}^t$,
the probability that one of them is not rejected is at most
$$
\sum_{v\in[k,n]} P_v \,\le\,\sum_{v\ge k}v^{-2}\,<\,1
$$
if $k_0$ is large. This establishes our claim and so proves the lemma.
\end{proof}

We can now prove our main result about list colouring. The proof follows by
feeding Theorem~\ref{thm:uniform} into Lemma~\ref{lem:uncompatible_lists}
(for regular graphs we use Corollary~\ref{cor:sparse_container} instead of
Theorem~\ref{thm:uniform}). The lower bound on $\chi_l(G)$ given by
Lemma~\ref{lem:uncompatible_lists} is $(1+o(1))\log k/\log
(1/c)$. Comparing condition~(b) in Theorem~\ref{thm:uniform} with that in
Lemma~\ref{lem:uncompatible_lists} shows that $k$ is not far from
$\zeta^2/\tau$, and we know that $\tau$ for simple graphs can be roughly
$d^{-1/(r-1)}$. This explains where the $\log d$ in the theorem comes
from.

To get the best result, we want the number $c$ in
Lemma~\ref{lem:uncompatible_lists} to be as large as possible, which, by
comparing Lemma~\ref{lem:uncompatible_lists}(a) with
Theorem~\ref{thm:uniform}(c) means making $\zeta$ small (unlike in other
applications where typically $\zeta=1/12r!$ is a good choice.) However if
$\zeta$ is too small then $k$ becomes small. For these reasons we choose
$\zeta=\zeta(d)$ so that, as $d\to\infty$, then $\zeta=o(1)$ and
$\zeta=d^{o(1)}$, the exponent here being negative.

\begin{proof}[Proof of Theorem~\ref{thm:chil}.]
  As explained in the preceding discussion, we take $\zeta=\zeta(d)$ so
  that, as $d\to\infty$, then $\zeta=o(1)$ and $\zeta=d^{o(1)}$.
  Let
  $\tau=d^{-1/(r-1)}\zeta^{-3}$. We now check that the conditions of
  Theorem~\ref{thm:uniform} are satisfied.  Certainly $\zeta\le 1/12r!$
  because $\zeta=o(1)$. Also, recalling Definition~\ref{defn:delta},
  we have $d^{(j)}(v)\le 1$ by simplicity and
  $\delta_j = \sum_vd^{(j)}(v)/\tau^{j-1}nd\le\zeta^2=o(\zeta)$, so
  $\delta(G,\tau)\le \zeta$. Moreover $\tau\le \zeta^2/r$ because
  $\zeta=d^{o(1)}$.

  Let $k=\lfloor\zeta^3/\tau\log(1/\tau)\rfloor$. Then $\log
  k=(1/(r-1)+o(1))\log d$. Let $v\in V(G)$. For each $u\in V(G)$, at most
  one edge contain both $u$ and $v$,
  and so $d(v)\le n$. It follows that $\mu([k])\le
  (1/nd)kn\le\zeta/2r!$. This completes
  the check of the conditions of Theorem~\ref{thm:uniform}. Therefore
  there exists a collection $\C$ of containers for the independent sets
  of~$G$, satisfying properties~(b) and~(c) of Theorem~\ref{thm:uniform},
  and since $\zeta^{-2}r\tau\log(1/\tau)< 1/k$ it follows that
  conditions~(a) and~(b) of Lemma~\ref{lem:uncompatible_lists} are
  satisfied, with $c=1/r!-8\zeta\ge (1+o(1)) r^{-(r-1)}$.

  Consequently there are lists of size $(1+o(1))\log k/\log(1/c)$ that are
  not $\C$-compatible, which is to say lists of size at least
  $(1/(r-1)+o(1))\log d /\log(1/c)\ge(1/(r-1)^2+o(1))\log_r d$. Since $\C$
  is a set of containers for the independent sets of~$G$, the first claim
  of the theorem follows.

  The proof for regular graphs is similar, except that in
  Lemma~\ref{lem:uncompatible_lists} we are able to take $c=1/r+o(1)$. To
  achieve this we make use of Corollary~\ref{cor:sparse_container} instead
  of Theorem~\ref{thm:uniform}. With $\tau$, $\zeta$ and $k$ defined as
  before, we can take $\epsilon=\zeta$ in
  Corollary~\ref{cor:sparse_container} because
  $\delta(G,\tau)=o(\zeta)$. We obtain a collection $\C$ of containers such
  that $e(G[C])\le \zeta e(G)=o(e(G))$ for all $C\in\C$. Because $G$ is
  regular this implies, as mentioned after inequality~(\ref{eqn:mubig}),
  that $|C|\le(1-1/r+o(1))n$ where $n=|G|$. We can now apply
  Lemma~\ref{lem:uncompatible_lists} by defining $g(C_1,\ldots,C_t)=n$
  for all $(C_1,\ldots,C_t)$; note that condition~(b) of the theorem is
  satisfied because, by Corollary~\ref{cor:sparse_container}, $\log|\C|=O(
  \log(1/\epsilon)n\tau\log(1/\tau))<n/k$. The remainder of the proof is
  the same.
\end{proof}

The bound given for $r$-graphs of average degree~$d$ is weaker
than that for regular $r$-graphs because we only had containers of measure
$1-1/r!$ available, rather than $1-1/r$. Probably this is an artifact of
our algorithm, and $\chi_l(G)\ge(1/(r-1)+o(1))\log_rd$ holds
for $r$-graphs of average degree~$d$.

Observe that, since the proof uses Corollary~\ref{cor:sparse_container}
instead of Theorem~\ref{thm:uniform} for regular graphs, it is not
necessary to impose the condition of simplicity in the regular case. The
proof in fact works provided $d^{(j)}(v)\le d^{(r-j)/(r-1)+o(1)}$ as
$d\to\infty$ for every $v\in V(G)$ and every $2\le j\le r$ (recall
Definition~\ref{defn:delta}), since $\tau$ can then be chosen to ensure
$\delta(G,\tau)\le\zeta$, and the bound on $d^{(2)}$ implies a bound on the
maximum degree which in turn bounds $\mu([k])$. This implies a theorem of
Alon and Kostochka~\cite{AK2} in the case of regular hypergraphs.

As far as non-simple regular graphs go, the bound is tight. Indeed, let
$K(r,m)$ be the complete $r$-partite $r$-graph with $m$ vertices in each
class. Suppose that lists of size $\ell$ are given to the
vertices. Randomly choose, for each colour in the palette, a vertex class
on which that colour is forbidden to be used; then the expected number of
vertices with no available colour is $rmr^{-\ell}$ which is less than one
if $\ell> 1+\log_r m$, and so $\chi_l\le2+\log_r m$ (see Haxell and
Verstra\"ete~\cite{HV}). This graph is $d$-regular where $d=m^{r-1}$ so
$\chi_l\le2+(1/(r-1))\log_r d$. Note that
$d^{(j)}(v)=m^{r-j}=d^{(r-j)/(r-1)}$.

It is not hard to construct an $m$-regular simple subgraph $G$ of $K(r,m)$,
and so (putting $d=m$) we have simple $d$-regular $r$-graphs with
$\chi_l\le2+\log_r d$. Quite possibly $\chi_l\le2+(1/(r-1))\log_r
d$ in this case too, because a subgraph of $G$ with
$d^{1-1/(r-1)}$ vertices in each class is likely to be very sparse, and a
random colouring might be repairable if $rdr^{-\ell}< d^{1-1/(r-1)}$, or
$\ell>1+(1/(r-1))\log_r d$. But this argument is far from rigorous.

As an illustration of the use of containers for non-independent sets we
finish with the next result.

\begin{thm}
  \label{thm:list_colour_planar}
  Let $G$ be a graph with average degree~$d$. Then, for each $u\in V(G)$
  there is a list $L_u$ of $(1+o(1))\log_2d$ colours, such that it is not
  possible to choose a colour $c(u)\in L_u$ with the vertices of each
  colour spanning a planar graph.
\end{thm}
\begin{proof}
  We follow the proof of Theorem~\ref{thm:chil} with $r=2$, except we use
  a set $\C$ of containers for those subsets $I$ for which $G[I]$ is
  planar. Since a planar graph is $5$-degenerate, we can apply
  Theorem~\ref{thm:uniform} and continue with the proof exactly as before,
  provided $5\le \tau d\zeta/r$. But $\tau=d^{-1}\zeta^{-3}$ so this condition
  holds comfortably.
\end{proof}

It is possible to extend the colouring results here to non-simple $r$-graphs ---
see~\S\ref{sec:postscript}.

\section{$H$-free graphs}\label{sec:hfree}

In this section we prove Theorem~\ref{thm:ffree_cover}. In fact we will show
a slight strengthening of it. We will apply the container theorem given by
Corollary~\ref{cor:sparse_container} to the following hypergraph, whose
independent sets correspond to $H$-free $\ell$-graphs on vertex set $[N]$.

\begin{defn}\label{def:GNH}
 Let $H$ be an $\ell$-graph. Let $r=e(H)$. The $r$-graph $G(N,H)$ has vertex
set $[N]^{(\ell)}$, where $B=\{v_1, ..., v_r\} \in V(G)^{(r)}$ is an edge
whenever $B$, considered as an $\ell$-graph with vertices in $[N]$ and with
$r$ edges, is isomorphic to $H$.
\end{defn}

We re-emphasise that all our results about $H$-free graphs are simple
consequences of Theorem~\ref{thm:ffree_cover}, and that this theorem is
itself just a restatement, in graphical language, of the container theorem
applied to the hypergraph $G(N,H)$. As already mentioned, $H$-free graphs
are precisely independent subsets of~$G(N,H)$, and the graphs that contain
the $H$-free graphs, which themselves contain few $H$-free graphs, are
precisely the containers given by Corollary~\ref{cor:sparse_container}. In
order to apply the corollary to~$G(N,H)$, all that is needed is to estimate
$\delta(G(N,H),\tau)$. The easy calculation is carried out in
Lemma~\ref{lem:ffree_delta}. As discussed in~\S\ref{subsec:optimality},
the outcome is more or less optimal, for every~$H$.

That said, we do permit ourselves a variation on this theme. For certain
purposes, it turns out that we want to work, not with all possible copies
of~$H$, but with only a subset of them of particular interest. For example,
we may care only about copies of~$H$ whose vertices are aligned with some
partition of~$[N]$; this is the case for the K{\L}R conjecture
(Theorem~\ref{thm:klr3}). It will be seen that all that is needed here is
to apply the container theorem to a subgraph of $G(N,H)$ rather than to
$G(N,H)$ itself. Hence the next theorem generalizes
Theorem~\ref{thm:ffree_cover} by allowing this. At the same time we also
strengthen the theorem by providing a collection of containers for graphs
that are not necessarily $H$-free, but nonetheless contain few copies
of~$H$. This is allowed by the container theorem, in which the sets $I$ do
not have to be completely independent.

The notation $\widetilde{G} \subset G(N,H)$ means $\widetilde{G}$ is a
subgraph of $G(N,H)$, and in every case of interest
$V(\widetilde{G})=V(G(N,H))= [N]^{(\ell)}$. The edges of $\widetilde{G}$
thus represent a subcollection of the copies of~$H$ on vertex set~$[N]$.
So if $A \subset [N]^{(\ell)}$, that is, if $A$ is an $\ell$-graph on vertex
set $[N]$, then the induced $e(H)$-graph $\widetilde{G}[A]$ corresponds to
all copies of~$H$ in the collection~$\widetilde{G}$ that are present
in~$A$, and $e(\widetilde{G}[A])$ is the number of copies of~$H$ that are
both in the collection~$\widetilde{G}$ and present in~$A$.

Recall that $\pi(H)=\lim_{N\to\infty}{\rm ex}(N,H)\binom{N}{\ell}^{-1}$.

\begin{thm}\label{thm:ffree_cover_heavy}
Let $H$ be an $\ell$-graph with $e(H)\ge2$ and let $\epsilon>0$.
There exists $c>0$ such that the following is true.
Let $N \ge c$.
Let $\widetilde{G} \subset G(N,H)$ with $e(\widetilde{G}) \ge \eps N^{v(H)}$.
Let~$q$ satisfy $N^{-1/m(H)} \le q \le 1/c$.
Then there exists a collection
$\C$ of $\ell$-graphs on vertex set~$[N]$ such that
\begin{itemize}
 \item[(a)]
   for every $\ell$-graph $I \subset [N]^{(\ell)}$ with
   $e(\widetilde{G}[I]) < q^{e(H)} N^{v(H)}$, there exists
   $C\in\C$ with $I\subset C$,
 \item[(b)]
   every $C \in \mathcal{C}$ satisfies
   $e(\widetilde{G}[C]) \le \epsilon N^{ v(H)}$, and
   moreover if $\widetilde{G} = G(N,H)$ then
   $e(C) \le (\pi(H)+\epsilon) {N \choose \ell}$,
 \item[(c)]
   $\log |\mathcal{C}| \leq c q N^\ell \log N$.
 \item[(d)]
   moreover, for every $I$ in (a), there exists
   $T=(T_1,\ldots,T_s)$ where $T_i \subset I$, $s\le c$ and
   $\sum_ie(T_i) \le c q N^\ell$, such that $C=C(T)$,
\end{itemize}
\end{thm}

Theorem~\ref{thm:ffree_cover} follows immediately from
Theorem~\ref{thm:ffree_cover_heavy} by taking $\widetilde{G} = G(N,H)$ and
$q=N^{-1/m(H)}$.

All that remains before applying the container theorem to $G(N,H)$, or more
generally to some dense subgraph $\widetilde{G}$, is to calculate
$\delta(\widetilde{G},\tau)$.

\begin{lem}\label{lem:ffree_delta}
 Let $H$ be an $\ell$-graph with $r=e(H)\ge2$ and let $\gamma\le1$.
 Let $N$ be sufficiently large.
 Let $\widetilde{G} \subset G(N,H)$ with $e(\widetilde{G}) = \alpha N^{v(H)}$
 for some $\alpha>0$. Then
\[
 \delta\big(\widetilde{G},\gamma^{-1}N^{-1/m(H)}\big) \le
 2^{r^2}v(H)!\gamma/\alpha. 
\]
\end{lem}

\begin{proof}
  Consider $\sigma \subset [N]^{(\ell)}$ (so $\sigma$ is both a set of
  vertices of $\widetilde{G}$ and an $\ell$-graph on vertex set~$[N]$). The
  degree of $\sigma$ in $\widetilde{G}$ is at most the number of ways of
  extending $\sigma$ to an $\ell$-graph isomorphic to $H$. If $\sigma$ as
  an $\ell$-graph is not isomorphic to any subgraph of $H$, then clearly
  $d(\sigma)=0$. Otherwise, let $v(\sigma)$ be the number of vertices in
  $\sigma$ considered as an $\ell$-graph, so there exists $V \subset [N]$,
  $|V|=v(\sigma)$ with $\sigma \subset V^{(\ell)}$.  Edges of
  $\widetilde{G}$ containing $\sigma$ correspond to copies of $H$ in
  $[N]^{(\ell)}$ containing $\sigma$, each such copy given by a choice of
  $v(H)-v(\sigma)$ vertices in $[N]-V$ and a permutation of the vertices of
  $H$.  Hence for $N$ sufficiently large,
\[
 d(\sigma) \le v(H)! {N -v(\sigma) \choose v(H)-v(\sigma)} \le v(H)!
 N^{v(H)-v(\sigma)} 
\]

For $v\in V(G)$ and $1\le j\le r$, the quantity $d^{(j)}(v)$ is the maximum of
$d(\sigma)$ over all $\sigma \subset [N]^{(\ell)}$ with
$v \in \sigma$ and $|\sigma| = j$. Thus
\[
 d^{(j)}(v) \le v(H)! N^{v(H)-f(j)},
   \quad \mbox{where } f(j) = \min_{H' \subset H,\, e(H') = j} v(H').
\]
Let $\tau = \gamma^{-1} N^{-1/m(H)}$. Since $\sum_v d^{(1)}(v) = \alpha r
N^{v(H)}$, for $2\le j\le e(H)$ we have
\[
 \delta_j = \frac{\sum_v d^{(j)}(v)}{\tau^{j-1} \alpha r N^{v(H)}}
   \le (v(H)!/r\alpha) \tau^{1-j} N^{\ell-f(j)}
   \le (v(H)!/r\alpha) N^{\ell-f(j)+(j-1)/m(H)} \gamma.
\]
By definition of $f(j)$ and $m(H)$, $\ell-f(j)+(j-1)/m(H)\le0$.
Hence $\delta_j \le (v(H)!/r\alpha) \gamma$ and so
\[
 \delta(G,\tau) = 2^{{r\choose2}-1} \sum_{j=2}^r 2^{-{j-1\choose2}}\delta_j
   \le 2^{r^2} (v(H)!/\alpha) \gamma
\]
as claimed.
\end{proof}

A well-known supersaturation theorem bounds the number of edges in
containers.

\begin{prop}[Erd\H{o}s and Simonovits~\cite{ES}]\label{prop:supersaturation}
Let $H$ be an $\ell$-graph and let $\epsilon > 0$. There exists $N_0$ and
$\eta > 0$ such that if $C$ is an $\ell$-graph on $N \ge N_0$ vertices
containing at most $\eta N^{v(H)}$ copies of $H$ then
$e(C) \le (\pi(H) + \epsilon){N \choose \ell}$.
\end{prop}

\begin{proof}[Proof of Theorem~\ref{thm:ffree_cover_heavy}]
In what follows, $c$ is taken to be sufficiently large
(depending on $\eps$ and~$H$).
Let $\eta=\eta(\eps,H)$ be given by Proposition~\ref{prop:supersaturation},
and let $\beta = \min\{\eps, \eta\}$.
Recall that $r=e(H)$.
Apply Corollary~\ref{cor:sparse_container} to~$\widetilde{G}$ with
$\tau = \sqrt{c} q$ and with $\beta$ playing the role
of~$\eps$ in the corollary. Thus $\sqrt{c}N^{-1/m(H)}\le \tau$, and so
Lemma~\ref{lem:ffree_delta} implies that $\delta(\widetilde{G},\tau) \le
\beta / 12r!$ if $c$ is large. 
Moreover $\tau\le 1/\sqrt c < 1/2$ if $c$ is large.
Hence the conditions of
Corollary~\ref{cor:sparse_container} are satisfied; denote by~$\tilde{c}$
the constant~$c$ appearing in the corollary. 
The collection of containers~$\C$ satisfies the following.
\begin{itemize}
 \item For every $I$ with $e(\widetilde{G}[I]) \le 24 \beta r! r\tau^r
   e(\widetilde{G})$, 
  there exists some $C \in \C$ with $I \subset C$. This implies
  condition~(a) of the present theorem, since $q^{e(H)}N^{v(H)}=(\tau/\sqrt
  c)^rN^{v(H)} \le (\tau/\sqrt c)^r(1/\eps)e(\widetilde{G})
  \le 24 \beta r! r\tau^r e(\widetilde{G})$
  provided that~$c$ is sufficiently large.
  
\item For each $C\in\C$, we have $e(\widetilde{G}[C]) \le \beta
  e(\widetilde{G}) \le \beta N^{v(H)}$. In the case that $\widetilde{G} =
  G(N,H)$, Proposition~\ref{prop:supersaturation} implies $e(C)
  \le(\pi(H)+\epsilon){N \choose \ell}$, because we chose $\beta \le \eta$.
 This gives condition~(b).

\item The size of the collection is $\log |\C| \le \tilde{c} \log(1/\beta)
  \binom{N}{\ell} \tau \log(1/\tau)$, which gives condition~(c), again
  provided that~$c$ is sufficiently large.

\item Finally, for every set $I$ as above, there exists
  $T = (T_1, \ldots, T_s) \in \mathcal{P}(I)^r$ such that $C=C(T)$,
  $|T_i| \le \tilde{c} \tau \binom{N}{\ell}$, and
  $s \le \tilde{c} \log (1/\beta)$. This implies condition~(d) of the present
  theorem, provided that~$c$ is sufficiently large.
\end{itemize}
This completes the proof.
\end{proof}

We now prove the theorems about induced $H$-free graphs that were stated
in~\S\ref{subsec:ihfree}. As mentioned there, the proofs are very similar
to those just given for $H$-free graphs, so we shall sketch the
details. The crucial difference is that we need to consider containers not
in $G(N,H)$ but in another hypergraph that captures induced copies of~$H$.

We already discussed in~\S\ref{subsec:ihfree} how the notion of 2-coloured
multigraphs can help. We say that a 2-coloured $\ell$-multigraph $J$ on
vertex set $[N]$ is {\em entire} if $J_R\cup J_B=[N]^{(\ell)}$. One can
think of $J$ as representing a class of $\ell$-graphs on vertex set~$[N]$, in
each of which the edges of $J_R\setminus J_B$ are present, the edges
of $J_B\setminus J_R$ are absent, but edges in $J_R\cap J_B$ can
be present or absent.

Let $r={v(H) \choose \ell}$. Let $G_{\rm i}(N,H)$ be the $r$-graph whose
vertex set is two copies of $[N]^{(\ell)}$, denoted by $V_R$ and $V_B$
(vertices in $V_R$ correspond to $\ell$-edges and vertices in $V_B$
correspond to non-$\ell$-edges), and whose edges correspond to induced
copies of $H$; thus $f \in (V_R \cup V_B)^{(r)}$ is an edge of $G_{\rm
  i}(N,H)$ whenever $f \cap V_R$ and $f \cap V_B$ are the edges and
non-edges, respectively, of an $\ell$-graph isomorphic to $H$ with vertices
in~$[N]$.  Note that every induced-$H$-free $\ell$-graph $I \subset
[N]^{(\ell)}$ corresponds to an independent set of $G_{\rm i}(N,H)$, namely
the set $I_R\subset V_R$ corresponding to the edges of $I$ together with
the set $I_B\subset V_B$ corresponding to non-edges of~$I$. Observe that,
regarded as a 2-coloured multigraph with edges $I_R$ and $I_B$, $I$ is
entire. In general, every subgraph of $G_{\rm i}(N,H)$ corresponds to a
2-coloured $\ell$-multigraph on vertex set~$[N]$, and if any such subgraph
$C$ contains an independent set $I$ representing an induced $H$-free graph
as just described, then $C$ is entire.

The graph $G_{\rm i}(N,H)$ has very similar properties to those of $G(N,K)$
where $K=K_{v(H)}^{(\ell)}$ is the complete $\ell$-graph on $v(H)$
vertices. In particular, for fixed $\tau$, the $\delta_j$ for $G_{\rm
  i}(N,H)$ differ by only a constant factor from those for $G(N,K)$. Let
$m=m(K)=\left({v(H)\choose \ell}-1\right)/(v(H)-\ell)$. Then
Lemma~\ref{lem:ffree_delta}, or the calculation in its proof, shows that
$\delta(G_{\rm i}(N,H),N^{-1/m}/\epsilon)=O(\epsilon)$. We are now ready to
establish Theorem~\ref{thm:ind_ffree_cover}.

\begin{proof}[Proof of Theorem~\ref{thm:ind_ffree_cover}]
  We mimic the proof of Theorem~\ref{thm:ffree_cover_heavy} by applying
  Corollary~\ref{cor:sparse_container}, but this time to the graph $G_{\rm
    i}(N,H)$. We just noted that $\delta(G_{\rm
    i}(N,H),N^{-1/m}/\epsilon)=O(\epsilon)$, so we can choose
  $\tau=O(N^{-1/m}/\epsilon)$ so that the conditions of
  Corollary~\ref{cor:sparse_container} are satisfied. The corollary yields
  a collection $\C \subset \mathcal{P}(V_R \cup V_B)$, where
  each $C\in\C$ is identified with a 2-coloured $\ell$-multigraph in the
  natural way. The properties of $\C$ claimed in
  Theorem~\ref{thm:ind_ffree_cover} follow directly from those provided by
  the corollary.
\end{proof}

Next we derive Theorem~\ref{thm:ind_ffree_measure}.  To do so, we need a
suitable version of Proposition~\ref{prop:supersaturation}. Recall the
definition of the function $H_p$ given in~\S\ref{subsec:ihfree}.

\begin{lem}[Supersaturation for induced $\ell$-graphs]
\label{lem:ind_supersaturation}
  Let $H$ be an $\ell$-graph and let $0<\epsilon,p<1$. There exists $N_0$
  and $\eta > 0$ such that if $C$ is an entire 2-coloured
  $\ell$-multigraph on $N \ge N_0$ vertices containing at most $\eta
  N^{v(H)}$ copies of $H$ then $H_p(C) \ge (h_p(H) - \epsilon){N \choose
    \ell}$.
\end{lem}
\begin{proof}
  Since $H_p(C)\ge 0$ for all $C$, we may that assume $h_p(H)\ge \epsilon$, the
  lemma being otherwise trivial.
  By the definition of $h_p(H)$, there exists some $m\ge1$ such that
  every entire 2-coloured $\ell$-multigraph $D$ on $m$ vertices with
  $H_p(D) < (h_p(H) - \epsilon/2){m \choose \ell}$ contains a copy of~$H$.
  Let $\mathcal{M} \subset [N]^{(m)}$ be the collection of $m$-sets $M$
  such that $C[M]$ contains $H$. Each edge of $C$ appears in $C[M]$ for 
  ${n-\ell\choose m-\ell}$ sets $M\in [N]^{(m)}$. By considering the
  contribution of each edge to $H_p(C)$ and to $H_p(C[M])$ we see that
\begin{align*}
 H_p(C){N-\ell\choose m-\ell}
 =\sum_{M\in[N]^{(m)}} H_p(C[M])
 &= \sum_{M\in\mathcal{M}} H_p(C[M]) + \sum_{M\in[N]^{(m)}-\mathcal{M}} H_p(C[M]) \\
  &\ge 0 + \left({N\choose
      m}-|\mathcal{M}|\right)(h_p(H)-\epsilon/2){m\choose \ell} 
\end{align*}
Suppose that $C$ contains at most $\eta N^{v(H)}$ copies of $H$ for some
$\eta>0$. Each copy of $H$ is contained in ${N-v(H)\choose
  m-v(H)}$ subgraphs $C[M]$, $M \in \mathcal{M}$, and hence
\[
 |\mathcal{M}| \le \eta N^{v(H)} {N-v(H)\choose m-v(H)} \le \eta N^m
\,\le\,{\eps\over2h_p(H)}{N\choose m},
\]
where the last inequality holds if $\eta$ is small and $N$ is large.
Dividing through by ${N-\ell\choose m-\ell}$ in the above expression for
$H_p(C)$, we obtain
\[
 H_p(C) \ge (1-\eps/2h_p(H))(h_p(H)-\epsilon/2){N \choose \ell}
\ge(h_p(H)-\epsilon){N\choose \ell}\,,
\]
as claimed.
\end{proof}

\begin{proof}[Proof of Theorem~\ref{thm:ind_ffree_measure}]
  Recall from~\S\ref{subsec:ihfree} the definition of $\mbox{hex}_p(H,N)$.
  Take an entire 2-coloured $\ell$-multigraph $J$ satisfying $H\not\subset
  J$ and $H_p(J)=\mbox{hex}_p(H,N) = (h_p(H)+o(1)) {N \choose \ell}$. The
  probability that $ G^{(\ell)}(N,p)$ is induced $H$-free is at least the
  probability that $ G^{(\ell)}(N,p)\subset J$, which equals $2^{-H_p(J)}$,
  giving the lower bound in the theorem. Now let $\epsilon>0$ and let
  $\eta$ be given by Lemma~\ref{lem:ind_supersaturation}.  Let $\C$ be the
  collection of 2-coloured $\ell$-multigraphs given by
  Theorem~\ref{thm:ind_ffree_cover} satisfying $|\C|=2^{o(N^\ell)}$ and for
  every $C\in\C$, the number of copies of $H$ in $C$ is at most $\eta
  N^{v(H)}$. Let $\C'\subset\C$ consist of those $C\in\C$ that are
  entire. By Lemma~\ref{lem:ind_supersaturation}, for each $C\in\C'$ we
  have $H_p(C)\ge(h_p(H)-\epsilon){N\choose \ell}$.  Since every
  induced-$H$-free graph on vertex set $[N]$ is contained in some $C\in\C'$,
\[
 \mathbb{P}(G^{(\ell)}(N,p) \mbox{ is induced-$H$-free})
  \le \sum_{C\in\C'} 2^{-H_p(C)}
  \le 2^{-(h_p(H)-\epsilon+o(1)){N\choose \ell}}.
\]
But $\epsilon>0$ was arbitrary and this completes the proof of
Theorem~\ref{thm:ind_ffree_measure}.
\end{proof}

\section{Sparsity}\label{sec:sparse}

In this section we prove Theorem~\ref{thm:ffree_sparse} and related
theorems. We remark once again that there are no further applications of a
container theorem here; we just use Theorem~\ref{thm:ffree_cover} or the
slightly more technical Theorem~\ref{thm:ffree_cover_heavy} together with
some straightforward probabilistic arguments. (In the same way, sparse
arithmetical results such as Theorem~\ref{thm:szem_sparse} and those
obtained in~\cite{STa} follow from a theorem analogous to
Theorem~\ref{thm:ffree_cover} about solution-free sets.)

Note that the condition $p \ge cN^{-1/m(H)}$ in Theorem~\ref{thm:ffree_sparse}
is tight up to the value of~$c$. Indeed, if $p=o(N^{-1/m(H)})$, it is readily
checked that for some subgraph $H'\subset H$ with $m(H')=m(H)$, the expected
number of copies of $H'$ is much less than the number of edges, and removing
very few edges will result in an $H$-free subgraph.

As a further illustration of the paradigm described
in~\S\ref{subsec:sparse}, we prove two other conjectures of Kohayakawa,
{\L}uczak and R\"odl~\cite{KLR}. The first of these has already been
proved, by Conlon and Gowers~\cite{CG} for strictly balanced graphs and by
Samotij~\cite{S2}, following Schacht~\cite{Sch}, for all graphs. It states
that, for non-bipartite~$H$, not only does every $H$-free subgraph $I$ of a
random graph have at most $(1+o(1))p\pi(H){N \choose 2}$ edges, but in the
case that $I$ has close to $p \pi(H){N\choose2}$ edges, it can be made
$(\chi(H)-1)$-partite by removing a small number of edges.

\begin{thm}\label{thm:klr_conj}
Let $H$ be a $2$-graph with $\pi(H) > 0$ and let $0<\gamma<1$.
There exist constants $\epsilon,c>0$ such that for $N$ sufficiently large and
for $p \ge cN^{-1/m(H)}$,
the following is true. Let $E_0$ be the event that there
exists an $H$-free subgraph $I\subset G(N,p)$ with
$e(I)\ge(1-\frac{1}{\chi(H)-1}-\epsilon)p{N\choose2}$ which cannot be made
$(\chi(H)-1)$-partite by removing at most $\gamma p{N\choose2}$ edges. Then
$
 \mathbb{P}(E_0) \le \exp\{-\epsilon^2 p {N \choose 2} \}.
$
\end{thm}

The dense ($p=1$) version of this theorem is the stability theorem of
Erd\H{o}s and Simonovits~\cite{E1,E2,Sim}; indeed this theorem states that
{\em every} sufficiently dense $H$-free graph $I$ can be made
$(\chi(H)-1)$-partite in the way described. Theorem~\ref{thm:klr_conj} is
therefore the assertion that a similar phenomenon holds with high
probability in sparser random graphs.

The other conjecture from~\cite{KLR}, sometimes known as the K\L{R}
conjecture, has a more technical statement. Let $G$ be a graph. For $U,W
\subset V(G)$, write $E_G(U,W) \subset E(G)$ for the set of edges of $G$
with one vertex in~$U$ and one vertex in~$W$.  Let $e_G(U,W)=|E_G(U,W)|$
and write $d_G(U,W)=e_G(U,W)/(|U||W|)$ for the \emph{edge density}.  For
$0<\eta,p\le1$, say that the pair $(U,W)$ is \emph{$(\eta,p)$-regular} if
for every $U'\subset U$ with $|U'|\ge \eta|U|$ and $W'\subset W$ with
$|W'|\ge\eta|W|$, the edge density satisfies
\[
 |d_G(U',W')-d_G(U,W)|\le\eta p.
\]
This extends the notion of regularity to sparse graphs of density~$p$.

Let $H$ be a graph on vertex set $[h]$. In what follows, $V_1,\ldots, V_h$
is a partition of $[N]=[hn]$, where each part has size $|V_i|=n$. Let $G$
be a graph on vertex set~$[N]$.  We say that $G$ is {\em
  $(H,\eta,p)$-regular} if for every pair $(V_i,V_j)$ with $\{i,j\}\in
E(H)$, the bipartite subgraph of~$G$ between $V_i$ and $V_j$ is
$(\eta,p)$-regular. A {\em canonical copy} of~$H$ in $G$ (whether regular
or not) is a set of vertices $v_1, \ldots, v_h$ with $v_i\in V_i$ such that
$\{v_i,v_j\}$ is an edge of~$G$ whenever $\{i,j\}\in E(H)$; we say that
such a copy of $H$ is {\em aligned} to the partition
$V_1,\ldots,V_h$. Denote by $i_H(G)$ the number of canonical copies of $H$
in~$G$. We say that $G$ is {\em $H$-free} if it does not contain any
canonical copies of~$H$; that is, $i_H(G)=0$.  Finally, denote by
$G=G(n,M,H)$ a graph chosen uniformly at random from all $h$-partite graphs
with parts $V_1,\ldots,V_h$, having $e_G(V_i,V_j)=M$ if $\{i,j\}\in E(H)$
and $e_G(V_i,V_j)=0$ otherwise.

We shall be interested in whether $G=G(n,M,H)$ is $(H,\eta,p)$-regular, where
we shall always take $p = M / n^2$. (This may seem
like a strict requirement in the definition of regularity. However,
the value of~$p$ does not matter up to a constant
factor, since it is only the value of~$\eta p$ that is used in the definition,
and so~$\eta$ may be adjusted appropriately.)
In the case when $M=\Omega(n^2)$, that is, when the graph $G$ is dense, and
in addition~$G$ is $(H,\eta,p)$-regular, then the well-known embedding lemma
states that~$G$ must contain a canonical copy of~$H$. We would like to extend
this to the sparse case, when $p \ll 1$. 
In fact,
{\L}uczak~\cite{KR} showed that when $p=o(1)$, then there exist graphs~$G$
that are $(H,\eta,p)$-regular and are $H$-free.
The K{\L}R conjecture states that although such examples exist, there are
very few of them; few enough so that a typical random graph does not contain
any such example, and thus with high probability every $(H,\eta,p)$-regular
subgraph of a random graph contains a canonical copy of~$H$. Even this
claim will fail if $p$ is really small, and indeed examples similar to
those mentioned earlier show that we must require $p \gg n^{-1/m(H)}$. 

The counting lemma is a strengthening of the embedding lemma; it says that,
for constant~$p$, small $\eta$ and large~$n$, we have not just $i_H(G)>0$
but $i_H(G)=(1+o(1))n^{v(H)} p^{e(H)}$. One could hope to generalize the
counting lemma too to the sparse setting.  The hypergraph container methods
do not seem appropriate for establishing such a precise count, but
nonetheless they are enough to establish something weaker, namely that
there are very few $(H, \eta, p)$-regular graphs with $o(n^{v(H)}
p^{e(H)})$ copies of~$H$.

The next theorem verifies the K{\L}R conjecture, and further gives the
analogous result for the weak counting lemma.

\begin{thm}\label{thm:klr3}
  Let $H$ be a graph and let $\alpha>0$. There exists $c>0$, such
  that for~$n$ sufficiently large and $M \ge cn^{2-1/m(H)}$, if
  $G=G(n,M,H)$ is chosen at random, then
$$
  \mathbb{P}\left(\mbox{ $G$ is }(H\,,\,{1\over c}\,,\,{M\over
    n^2})\mbox{-regular \ and \ }
    i_H(G)\le {1\over c}\,n^{v(H)}\left(M\over n^2\right)^{e(H)} \right) \le
  \alpha^M.
$$
\end{thm}

This theorem with the stronger constraint $i_H(G)=0$ is what is often
referred to as the K{\L}R conjecture; it was proved for balanced $H$ by
Balogh, Morris and Samotij~\cite{BMS}. As
mentioned, the weak counting lemma for $(H, \eta, p)$-regular subgraphs of
a random graph follows from Theorem~\ref{thm:klr3} by the union bound over
all possible bad subgraphs.  Conlon, Gowers, Samotij and
Schacht~\cite{CGSS} proved this lemma directly (i.e., they showed that the
number of~$H$ in every $(H,\eta,p)$-regular subgraph has the correct order
of magnitude with high probability); moreover, for strictly balanced~$H$,
they obtained the precise counting lemma (i.e., the number
of~$H$ is $(1+o(1))n^{v(H)} p^{e(H)}$ with high probability).

We turn now to the proofs of the theorems. As mentioned, the derivation of
the theorems is straightforward once an appropriate container theorem is
available, and the arguments here are routine (similar to those
in~\cite{BMS}), but we include details for completeness.

We already indicated in~\S\ref{subsec:sparse} how to prove
Theorem~\ref{thm:ffree_sparse}. Each $H$-free graph $I$ contains no more
than $(\pi(H) +o(1)){N \choose \ell}$ edges, and with high probability
$G^{(\ell)}(N,p)$ contains not much more than $(\pi(H) +o(1))p{N \choose
  \ell}$ of these edges. This is not in itself enough to prove
Theorem~\ref{thm:ffree_sparse} because there are too many independent
sets. But the argument is valid with containers $C$ instead of independent
sets~$I$, and the theorem then does follow via the union bound, because
there are few containers.

Very similar arguments are used to prove Theorems~\ref{thm:klr_conj} and
Theorem~\ref{thm:klr3}. For Theorem~\ref{thm:klr_conj} we show that the
containers that matter are close to being $(\chi(H)-1)$-partite, so a
randomly chosen subgraph has the same property. For Theorem~\ref{thm:klr3}
we show that the containers must contain a very sparse subset, from which
$G=G(n,M,H)$ is very unlikely to have chosen many edges, and hence is
unlikely to be regular. The union bound then finishes the job.

For each application we need the following probabilistic lemma. (Strictly
speaking, Theorem~\ref{thm:klr3} uses a modification where the
hypergeometric distribution is used instead of the binomial.)  This lemma
is the place where the condition $T\subset I$ appearing in the container
theorem actually matters, and so is important for that reason. It is
phrased in a slightly cumbersome way because of the need to cover each of our
required applications, but the principle is simple. If a randomly chosen
subset $X\subset V(G)$ contains an unexpectedly large independent set, then
$X$ meets some container $C(T)$ in more vertices than expected. This event
is unlikely for two independent reasons: it requires both that $T\subset X$
and that $X\cap (C(T)-T)$ be large. Both of these contribute to making the
overall probability small.

\begin{lem}\label{lem:sparse}
  Given $0<\nu<1$ and $s\ge1$, there is a constant $\phi=\phi(\nu,s)$ such
  that the following holds.  Let $L$ be a set, $|L|=n$, and let
  $\mathcal{I} \subset \mathcal{P}(L)$.  Let $t\ge1$, let $\phi t/n \le p \le1$
  and let $\nu n/2\le d \le n$.  Suppose for each $I\in\mathcal{I}$ there
  exists both $T_I=(T_1,\ldots,T_{s'})\in\mathcal{P}(I)^{s'}$ and
  $D=D(T_I)\subset L$, where $s' \le s$, $\sum_i|T_i| \le t$ and $|D(T_I)|
  \le d$.  Let $X\subset L$ be a random subset where each element is chosen
  independently with probability $p$. Then
\begin{equation}\label{eqn:sparse}
\mathbb{P}\left( |D(T_I) \cap X|>(1+\nu)pd\mbox{ for some }
  I \subset X,\, I \in \mathcal{I} \right)
  \le \exp\{ -\nu^2 pd/32 \}.
\end{equation}
\end{lem}
\begin{proof}
Consider $I\in\mathcal{I}$ and $T=T_I=(T_1,\ldots,T_{s'})$.
Let $J(T) = T_1 \cup \cdots \cup T_{s'}$. Let $E_T$ be the event that
\[
 J(T) \subset X \quad\mbox{and}\quad |D(T)\cap X| \ge (1+\nu)pd.
\]
The event $E_T$ is contained in $F_T\cap G_T$, where
$F_T$ is the event that $J(T) \subset X$ and $G_T$ is the event that
$|(D(T) - J(T))\cap X| \ge (1+\nu)pd-|J(T)|$. Since $F_T$ and $G_T$
are independent, $\mathbb{P}(E_T)\le\mathbb{P}(F_T)\mathbb{P}(G_T)$.
Choose a set $D'\supset D(T)$ with $|D'|=d$. Note that $|J(T)|\le t\le
pn/\phi \le2pd/\phi\nu\le \nu pd/2$ if $\phi$ is large. Hence, using
standard estimates for the binomial random variable $\mbox{Bin}(n,p)$ 
(e.g.,~\cite[Corollary 2.3]{JLR}),
\begin{align*}
 \mathbb{P}(G_T) &= \mathbb{P}(|(D(T) - J(T))\cap X| \ge
 (1+\nu)pd-|J(T)|)\\
 &\le\mathbb{P}(|D'\cap X| \ge (1+\nu)pd-|J(T)|)\\
 &\le\mathbb{P}(|D'\cap X| \ge (1+\nu/2)pd)
 \,=\,\mathbb{P}(\mbox{Bin}(d,p) \ge (1+\nu/2)p d)\\
 &\le\exp\{ -\nu^2 pd/16 \}.
\end{align*}
Note that $\mathbb{P}(F_T)=p^{|J(T)|}$. Given some set $J\subset L$ with
$|J|=j$, there are at most $2^{sj}$ tuples $T$ such that $J(T)=J$,
because, for each $i\in J$, there are at most $2^s$ ways to specify which of
the subsets $T_1,\ldots,T_{s'}$ contain~$i$.  Let $x=pn/t\ge \phi$, so
$t\le 2pd/x\nu$.  If $\phi$ is large we may assume $p(n-t)>t$, so, summing
over the possible sizes of~$J$,
$$
 \sum_T \mathbb{P}(F_T)
   \le \sum_{j=0}^{t} {n \choose j} 2^{sj} p^j 
   \le (t+1) \left(\frac{ne2^sp}{t}\right)^{t}
   \le (xe^22^s)^t \le(xe^22^s)^\frac{2pd}{x\nu}
\le\exp\{ \nu^2 pd/32\}
$$
holds if $\phi$, and therefore $x$, is large.
If there exists $I\subset X$, $I\in\mathcal{I}$ with
$|D(T_I)\cap X|\ge(1+\nu)pd$, then the event $E_{T_I}$ holds.
Hence the probability in~(\ref{eqn:sparse}) is bounded by
\[
 \sum_{T} \mathbb{P}(F_T)\mathbb{P}(G_T) \le
 \exp\{ \nu^2 pd/32\}\exp\{ -\nu^2 pd/16 \}\le\exp\{ -\nu^2 pd/32 \}
\]
as claimed.
\end{proof}

\begin{proof}[Proof of Theorem~\ref{thm:ffree_sparse}]
Let $\mathcal{I}$ be the set of $H$-free $\ell$-graphs on vertex set~$[N]$.
Let $\epsilon=\gamma/4$ and $L=[N]^{(\ell)}$.
For $I\in\mathcal{I}$, let $T=T_I$, $C=C(T)$ and $c'=c(H,\epsilon)$ be given by
Theorem~\ref{thm:ffree_cover}.
Our aim is to apply Lemma~\ref{lem:sparse} with $D(T)=C(T)$ and
\[
 \nu = \gamma/2, \quad
 d = (\pi(H)+\epsilon){N \choose \ell}, \quad
 s = c', \quad
 t = c' N^{\ell-1/m(H)}.
\]
The conditions of Lemma~\ref{lem:sparse} then hold with $n={N \choose \ell}$,
noting that $d\ge \nu n/2$ and that $p\ge cN^{-1/m(H)}\ge \phi t/n$ if $c$
is large enough. Finally, note that in ($\ref{eqn:sparse}$), each $H$-free
$\ell$-graph $I\in\mathcal{I}$ is contained in $C(T_I)$ and $(1+\nu)pd
\le (\pi(H)+\gamma)p{N \choose \ell}$, so the probability in the statement of
the theorem is bounded by
\[
 \exp\{-\nu^2pd/32\}
  \le \exp\left\{-\gamma^3  p{N \choose \ell} / 512 \right\},
\]
completing the proof.
\end{proof}

\begin{proof}[Proof of Theorem~\ref{thm:klr_conj}]
Notice that $\pi(H)=1-1/(\chi(H)-1)$ and $\chi(H)\ge3$.
It is a standard exercise, either using the stability arguments of
Erd\H{o}s and Simonovits or using Szemer\'edi's regularity lemma, that there
exists $\epsilon>0$ such that if $C$ is a $2$-graph on vertex set $[N]$
for $N$ sufficiently large with
$e(C)\ge(1-\frac{1}{\chi(H)-1}-11\epsilon){N\choose 2}$ and such that $C$
contains at most $\epsilon{N \choose v(H)}$ copies of $H$, then there exists
a subgraph $F\subset C$ of size $e(F) \le (\gamma/2){N \choose 2}$ such that
$C-F$ is $(\chi(H)-1)$-partite. We may and shall assume that $\epsilon\le1/66$
and $65\epsilon^2\le\gamma^3$. 

(A word of explanation is included here for those not so familiar with such
arguments. In the case that $C$ is $H$-free, the assertion just made is
precisely the stability theorem, as we mentioned after the statement of
Theorem~\ref{thm:klr_conj}. The stability proof is readily adapted to the
present situation where $C$ contains few copies of $H$. An alternative
approach uses Szemer\'edi's regularity lemma and the original stability
theorem, though the constants involved are much larger. In this argument,
the graph is partitioned by the regularity lemma, and the reduced graph,
whose vertices represent the parts of the partition and whose edges
represent regular pairs of positive density, must itself have density at
least $1-\frac{1}{\chi(H)-1}-11\epsilon$. By a counting lemma very like
Lemma~\ref{lem:weak_counting} below, the reduced graph cannot contain
$K_{\chi(H)}$, and the stability theorem applied to the reduced graph then
shows the reduced graph is close to $(\chi(H)-1)$-partite. Hence the same
holds for~$C$ itself.)

The argument is now roughly as follows. An $H$-free subgraph of $G(N,p)$
must lie in some container~$C$. If the subgraph has size larger than
$p(1-\frac{1}{\chi(H)-1}-\epsilon){N\choose 2}$ then it is unlikely that
$C$ has size smaller than $(1-\frac{1}{\chi(H)-1}-11\epsilon){N\choose
  2}$. But if $C$ has size larger than this then it can be made
$(\chi(H)-1)$-partite by removing few edges, and the same will then likely
be true of the random subgraph.

Let $\mathcal{I}$ be the set of $H$-free graphs on vertex set $[N]$.
For $I\in\mathcal{I}$ let $T=T_I$, $C=C(T)$ and $c'=c(H,\epsilon)$ be given by
Theorem~\ref{thm:ffree_cover} with $\epsilon$ as above. Let
\begin{align*}
 \mathcal{I}_1 &= \left\{I \in \mathcal{I} :
    e(C(T_I))\ge\left(1-\frac{1}{\chi(H)-1}-11\epsilon\right){N\choose 2}
  \right\}, \\ 
 \mathcal{I}_2 &= \mathcal{I} - \mathcal{I}_1.
\end{align*}
For $I\in\mathcal{I}_1$ let $F=F(T_I)\subset C(T_I)$ be as above, so that
$C(T_I)-F(T_I)$ is $(\chi(H)-1)$-partite.

Let $X=G(N,p)$. Let $E_1$ be the event that there exists $I\subset X$,
$I\in\mathcal{I}_1$ such that $|F(T_I)\cap X|\ge\gamma p{N\choose2}$.
Let $E_2$ be the event that there exists $I\subset X$, $I\in\mathcal{I}_2$
such that $|C(T_I)\cap X| \ge (1-\frac{1}{\chi(H)-1}-\epsilon)p{N\choose2}$.
Observe that $E_0 \subset E_1 \cup E_2$.

The probability of $E_2$ is bounded by applying Lemma~\ref{lem:sparse}
to the collection $\mathcal{I}_2$, with 
$L=[N]^{(2)}$, $n={N\choose2}$,
$D(T_I)=C(T_I)$, $\nu=10\epsilon$,
$d=(1-\frac{1}{\chi(H)-1}-11\epsilon){N\choose2}\ge\frac{1}{3}{N\choose2}$,
$s=c'$ and $t=c'N^{2-1/m(H)}$; provided
$p \ge cN^{-1/m(H)}$ and $c,N$ are sufficiently large,
\[
 \mathbb{P}(E_2)
   \le \exp\left\{ -\nu^2 pd / 32\right\}
   \le \exp\left\{ -25\epsilon^2 p {N\choose2} / 24\right\}.
\]

The probability of $E_1$ is bounded by applying Lemma~\ref{lem:sparse}
to the collection $\mathcal{I}_1$, with $D(T_I)=F(T_I)$, $\nu=\gamma$,
$d=(\gamma/2){N\choose2}$, $s=c'$ and $t=c'N^{2-1/m(H)}$; provided
$p \ge cN^{-1/m(H)}$ and $c,N$ are sufficiently large,
\[
 \mathbb{P}(E_1)
   \le \exp\left\{ -\nu^2 pd / 32\right\}
   = \exp\left\{ - p\gamma^3 {N\choose2} / 64\right\}
   \le \exp\left\{ - 65p\epsilon^2 {N\choose2} / 64\right\}.
\]
Since $\mathbb{P}(E_0)\le\mathbb{P}(E_1)+\mathbb{P}(E_2)$ and
$pN^2$ is large, this completes the proof of Theorem~\ref{thm:klr_conj}.
\end{proof}

In order to prove Theorem~\ref{thm:klr3}, we use a slight variation of a
standard counting lemma. It says that containers with few canonical copies
of $H$ must contain a bipartite subgraph that has a substantial number of
vertices but is nevertheless very sparse.

\begin{lem}
 \label{lem:weak_counting}
 Let $H$ be a graph and let $f : (0,1) \to (0,1)$.  Then there exists $\eta>0$
 and $\eps>0$ such that the following is true.  Let $C$ be a graph of order
 $N=nh$, where $h=v(H)$, whose vertices are partitioned into sets
 $V_1,\ldots,V_h$ each of size~$n$. Suppose $i_H(G)< \eps N^h$.  Then
 there exists $\gamma\ge \eta$, $\{i,j\} \in E(H)$ and $A \subset V_i$, $B
 \subset V_j$, of size $|A|, |B| = \gamma n$, such that $e(C[A,B]) \le
 f(\gamma) n^2$.
\end{lem}
\begin{proof}
Let $\delta_0 =1$ and, for $i=1,2,\ldots,h$, define
\[
\eps_i=\frac{1}{2h}\prod_{k<i} \delta_k\qquad\mbox{and}\qquad
 \delta_i = f(\eps_i).
\]
Let $\eta=\eps_h$ and $\eps=\prod_{i=1}^h\eps_i$.
The following process generates canonical copies of~$H$ with vertices labelled
by $v_1, \ldots, v_h$.
Let $V_{i,1} = V_i$ for $i\in[h]$. For each $i=1, \ldots, h$, do the
following.
\begin{enumerate}
 \item For each $j > i$ such that $\{i,j\} \in E(H)$, let
 \[
  A_{i,j} = \{ v \in V_{i,i} : |V_{j,i} \cap N(v)| < \delta_i |V_{j,i}| \}.
 \]
 \item Select $v_i$ from $V_{i,i} \setminus \bigcup_{j > i : \{i,j\} \in E(H)} A_{i,j}$.
  (If no such vertex exists then stop.)
 \item For each $j>i$, let $V_{j,i+1} \subset V_j$ be a set of size
   $\delta_i |V_{j,i}|$, 
 chosen arbitrarily from $V_{j,i} \cap N(v_i)$ if $\{i,j\} \in E(H)$, and
 otherwise 
 chosen arbitrarily from $V_{j,i}$. 
\end{enumerate}
Note that
\[
 |V_{j,i}|
  = \delta_{i-1} |V_{j,i-1}|
  = \delta_{i-1} \delta_{i-2} |V_{j,i-2}|
  = \cdots
  = \Big( \prod_{k<i} \delta_k \Big) n
  = 2h\eps_in
  = 2\eps_i N.
\]
Now if $|A_{i,j}| \le |V_{i,i}|/2h$ for every $\{i, j\} \in E(H)$, then the
number of choices for $v_i$ is at least
\[
 \Big| V_{i,i} \setminus \bigcup_{j > i : \{i,j\} \in E(H)} A_{i,j} \Big|
   \ge |V_{i,i}|/2 \ge \eps_i N\,,
\]
giving at least $\eps N^h$ canonical copies of~$H$, a contradiction.
Thus there exists $\{i, j\} \in E(H)$ with $|A_{i,j}| \ge |V_{i,i}|/2h=\eps_in$.
Then putting $\gamma = \eps_i\ge\eta$ and taking
$A \subset A_{i,j}$ and $B \subset V_{j,i}$ of size $|A|, |B| = \gamma n$
gives a pair of sets with
$
 e(C[A,B])
  \le |A| \delta_i |V_{j,i}|
  \le \delta_i n^2
  = f(\gamma)n^2
$
as required.
\end{proof}

\begin{proof}[Proof of Theorem~\ref{thm:klr3}]
  The argument is broadly this. Each $G(n,M,H)$ with relatively few
  canonical copies of $H$ must lie in a container $C$ which itself has
  few copies. Lemma~\ref{lem:weak_counting} states that $C$ has a very
  sparse bipartite subgraph, from which $G(n,M,H)$ is very unlikely to pick
  up many edges. But if $G(n,M,H)$ fails to pick up such edges it will
  fail to be regular. It is a crucial feature of the argument that, in
  order that $\alpha$ can be as small as we like, the bipartite subgraph
  can be made as sparse as we like whilst not being too small. This
  accounts for the appearance of Lemma~\ref{lem:weak_counting}.

  Here then are the details. We shall take $c$ to be sufficiently large as
  necessary. Let $N=hn$, let $p=M/n^2$ and let $X=G(n,M,H)$. For ease of
  notation we shall often identify graphs with their edge sets.  Define
  $f:(0,1) \to (0,1)$ by $2f(\gamma) = (\alpha/4)^{2/\gamma^2}$.  Let
  $\eps, \eta>0$ be given by Lemma~\ref{lem:weak_counting}.  Let
  $\widetilde{G} \subset G(N,H)$ be the set of canonical copies of~$H$
  (that is, the $n^h = N^h / h^h$ copies whose vertices are aligned to the
  partition $V_1 \cup \cdots \cup V_h$ of $[N]$.)  Reduce $\eps$ if
  necessary so that $e(\widetilde{G})\ge\eps N^h$.  Choose $\tilde{c}\ge2$
  larger than $c(H,\eps)$ as given by Theorem~\ref{thm:ffree_cover_heavy},
  and large enough so that inequality~(\ref{eqn:tildec}) below holds.  Set
  $q = (\eta^2 / h^2\tilde{c}^3)p$. Certainly $q\le 1/\tilde{c}$, and
  moreover $q \ge N^{-1/m(H)}$ holds if $c$ is large enough because $M \ge
  cn^{2-1/m(H)}$.  Hence we may apply Theorem~\ref{thm:ffree_cover_heavy}
  with~$H$, $\eps$, $q$ and~$\widetilde{G}$. Let~$\C$ be given by the
  theorem.

Consider the tuples $T=(T_1,\ldots,T_s)$ described in
Theorem~\ref{thm:ffree_cover_heavy}. For each such~$T$, let
$J(T)=T_1\cup \cdots \cup T_s$, and define the following probabilistic events:
\begin{align*}
 E_T &: \mbox{$J(T) \subset X \subset C(T)$ and $X$ is $(H,\eta,p)$-regular},\\
 F_T &: \mbox{$J(T) \subset X$}, \\
 G_T &: \mbox{$X\subset C(T)$ and $X$ is $(H,\eta,p)$-regular}.
\end{align*}
Theorem~\ref{thm:ffree_cover_heavy} states that if $i_H(X)< q^{e(H)}
N^h$ then there
exists~$T=(T_1,\ldots,T_s)$ with $J(T)\subset X \subset C(T)$; thus if in
addition $X$ is $(H,\eta,p)$-regular then~$E_T$ holds.
We may assume that $1/c\le\eta$ and also that $n^hp^{e(H)}/c< q^{e(H)}
N^h$. Therefore, by the union bound, 
to complete the proof it is enough to show that
$\sum_T \mathbb{P}(E_T) \le \alpha^M$.
Note that $E_T = F_T \cap G_T$ and so
$\mathbb{P}(E_T)=\mathbb{P}(F_T)\mathbb{P}(G_T|F_T)$.

In order to bound $\mathbb{P}(G_T|F_T)$, let $T$ be fixed, let $J=J(T)$ and
let $C=C(T)\in\C$.
Theorem~\ref{thm:ffree_cover_heavy} guarantees that~$C$ contains
at most $\eps N^{v(H)}$ canonical copies of~$H$. Thus by
Lemma~\ref{lem:weak_counting}, there exists $\{i,j\} \in E(H)$,
$A \subset V_i$, $B \subset V_j$, with $|A|, |B| = \gamma n$ and
$e(C[A,B]) \le f(\gamma) n^2$, where $\gamma\ge\eta$.
If $G_T$ holds, then $X$ is $(H,\eta,p)$-regular and so
\begin{equation}\label{eq:klr3_gt}
 |(A \times B)\cap (X-J)| \ge (1-\eta)p|A||B| - |J| \ge \gamma^2 M / 2
\end{equation}
(here we assumed, as we may, that $\eta\le 1/4$, and noted that $|J| \le
\tilde{c} q N^2 =\eta^2 pn^2/\tilde{c}^2\le \gamma^2 M/4$).  However
$|(A\times B) \cap C| \le f(\gamma) n^2$ and if $G_T$ holds then $X\subset
C$, so the probability of~(\ref{eq:klr3_gt}) is small. Specifically, in
generating the random graph~$(X-J)\cap (V_i \times V_j)$ when conditioned
on $J\subset X$, we are selecting a set of $M-|J\cap (V_i \times V_j)|\le
M$ edges uniformly from at least $n^2-|J\cap (V_i \times V_j)| \ge n^2/2$
possible edges, and for~(\ref{eq:klr3_gt}) to hold, we must select at least
$\gamma^2 M/2$ edges from a set of at most $f(\gamma) n^2$ possibilities.
This probability is at most
\[
 \mathbb{P}(G_T|F_T)
   \le {M\choose \gamma^2 M/2} \left(\frac{f(\gamma)
       n^2}{n^2/2}\right)^{\gamma^2 M/2} 
   \le 2^M (2f(\gamma))^{\gamma^2 M/2} 
   \le (\alpha/2)^M\,,
\]
by the definition of $f$.

Thus $\sum_T\mathbb{P}(E_T)\le\sum_T\mathbb{P}(F_T)\mathbb{P}(G_T|F_T)
\le (\alpha/2)^M\sum_T\mathbb{P}(F_T)$, and to finish the proof it is
enough to show that $\sum_T\mathbb{P}(F_T)\le 2^M$. Now
$$
\sum_T\mathbb{P}(F_T)=\sum_T\mathbb{P}(J(T)\subset X)
=\sum_T\mathbb{E}\mathbf{1}_{J(T)\subset X}
=\,\mathbb{E}\,|\{T\,:\,J(T)\subset X\}|\,.
$$
Since $T=(T_1,\ldots,T_s)$ where $s\le \tilde{c}$, Lemma~\ref{lem:entropy}
tells us that $|\{T\,:\,J(T)\subset X\}| \le\exp\{\tilde{c}\theta
  |X|(1+\log(1/\theta))\}$ where $\theta |X|$ is the average size of
the~$T_i$. Now $|X|=Me(H) = pn^2e(H)$ and $\theta|X|\le \tilde{c}
qN^2$. Thus
\begin{align}
 \tilde{c} \theta |X| (1 + \log (1/\theta))
   &\le \tilde{c} (\tilde{c}qN^2) \Big(1+\log
   \frac{pn^2e(H)}{\tilde{c}qN^2}\Big) \nonumber \\ 
   &= M \big(\eta^2/\tilde{c}\big) \big(1+\log ( e(H) \tilde{c}^2 /
   \eta^2)\big)  \nonumber \\ 
   &< M \log 2\label{eqn:tildec}
 \end{align}
as required.
\end{proof}

\section{Optimality}\label{sec:optimality}

We finish with the proof of Theorem~\ref{thm:optimality}
from~\S\ref{subsec:optimality}. The ideas behind the proof, which is a kind
of converse to the proof of Theorem~\ref{thm:ffree_sparse}, have already
been sketched out but here are the details.

\begin{proof}[Proof of Theorem~\ref{thm:optimality}.]
  We may assume that $c<1/4$. We may assume that $tn$ is large, since we can
  choose $\gamma$ so that $\gamma tn<1/2$ for small values of $tn$, in
  which case the theorem is immediate.
  We prove~(i) in a way that can readily be adapted
  for~(ii). 
  Put $\epsilon=\min\{1,(rc/9)^{1/(r-1)}\}$. 
  Select a subset $X\subset [n]$ by
  choosing vertices independently with probability~$p$, where
  $p=\epsilon t = \epsilon d^{-1/(r-1)}$.
  As mentioned in~\S\ref{subsec:optimality}, $X$ is likely to be close to
  independent. To be precise, by standard
  estimates for the binomial distribution (such as~\cite[Corollary
  2.3]{JLR}) we have $\mathbb{P}(|X|\le (1-c/3)pn) \le 2 e^{-c^2pn/40} < 1/3$
  since $pn=\epsilon tn$ is large. 
  The expected value of $e(G[X])$ is $p^rnd/r$ so
  $\mathbb{P}(e(G[X])>3p^rnd/r)\le 1/3$. Hence, by removing a vertex from each
  edge of $G[X]$, we see that, with probability at least $1/3$, $X$
  contains an independent subset $I$ with $|I|\ge (1-c/3)pn - 3p^rnd/r \ge
  (1-2c/3)pn$, where the last inequality holds because $\epsilon^{r-1}\le
  rc/9$. In summary
\begin{equation}\label{eqn:Cint}
  \mathbb{P}(\mbox{there exists independent $I\subset X$ with }|I|\ge
  (1-2c/3)pn) \,\ge\,{1\over3}\,.
\end{equation}
There must be some $C\in\C$ with $I\subset C$, so $\mathbb{P}(\mbox{there
  exists $C\in\C$ with }|X\cap C|\ge (1-2c/3)pn)\ge1/3$.  
This can happen only if $|\C|$ is large, since for an individual container
$C\in\C$ the event $|X\cap C|\ge (1-2c/3)pn$ is unlikely, because $|C|\le
(1-c)n$. Indeed, choosing
 $C'$ containing $C$ with $|C'|=(1-c)n$, and again using standard
estimates (e.g.~\cite[Corollary 2.3]{JLR}), we have
$\mathbb{P}(|X\cap C|\ge
(1-2c/3)pn)\le \mathbb{P}(|X\cap C'|\ge (1-2c/3)pn)\le 2
e^{-c^2pn/40}$. Therefore $|\C|2e^{-c^2pn/40}\ge 1/3$ or $|\C|\ge
(1/6)e^{c^2pn/40}$, which proves~(i).

We can sharpen this calculation if $\C$ is internally generated and
$\C=\{C(T):T\in\mathcal{T}\}$. If $I\subset X$ is an independent set then
there is some $T\in\mathcal{T}$ with $T\subset I\subset C(T)$. So
inequality~(\ref{eqn:Cint}) implies
\begin{equation}\label{eqn:TCint}
\mathbb{P}\left( |C(T) \cap X|\ge (1-2c/3)pn\mbox{ for some }
  T\in\mathcal{T}, T\subset X\ \right)
   \,\ge\,{1\over3}\,.
\end{equation}
We next show that small sets $T$ cannot make much contribution
to~(\ref{eqn:TCint}).
Let $\mathcal{T}^*=\{T\in\mathcal{T}: |T|\le \gamma tn\}$. Apply
Lemma~\ref{lem:sparse} with $L=[n]$, $s=1$, $t$ in the lemma equal to
$\gamma tn$ here, $D(T)=C(T)$ and $d=(1-c)n$. Let $\nu=c/3$, so $(1+\nu)pd
< (1-2c/3)pn$. Choose $\gamma$ so that $\phi\gamma\le \epsilon$. If we take
$\mathcal{I}$ in the lemma to be those independent sets for which there
exists some $T\in\mathcal{T}^*$ with $T\subset I\subset C(T)$, then the
conditions of the lemma are satisfied, and so
$$
\mathbb{P}\left( |C(T) \cap X|\ge (1-2c/3)pn\mbox{ for some }
  T\in\mathcal{T^*}, T\subset X\ \right) \le \exp\{ -\nu^2 pd/32 \} \le
{1\over6}
$$
because $tn$ is large. 

It follows now from inequality~(\ref{eqn:TCint}) that
$$
\mathbb{P}\left( |C(T) \cap X|\ge (1-2c/3)pn\mbox{ for some }
  T\in\mathcal{T}\setminus\mathcal{T}^*, T\subset X\ \right)
   \,\ge\,{1\over6}\,.
$$
In particular $\mathbb{P}(T\subset X\mbox{ for some }
T\in\mathcal{T}\setminus\mathcal{T}^*)\ge1/6$. But $\mathbb{P}(T\subset
X)=p^{|T|}$ and so
$$
\mathbb{P}(T\subset X\mbox{ for some }T\in\mathcal{T}\setminus\mathcal{T}^*)
\,\le\,\sum_{T\notin\mathcal{T}^*}p^{|T|}\le
|\mathcal{T}\setminus\mathcal{T}^*|\,p^{\gamma tn}
$$
by definition of $\mathcal{T}^*$. Hence
$|\mathcal{T}\setminus\mathcal{T}^*|p^{\gamma tn}\ge 1/6$.
Therefore $\mathcal{T}\setminus\mathcal{T}^*\ne\emptyset$, which proves
the first part of~(ii), and moreover
$|\mathcal{T}|\ge (1/6)(1/p)^{\gamma tn}\ge e^{\gamma' nt\log (1/t)}$ for
some constant $\gamma'$, implying the second part of~(ii).

Finally, suppose that $G$ is vertex and edge transitive and that
$\delta(G,\tau)\ge1$. Recalling Definition~\ref{defn:delta}, the vertex
transitivity of $G$ means that there are numbers $D_j$, $2\le j\le r$, such
that $d^{(j)}(v)=D_j$ for all $v\in[n]$. Let $F_j=\{\sigma\subset[n]:
|\sigma|=j,\, d(\sigma)=D_j\}$ and let $G_j$ be the $j$-graph with edge set
$F_j$. By the edge transitivity of $G$, every edge of $G$ includes a member of
$F_j$. In particular, every subset $I\subset [n]$ that is independent in
$G_j$ is independent in $G$, and so $\C$ is a collection of containers for
$G_j$ too.

Let $d_j$ be the average degree of~$G_j$. Then
$D_jnd_j/j=D_je(G_j)=D_j|F_j|=\sum_{\sigma\in F_j}d(\sigma)\le{r\choose j}
e(G) = {r\choose j}nd/r$. Thus $d_j\le {r-1\choose j-1}(d/D_j)$.  Recalling
again Definition~\ref{defn:delta}, we have $\delta_j
\tau^{j-1}nd=nD_j$. Now $\delta(G,\tau)\ge1$, so for some~$j$ we have
$2^{{r\choose2}-1-{j-1\choose2}}\delta_j>1/r$. This certainly implies
$r2^{{r\choose2}}\delta_j>1$, and so
$r2^{{r\choose2}}D_j>\tau^{j-1}d$. Hence, for this value of~$j$, we have
$d_j\le {r-1\choose j-1}(d/D_j)\le{r-1\choose
  j-1}r2^{{r\choose2}}\tau^{1-j}$. In particular $d_j^{-1/(j-1)}\ge
\gamma'\tau$ for some positive number $\gamma'$ depending on~$r$.

We noted previously that $\C$ is a collection of containers for
$G_j$. Hence, for every~$j$, properties~(i) and~(ii) apply with
$t=d_j^{-1/(j-1)}$. But there is some $j$ with $d_j^{-1/(j-1)}\ge
\gamma'\tau$, and so (i) and~(ii) apply with $t=\tau$, once the value of
$\gamma$ has been suitably adjusted.
\end{proof}

\section{Postscript}\label{sec:postscript}

We are extremely grateful to the referees of this paper, who expended a
great deal of care and thought on their work, and made many valuable
suggestions. In particular, their conscientious reading showed up a subtle
error in the original version, relating to the online property. The error
arose due to the use of the condition $d_s(\sigma)\ge 2^s\tau
d_{s+1}(\sigma)$ to define entry of non-singleton sets into $\Gamma_s$. The
problem with this definition is that it is a condition relative to what
happens in
$P_{s+1}$, unlike the absolute condition $d_s(v)\ge \tau^{r-s}d(v)$ used for
the entry of vertices. This relativity breaks the online property, which
is why in the present paper the online property is claimed only for
simple graphs (for which non-singleton sets in $\Gamma_s$ are irrelevant).

As a consequence of this, we were prompted to revisit our earlier ideas for
entry conditions. It turns out to be possible to specify an absolute
condition for entry of $\sigma$ into $\Gamma_s$, which nevertheless implies
the inequality $d_s(\sigma)\ge 2^s\tau d_{s+1}(\sigma)$. This gives
rise to a slightly different algorithm, but one which still yields all the
theorems of the present paper.

However the modified algorithm has many advantages. It needs only one pass
through the vertex set, constructing the hypergraphs $P_s$ simultaneously,
rather than the $r-1$ consecutive passes of the present method. The
operation of the algorithm is thus more transparent. Moreover only one set
$T$ is produced in prune mode, rather than the tuple $(T_{r-1},\ldots,T_0)$
described here. The use of an absolute entry condition makes much clearer
how the co-degree function $\delta(G,\tau)$ arises, and how
slightly different functions could be used at the expense of somewhat
larger containers.  Finally, the single pass approach yields the online
property immediately for all hypergraphs, not just simple ones; in
particular Theorem~\ref{thm:uniform} holds for all $r$-graphs, giving more
general colouring results.  We hope to describe the modified algorithm
elsewhere~\cite{STa}.

\end{document}